\documentclass[11pt]{article}%
\usepackage{amsfonts}
\usepackage{amsmath}
\usepackage{amssymb}
\usepackage{graphicx}
\usepackage{bbold}
\usepackage[colorlinks]{hyperref}
\providecommand{\U}[1]{\protect\rule{.1in}{.1in}}
\hypersetup{colorlinks=true, linkcolor=blue, citecolor=green, filecolor=black, urlcolor=blue}
\newtheorem{theorem}{Theorem}

\newtheorem{corollary}[theorem]{Corollary}

\newtheorem{definition}[theorem]{Definition}

\newtheorem{lemma}[theorem]{Lemma}

\newtheorem{proposition}[theorem]{Proposition}
\newtheorem{remark}[theorem]{Remark}

\newenvironment{proof}[1][Proof]{\noindent\textbf{#1.} }{\ \rule{0.5em}{0.5em}}

\begin{document}

\title{Deterministic and Stochastic Differential Equations in Hilbert Spaces
Involving Multivalued Maximal Monotone Operators\footnote{This is an electronic reprint of the original article published by the Panamer. Math. J. 6 (1996), no. 3, 83--119, \href{http://www.ams.org/mathscinet-getitem?mr=1400370}{MR1400370}. This reprint differs from the original in pagination and typographic detail.}}
\author{Aurel R\u{a}\c{s}canu$\bigskip$\\{\scriptsize Faculty of Mathematics, \textquotedblleft Alexandru Ioan
Cuza\textquotedblright\ University, }\\{\scriptsize Carol 1 Blvd., no. 11, Ia\c{s}i, Romania}\\{\scriptsize e-mail: aurel.rascanu@uaic.ro}}
\date{}
\maketitle

\begin{abstract}
This work deals with a Skorokhod problem driven by a maximal operator:%
\[
\left\{
\begin{tabular}
[c]{l}%
$du\left(  t\right)  +Au\left(  t\right)  (dt)\ni f\left(  t\right)
dt+dM\left(  t\right)  ,$\ $0<t<T,$\medskip\\
$u\left(  0\right)  =u_{0}\,,$%
\end{tabular}
\ \ \ \right.
\]
that is a multivalued deterministic differential equation with a singular
inputs $dM\left(  t\right)  ,$ where $t\rightarrow$ $M\left(  t\right)  $ is a
continuous function. The existence and uniqueness result is used to study an
It\^{o}'s stochastic differential equation
\[
\left\{
\begin{tabular}
[c]{l}%
$du(t)+Au(t)\left(  dt\right)  \ni f(t,u(t))dt+B(t,u(t))dW(t),$\ $0<t<T,$%
\medskip\\
$u(0)=u_{0},$%
\end{tabular}
\ \right.
\]
in a real Hilbert space $H$, where $A$ is a multivalued ($\alpha$-)maximal
monotone operator on $H,$ and $f(t,u)$ and $B(t,u)$ are Lipschitz continuous
with respect to $u$. Some asymptotic properties in the stochastic case are
also found.

\end{abstract}

\textbf{AMS Classification subjects: }60H10, 60H15, 47N20, 47N30.\medskip

\bigskip

{\small \textbf{CONTENT}}\smallskip

{\small \quad1.\quad Introduction }

{\small \quad2.\quad Deterministic evolution equations }

{\small \quad\quad\quad\quad2.1. Preliminaries }

{\small \quad\quad\quad\quad2.2. Generalized deterministic solutions }

{\small \quad\quad\quad\quad2.3. An example }

{\small \quad3.\quad Stochastic evolution equations }

{\small \quad\quad\quad\quad3.1. Preliminaries }

{\small \quad\quad\quad\quad3.2. $\alpha$-Monotone SDE with additive noise }

{\small \quad\quad\quad\quad3.3. Monotone SDE with state depending diffusion }

{\small \quad4.\quad Large time behaviour }

{\small \quad\quad\quad\quad4.1.Exponentially stability }

{\small \quad\quad\quad\quad4.2. Invariant measure }

\section{Introduction}

Generally, the stochastic model for parabolic evolution systems with
unilateral constraints (obstacle problem, one phase Stefan problem, Signorini
problem) is an infinite dimensional stochastic differential equation of the
form%
\begin{equation}
\left\{
\begin{tabular}
[c]{l}%
$du(t)+Au(t)\left(  dt\right)  \ni f(t,u(t))dt+B(t,u(t))dW(t),$\ $0<t<T,$%
\medskip\\
$u(0)=u_{0},$%
\end{tabular}
\right.  \label{1.1}%
\end{equation}
where $A$ is a maximal monotone operator in a Hilbert spaces $H$ and $f(t,u)$
and $B(t,u)$ defined for $(t,u)\in\left[  0,T\right]  \times H$ are Lipschitz
continuous with respect to $u$, and $\left\{  W(t)\right\}  _{t\geq0}$ is a
Wiener process with respect to a stochastic basis $(\Omega,\mathcal{F}%
,\mathbb{P},\left\{  \mathcal{F}_{t}\right\}  _{t\geq0})$.

For finite dimensional case we mention the works of P.L. Lions \& Sznitman
\cite{LSSD}, Y. Saisho \cite{SaSD} and the generalized result of E. Cepa
\cite{CeED}. The main ideas is to consider a generalized Skorohod problem%
\begin{equation}
\left\{
\begin{tabular}
[c]{l}%
$du(t)+Au(t)\left(  dt\right)  \ni f(t)dt+dM(t),$\ $0<t<T,$\medskip\\
$u(0)=u_{0},$%
\end{tabular}
\right.  \label{1.2}%
\end{equation}
and by continuity of Skorohod mapping%
\[
(u_{0},f,M)\rightarrow u=S(u_{0},f,M)
\]
we obtain the existence and the uniqueness of the solution of equation
(\ref{1.1}). If in finite dimensional case one assume that $intD(A)\neq
\emptyset$ and this assumption is essentially for the proof, in infinite
dimensional case this assumption is too restrictive; it is not satisfied, not
even for the obstacle problem. For this reason the step from finite to
infinite dimensional case is not so directly.

V. Barbu and A. R\u{a}\c{s}canu studied in \cite{BaRaPV} parabolic variational
inequalities in the determinist case, that is equations of the form
(\ref{1.2}) where $A=A_{0}+\partial\varphi$. Some stochastic parabolic
variational inequalities of the form (\ref{1.1}) with $A=A_{0}+\partial
\varphi$ are considered by A. Bensoussan and A. R\u{a}\c{s}canu in \cite{BRSV}
and \cite{BeRaPV}. In this paper using the idea of looking the solution as
image by a Skorohod mapping we prove the existence and the uniqueness of the
solution of (\ref{1.1}).

The paper is organized as follows. Some preliminaries determinist results with
a generalization of the solution for singular inputs are given in Section 2.
Section 3 contains the main existence result on stochastic equation
(\ref{1.1}). Finally in Section 4 we give some asymptotic properties of the solution.

\section{Deterministic evolution equations}

\subsection{2.1 Preliminaries}

A. Throughout in this work $H$ is a real separable Hilbert space with the norm
$\left\vert \cdot\right\vert $ and the scalar product $\left(  \cdot
,\cdot\right)  $, and $\left(  X,\left\Vert \cdot\right\Vert _{X}\right)  $ is
a real separable Banach space with separable dual $\left(  X^{\ast},\left\Vert
\cdot\right\Vert _{X^{\ast}}\right)  $. It is assumed that%
\[
X\subset H\cong H^{\ast}\subset X^{\ast},
\]
where the embedding are continuous with dense range. The duality paring
$\left(  X^{\ast},X\right)  $ is denoted also $\left(  \cdot,\cdot\right)  $.
Let $\gamma_{0}>0$ a constant of boundedness: $\left\vert \cdot\right\vert
_{H}\leq\gamma_{0}\left\Vert \cdot\right\Vert _{X}\;.$

B. If $\left[  a,b\right]  $ is a real closed interval and $Y$ is a Banach
space then $\;L^{r}\left(  a,b;Y\right)  ,$ $C\left(  \left[  a,b\right]
;Y\right)  $, $\;BV\left(  \left[  a,b\right]  ;Y\right)  ,$ $AC\left(
\left[  a,b\right]  ;Y\right)  ,$ are the usual spaces of $p$-integrable,
continuos,with bounded variation, and absolutely continuous $Y$-valued
function on $\left[  a,b\right]  $, respectively.$\;$By $W^{1,p}([a,b];Y)$ we
shall denote the space of $y\in L^{p}(a,b;Y)$ such that $y^{\prime}\in
L^{p}(a,b;Y),$ where $y^{\prime}$ is the derivative in the sense of
distributions. Equivalently (see e.g. \cite{BaNS}, pag. 19, or \cite{BrOM}):
$W^{1,p}([a,b];Y)=\{y\in AC([a,b];Y):\frac{dy}{dt}\in L^{p}(a,b;Y);$%
\ $y\left(  t\right)  =y(a)+\int_{a}^{t}\frac{dy}{dt}(s)ds,$ $\forall
t\in\lbrack a,b]\}.$ The space $W^{2,p}([a,b];X)$ is similarly defined

C. A multivalued operator $A:H\rightarrow2^{H}$ will be seen also as a subset
of $H\times H$ setting for $A\subset H\times H$:
\[
Ax=\{y\in H:\,[x,y]\in A\}\quad\text{and}\quad D(A)=\left\{  x\in
H:\,Ax\neq\emptyset\right\}  .
\]
The operator A is a maximal monotone operator if $A$ is monotone i.e.%
\[
\left(  y_{1}-y_{2},x_{1}-x_{2}\right)  \geq0,\text{ for all }\,\left[
x_{1},y_{1}\right]  \in A,\,\left[  x_{2},y_{2}\right]  \in A
\]
and it is\ maximal in the set of monotone operators: that is,
\[
\left(  v-y,u-x\right)  \geq0\;\;\;\forall\,\left[  x,y\right]  \in
A,\;\;\;\;\Rightarrow\;\left[  u,v\right]  \in A.
\]
Let $\varepsilon>0$ The following operators%
\[
J_{\varepsilon}x=(I+\varepsilon A)^{-1}(x)\text{ and }A_{\varepsilon}=\frac
{1}{\varepsilon}(x-J_{\varepsilon}x),
\]
are single-valued and they satisfy (see \cite{BaNS} and \cite{BrOM} ) the
properties for all $\varepsilon,\delta>0,\;\;x,y\in H:$
\begin{equation}%
\begin{tabular}
[c]{l}%
$a)\quad\left[  J_{\varepsilon}x,A_{\varepsilon}x\right]  \in A,$\medskip\\
$b)\quad\left\vert J_{\varepsilon}x-J_{\varepsilon}y\right\vert \leq\left\vert
x-y\right\vert ,$\medskip\\
$c)\quad\left\vert A_{\varepsilon}x-A_{\varepsilon}y\right\vert \leq\dfrac
{1}{\varepsilon}\left\vert x-y\right\vert ,$\medskip\\
$d)\quad\left\vert J_{\varepsilon}x-J_{\delta}x\right\vert \leq\left\vert
\varepsilon-\delta\right\vert \left\vert A_{\delta}x\right\vert ,$\medskip\\
$e)\quad\left\vert J_{\varepsilon}x\right\vert \leq\left\vert x\right\vert
+(1+\left\vert \varepsilon-1\right\vert )\left\vert J_{1}0\right\vert
,$\medskip\\
$f)\quad A_{\varepsilon}:H\rightarrow H\quad$is a maximal monotone operator.
\end{tabular}
\label{2.1}%
\end{equation}
Also%
\begin{equation}%
\begin{tabular}
[c]{l}%
$a)\quad\overline{D\left(  A\right)  \text{ }}$ is a convex set and
$\lim\limits_{\varepsilon\searrow0}J_{\varepsilon}x=\Pr_{\overline{D\left(
A\right)  }}x,\;\forall x\in H,$\medskip\\
$b)\quad\forall\,\left[  x,y\right]  \in A:\;\;Ax$ is a closed convex set
,\medskip\\
\quad\quad\quad\quad$\lim\limits_{\varepsilon\searrow0}A_{\varepsilon}%
x=\Pr_{Ax}\left\{  0\right\}  $ $\overset{def}{=}\;A^{0}x$ and$\;\left\vert
A_{\varepsilon}x\right\vert \leq\left\vert y\right\vert \;$\quad\medskip\\
$c)\quad$if $\left[  x_{n},y_{n}\right]  \in A\;$and\medskip\\
\quad\quad\quad\quad$x_{n}\rightarrow x$ \ (strongly) in $H,\;y_{n}%
\;\overset{w}{\rightarrow}y\;$(weakly) in $H,\;$or\medskip\\
\quad\quad\quad\quad$x_{n}\overset{w}{\rightarrow}x,\quad and\quad\quad
y_{n}\rightarrow y,\;\;$or\quad\medskip\\
\quad\quad\quad\quad$x_{n}\overset{w}{\rightarrow}x,\quad\;y_{n}%
\overset{w}{\rightarrow}y$ $,\quad\varlimsup_{n}\left(  x_{n},y_{n}\right)
\leq\left(  x,y\right)  $\quad\medskip\\
\quad\quad then $\left[  x,y\right]  \in A,$\medskip\\
$d)\quad$if $\varepsilon_{n}\rightarrow0,\;x_{n}\rightarrow x$
$,\;A_{\varepsilon_{n}}y_{n}\overset{w}{\rightarrow}y$ \ then $\left[
x,y\right]  \in A$%
\end{tabular}
\ \ \ \label{2.2}%
\end{equation}

Let $\alpha\in R$ be given. The operator$\;A:H\rightarrow2^{H}$ is called
$\alpha-maximal\;monotone$ operator if $A+\alpha I$ is the maximal monotone
operator ($I$ is the identity operator on $H$).

For $A:H\rightarrow2^{H}\;$an $\alpha$-maximal monotone operator , $u_{0}%
\in\overline{D\left(  A\right)  }$, $f\in L^{1}\left(  0,T;H\right)  ,$ the
strong solution of the Cauchy problem
\begin{equation}
\frac{du}{dt}+Au\ni f_{0}\left(  t\right)  ,\quad a.e.\;t\in\left(
0,T\right)  ,\;\;\;\quad\!\!u\left(  0\right)  =u_{0} \label{2.3}%
\end{equation}
is defined$\,$as a function $u\in W^{1,1}\left(  \left[  0,T\right]
;H\right)  \;$satisfying $u\left(  0\right)  =u_{0,\;}u\left(  t\right)  \in
D\left(  A\right)  \;$a.e. $t\in\left(  0,T\right)  ,$ and there exists $h\in
L^{1}\left(  0,T\right)  $, such that $\,h\left(  t\right)  \in Au\left(
t\right)  $ a.e. $t\in\left(  0,T\right)  $ and \ $du/dt+h\left(  t\right)
=f\left(  t\right)  ,\;a.e.\;t\in\left(  0,T\right)  .$ Such a solution is
noted $u=S\left(  u_{0},f\right)  $ and we remark that the strong solution is
unique when this exists.

We recall from \cite{BaOC}, p.31, that \ the following proposition holds :

\begin{proposition}
\label{p2.1} If $A$\ is $\alpha$-maximal monotone operator (\thinspace
$\alpha\in R\,$) on \thinspace$H,$\ $u_{0}\in D\left(  A\right)  $\ and
$f_{0}\in W^{1,1}\left(  \left[  0,T\right]  ;H\right)  $\ then the Cauchy
problem (\ref{2.3}) has a unique strong solution $u\in W^{1,\infty}\left(
\left[  0,T\right]  ;H\right)  $.\ Moreover if $A_{\varepsilon}^{\alpha}\;$is
the Yosida approximation of the operator $A+\alpha I\;$and $u_{\varepsilon}%
$\ is the solution of the approximate equation
\[
\frac{du_{\varepsilon}}{dt}+A_{\varepsilon}^{\alpha}u_{\varepsilon}-\alpha
u_{\varepsilon}=f,\quad u_{\varepsilon}\left(  0\right)  =u_{0}%
\]
then for all $\left[  x_{0},y_{0}\right]  \in A$ there exists a constant
$C=C\left(  \alpha,T,x_{0},y_{0}\right)  >0$\ such that

\noindent c$_{1}$)$\quad\left\Vert u_{\varepsilon}\right\Vert _{C\left(
\left[  0,T\right]  ;H\right)  }^{2}\leq C\left(  1+\left\vert u_{0}%
\right\vert ^{2}+\left\Vert f\right\Vert _{L^{1}(0,T;H)}^{2}\right)  $
and$\smallskip$

\noindent c$_{2}$)$\quad\lim\limits_{\varepsilon\searrow0}u_{\varepsilon}%
=u$\ in $C\left(  \left[  0,T\right]  ;H\right)  .$
\end{proposition}

Here we shall study the equation (\ref{1.2}) under the following basic
assumptions:$\,\,$%
\[
(H_{1})\left\{
\begin{tabular}
[c]{l}%
$i)\quad A:H\rightarrow2^{H}$ is $\alpha$-maximal monotone operator \ $\left(
\alpha\in R\right)  ,$\medskip\\
$ii)\quad\exists\;h_{0}\in H,\;\exists\,\,r_{0},a_{1},a_{2}>0$ \quad such
that\medskip\\
\quad\quad\quad$r_{0}\left\Vert y\right\Vert _{X^{\ast}}\leq\left(
y,x-h_{0}\right)  +a_{1}\left\vert x\right\vert ^{2}+a_{2},\;\forall\left[
x,y\right]  \in A.$%
\end{tabular}
\ \ \right.  \,\,
\]
$\quad\quad\quad\quad\quad\quad\quad\quad\quad\quad\quad\quad\quad\quad
\quad\quad\quad\quad\quad\quad\quad\,\,\quad\quad\,\,\quad\quad\quad$

\begin{theorem}
\label{t2.2}One of the following assumptions implies (H$_{1}$):$\smallskip$

\noindent(H$_{1}$-I)$\;\;\left\{
\begin{tabular}
[c]{l}%
$a)\quad A=A_{0}+\partial\varphi,\ $where $A_{0}:H\rightarrow H$ is a
continuous$\smallskip$\\
\quad\quad operator such that:$\smallskip$\\
\quad\quad\ $\exists\alpha\in R:\;\left(  A_{0}x-A_{0}y,x-y\right)
+\alpha\left\vert x-y\right\vert ^{2}\geq0\smallskip$\\
\quad\quad(i.e. $\alpha I+A_{0}$\ is continuous monotone operator
on\ $H$)$\smallskip$\\
$\quad\quad$and $\varphi:H\rightarrow\left]  -\infty,+\infty\right]  $ is a
proper convex$\smallskip$\\
\quad\quad lower-semicontinuous function,$\smallskip$\\
$b)\quad\exists h_{0}\in H,\;R_{0}>0,\;a_{0}>0$ such that$\smallskip$\\
\quad\quad$\varphi\left(  h_{0}+x\right)  \leq a_{0},\;\forall\,x\in
X,\;\left\Vert x\right\Vert _{X}\leq R_{0},$%
\end{tabular}
\right.  \smallskip$

or$\smallskip$

\noindent(H$_{1}$-II)$\;\left\{
\begin{tabular}
[c]{l}%
$a)\quad\exists\,V$ a separable Banach space such that\ $V\subset H\subset
V^{\ast}\smallskip$\\
\quad\quad densely and continuously and\ $V\cap X$\ is densely in
$X,\smallskip$\\
$b)\quad A:H\rightarrow2^{H}$\ is\ $\alpha$-maximal monotone
operator$\smallskip$\\
\quad\quad with $D\left(  A\right)  \subset V\smallskip$\\
$c)\quad\exists\,a,\beta\in R,\,a>0,$\ such that$\smallskip$\\
\quad\quad$(y_{1}-y_{2},x_{1}-x_{2})+\beta\left\vert x_{1}-x_{2}\right\vert
^{2}\geq a\left\Vert x_{1}-x_{2}\right\Vert ^{2},\smallskip$\\
\quad\quad$\forall\,\,[x_{1},y_{1}],\,\,[x_{2},y_{2}]\in A,\smallskip$\\
$d)\quad\exists\,h_{0}\in V,\;\exists\,r_{0},a_{0}>0$\ such that $h_{0}%
+r_{0}e\in D\left(  A\right)  $\ and$\smallskip$\\
\quad\quad$\left\Vert A^{0}\left(  h_{0}+r_{0}e\right)  \right\Vert _{V^{\ast
}}\leq r_{0}$\ for all $e\in V\cap X,\left\Vert e\right\Vert _{X}=1,$%
\end{tabular}
\ \right.  \smallskip$

or$\smallskip$

\noindent(H$_{1}$-III)$\left\{
\begin{tabular}
[c]{l}%
$a$)$\quad A$\ is $\alpha$-maximal monotone with $intD\left(  A\right)
\neq\emptyset\smallskip$\\
$b$)$\quad X=H$%
\end{tabular}
\ \right.  $
\end{theorem}

\begin{proof}
(H$_{1}$-I)$\Rightarrow$(H$_{1}$).

We prove (H$_{1}$-ii). Let $r_{0}\in\left[  0,R_{0}\right]  $ such that
$\left\vert A_{0}\left(  h_{0}+re\right)  -A_{0}\left(  h_{0}\right)
\right\vert \leq1$ for all $r\in\left(  0,r_{0}\right]  $ and $e\in
X,\;\left\Vert e\right\Vert _{X}=1.$

Let$\;\;y\in A_{0}x+\partial\varphi\left(  x\right)  .$ Then $\left(
y-A_{0}x,h_{0}+r_{0}e-x\right)  +\varphi\left(  x\right)  \leq\varphi\left(
h_{0}+r_{0}e\right)  \leq a_{0},\,\,\,$ $\forall$ $e\in X,\;\left\Vert
e\right\Vert _{X}=1.$

Since $\varphi$ is a convex l.s.c. function, $\exists\,b_{1},b_{2}\in R$ such
that $\varphi\left(  x\right)  \geq b_{1}\left\vert x\right\vert +b_{2}.$
Hence%
\begin{align*}
r_{0}\left(  y,e\right)   &  \leq\left(  y,x-h_{0}\right)  +\left(
A_{0}x,h_{0}+r_{0}e-x\right)  -b_{1}\left\vert x\right\vert +a_{0}-b_{2}\\
&  \leq\left(  y,x-h_{0}\right)  -\alpha\left\vert h_{0}+r_{0}e-x\right\vert
^{2}+\\
&  \quad+\left(  A_{0}\left(  h_{0}+r_{0}e\right)  ,h_{0}+r_{0}e-x\right)
-b_{1}\left\vert x\right\vert +a_{0}-b_{2}%
\end{align*}
which clearly yields (H$_{%
%TCIMACRO{\QTR{textbf}{1}}%
%BeginExpansion
1%
%EndExpansion
}$-ii).

(H$_{1}$-II) $\Rightarrow$ (H $_{1}$)

Let $\left[  x,y\right]  \in A.$ From (H$_{1}$-II-c) we have%
\[
\left(  A^{0}\left(  h_{0}+r_{0}e\right)  -y,h_{0}+r_{0}e-x\right)
+\beta\left\vert h_{0}+r_{0}e-x\right\vert ^{2}\geq a\left\Vert h_{0}%
+r_{0}e-x\right\Vert _{V}^{2}\
\]
for all $e\in V\cap X\;\left\Vert e\right\Vert _{X}=1$ and then%

\begin{align*}
r_{0}\left(  y,e\right)  +a\left\Vert h_{0}+r_{0}e-x\right\Vert _{V}^{2}  &
\leq\left(  y,x-h_{0}\right)  +\beta\left\vert h_{0}+r_{0}e-x\right\vert
^{2}+(A^{0}\left(  h_{0}+r_{0}e\right)  ,h_{0}+r_{0}e-x)\\
&  \leq\left(  y,x-h_{0}\right)  +\beta\left\vert h_{0}+r_{0}e-x\right\vert
^{2}+\frac{1}{2a}\left\Vert A^{0}\left(  h_{0}+r_{0}e\right)  \right\Vert
_{V^{\ast}}^{2}\\
&  \quad+\frac{a}{2}\left\Vert h_{0}+r_{0}e-x\right\Vert _{V}^{2}\;
\end{align*}
which gets (H$_{1}$-ii).

(H$_{1}$-III)$\Rightarrow$ (H$_{1}$).

An $\alpha$-maximal monotone operator is bounded on $\,IntD\left(  A\right)
.$ Hence $h_{0}\in H,\;r_{0}>0$ exist $\;$such that $\;h_{0}+x\in D\left(
A\right)  ,\;\;\left\vert A^{0}\left(  h_{0}+x\right)  \right\vert \leq
r_{0}\;\;$for all $x\in H,\;\left\vert x\right\vert \leq r_{0}$. This operator
satisfies (H$_{1}$-II) with $V=H=X,\;\beta=2\left\vert \alpha\right\vert
+1,\;a=\left\vert \alpha\right\vert +1$. So (H $_{1}$) holds.\hfill
\end{proof}

\subsection{Generalized deterministic solutions}

Let the spaces $X\subset H\subset X^{\ast}$ and the equation (\ref{1.2}) with
the assumptions:%
\begin{equation}%
\begin{array}
[c]{rl}%
i)\quad & A:H\rightarrow2^{H}\ \text{is }\alpha\text{-maximal monotone
operator,}\medskip\\
ii)\quad & f\in L^{1}\left(  0,T;H\right)  ,\medskip\\
iii)\quad & M\in C\left(  \left[  0,T\right]  ;X\right)  ,\;\;M\left(
0\right)  =0,\medskip\\
iv)\quad & u_{0}\in H
\end{array}
\label{2.4}%
\end{equation}
($\alpha\in\mathbb{R}$ \thinspace given).

\begin{definition}
\label{d2.1}A pair of function $(u,\eta)$\ is a generalized (deterministic)
solution of equation (\ref{1.2}) ($(u,\eta)=GD(A;u_{0},f,M)$) if the following
conditions hold:$\smallskip$

\noindent$d_{1})\quad u\in C([0,T];H),\;u(t)\in\overline{D(A)}\;\;\forall
t\in\lbrack0,T],\;u(0)=u_{0},$\medskip

\noindent$d_{2})\quad\eta\in C([0,T];H)\cap BV([0,T];X^{\ast}),\;\eta
(0)=0,$\medskip

\noindent$d_{3})\quad u(t)+\eta(t)=u_{0}+%
%TCIMACRO{\dint \nolimits_{0}^{t}}%
%BeginExpansion
{\displaystyle\int\nolimits_{0}^{t}}
%EndExpansion
f(s)ds+M(t),\;\forall t\in\lbrack0,T],$\medskip

\noindent$d_{4})\quad$there are the sequences $\left\{  u_{0n}\right\}
\subset D(A)\;\;\left\{  f_{n}\right\}  \subset W^{1,1}([0,T];H),$\medskip

$\quad\quad\,\,M_{n}\in C([0,T];X)\cap W^{2,1}([0,T];H)$\ such that
\begin{equation}%
\begin{array}
[c]{rl}%
i)\quad & u_{0n}\rightarrow u_{0}\text{ in }H,\;f_{n}\rightarrow f\text{ in
}L^{1}\left(  0,T;H\right)  ,\medskip\\
\quad & \quad\quad\quad\quad\quad\quad\quad\quad\quad\,M_{n}\rightarrow
M\text{ in }C\left(  \left[  0,T\right]  ;X\right)  ,\medskip\\
ii)\quad & u_{n}\rightarrow u\text{ in }C\left(  \left[  0,T\right]
;H\right)  ,\;\eta_{n}\rightarrow\eta i\text{ in }C\left(  \left[  0,T\right]
;H\right)  ,\medskip\\
iii)\quad & \left\Vert \eta_{n}\right\Vert _{BV\left(  \left[  0,T\right]
;X^{\ast}\right)  }\leq C,
\end{array}
\ \label{2.5}%
\end{equation}
where $C$ is a constant depending only on $\left(  A,u_{0},f,M,T\right)  $ and
$u_{n}\in W^{1,\infty}\left(  \left[  0,T\right]  ;H\right)  $\ is the
(strong) solution of the approximating problem%
\begin{equation}
\left\{
\begin{tabular}
[c]{l}%
$u_{n}\left(  t\right)  +%
%TCIMACRO{\dint _{0}^{t}}%
%BeginExpansion
{\displaystyle\int_{0}^{t}}
%EndExpansion
h_{n}\left(  s\right)  ds=u_{0n}+%
%TCIMACRO{\dint \nolimits_{0}^{t}}%
%BeginExpansion
{\displaystyle\int\nolimits_{0}^{t}}
%EndExpansion
f_{n}\left(  s\right)  ds+M_{n}\left(  t\right)  ,$\ $t\in\left[  0,T\right]
,$\medskip\\
$h_{n}\left(  t\right)  \in Au_{n}(t),$\ a.e. $t\in\left(  0,T\right)
,$\medskip\\
$\eta_{n}\left(  t\right)  =%
%TCIMACRO{\dint _{0}^{t}}%
%BeginExpansion
{\displaystyle\int_{0}^{t}}
%EndExpansion
h_{n}\left(  s\right)  ds,$\ $t\in\left[  0,T\right]  .$%
\end{tabular}
\ \right.  \label{2.5a}%
\end{equation}

\end{definition}

We remark that:

$\quad\eta_{n}\in W^{2,1}\left(  \left[  0,T\right]  ;H\right)  \subset
C\left(  \left[  0,T\right]  ;H\right)  \cap BV\left(  \left[  0,T\right]
;X^{\ast}\right)  $ and$\smallskip$

$\quad(u_{n},\eta_{n})=GD\left(  A;u_{0n},f_{n},M_{n}\right)  $

\begin{theorem}
\label{t2.3}\ Assume that%
\begin{equation}%
\begin{array}
[c]{rl}%
i)\quad & \text{the operator}\ \;A\,\;\text{satisfies }(H_{1})\text{,}%
\medskip\\
ii)\quad & u_{0}\in\overline{D\left(  A\right)  },\medskip\\
iii)\quad & f\in L^{1}\left(  0,T;H\right)  ,\medskip\\
iv)\quad & M\in C\left(  \left[  0,T\right]  ;X\right)  ,\;\;M\left(
0\right)  =0.
\end{array}
\ \label{2.6}%
\end{equation}
Then the equation (\ref{1.2}) has a unique generalized (deterministic) solution.

Moreover:

$c_{1})\quad$if $(u,\eta)=GD\left(  A;u_{0},f,M\right)  $\ and $(\overline
{u},\overline{\eta})=GD\left(  A;\overline{u}_{0},\overline{f},\overline
{M}\right)  $\ are two solutions, then:
\begin{equation}%
\begin{tabular}
[c]{l}%
$\left\Vert u-\overline{u}\right\Vert _{C\left(  \left[  0,T\right]
;H\right)  }^{2}\leq C[\left\vert u_{0}-\overline{u_{0}}\right\vert
^{2}+\left\Vert f-\overline{f}\right\Vert _{L^{1}\left(  0,T;H\right)  }^{2}%
+$\medskip\\
$\quad+\left\Vert M-\overline{M}\right\Vert _{C\left(  \left[  0,T\right]
;H\right)  }^{2}+\left\Vert M-\overline{M}\right\Vert _{C\left(  \left[
0,T\right]  ;X\right)  }\left\Vert \eta-\overline{\eta}\right\Vert _{BV\left(
\left[  0,T\right]  ;X^{\ast}\right)  }]$%
\end{tabular}
\ \label{2.7}%
\end{equation}
with $C=C\left(  \alpha,T\right)  $\ a positive constant, and

$c_{2})\quad$\ for every equiuniform continuous subset $K$\ of $C\left(
\left[  0,T\right]  ;X\right)  $, $M\in K$\ , there exists $C_{0}=C_{0}\left(
r_{0},h_{0},a_{1},a_{2},T,N_{K}\right)  $(\footnote{\textit{The constants
$r_{0},h_{0},a_{1},a_{2}$ are defined in (H$_{1}$); N$_{K}\in\mathbb{N}^{*}$
is a constant of equiuniformly continuity: sup$\{\left\Vert g\left(  t\right)
-g\left(  s\right)  \right\Vert _{X}:\left\vert t-s\right\vert \leq
T/N_{K}\}\leq r_{0}/4,\;\forall\,g\in K.$}})\ a positive constant $C_{0}$ such
that:
\begin{equation}
\left\Vert u\right\Vert _{C\left(  \left[  0,T\right]  ;H\right)  }%
^{2}+\left\Vert \eta\right\Vert _{BV\left(  \left[  0,T\right]  ;X^{\ast
}\right)  }\leq C_{0}[1+\left\vert u_{0}\right\vert ^{2}+\left\vert
f\right\vert _{L^{1}\left(  0,T;H\right)  }^{2}+\left\Vert M\right\Vert
_{C\left(  \left[  0,T\right]  ;H\right)  }^{2}] \label{2.8}%
\end{equation}

\end{theorem}

\begin{proof}
Uniqueness.\ The uniqueness follows from (\ref{2.7}), and to prove (\ref{2.7})
let $\left(  u_{n},\eta_{n}\right)  =GD\left(  A;u_{0n},f_{n},M_{n}\right)
\;$ and $\;\left(  \overline{u}_{n},\overline{\eta}_{n}\right)  =GD\left(
A;\overline{u}_{0n},\overline{f}_{n},\overline{M}_{n}\right)  $ where
$\left\{  u_{0n},f_{n},M_{n}\right\}  ,\left\{  \overline{u}_{0n}%
,\overline{f_{n}},\overline{M}_{n}\right\}  $ are chosen as in Definition
\ref{d2.1}. Then by an easy calculation involving equation (\ref{2.5a}) we
obtain:%
\[%
\begin{array}
[c]{l}%
\left\vert u_{n}\left(  t\right)  -M_{n}\left(  t\right)  -\overline{u}%
_{n}\left(  t\right)  +\overline{M}_{n}\left(  t\right)  \right\vert ^{2}%
\leq|u_{n}\left(  s\right)  -M_{n}\left(  s\right)  -\overline{u}_{n}\left(
s\right)  \medskip\\
\displaystyle\quad+\overline{M}_{n}\left(  s\right)  |^{2}+2\int%
\nolimits_{s}^{t}\left(  M_{n}-\overline{M}_{n},d\eta_{n}-d\overline{\eta}%
_{n}\right)  +2\alpha\int\nolimits_{s}^{t}\left\vert u_{n}-\overline{u}%
_{n}\right\vert ^{2}d\tau\medskip\\
\displaystyle\quad+2\int\nolimits_{s}^{t}\left(  f_{n}-\overline{f}_{n}%
,u_{n}-M_{n}-\overline{u}_{n}+\overline{M}_{n}\right)  d\tau,
\end{array}
\]
Passing to limit on a subsequences $n_{k}\rightarrow\infty,$ we get:%
\begin{equation}%
\begin{tabular}
[c]{l}%
$\left\vert u\left(  t\right)  -M\left(  t\right)  -\overline{u}\left(
t\right)  +\overline{M}\left(  t\right)  \right\vert ^{2}\leq\left\vert
u\left(  s\right)  -M\left(  s\right)  -\overline{u}\left(  s\right)
+\overline{M}\left(  s\right)  \right\vert ^{2}$\medskip\\
$\displaystyle\quad+2\int\nolimits_{s}^{t}\left(  M\left(  \tau\right)
-\overline{M}\left(  \tau\right)  ,d\eta\left(  \tau\right)  -d\overline{\eta
}\left(  \tau\right)  \right)  +2\alpha\int\nolimits_{s}^{t}\left\vert
u-\overline{u}\right\vert ^{2}d\tau$\medskip\\
$\displaystyle\quad+2\int\nolimits_{s}^{t}\left(  f-\overline{f}%
,u-M-\overline{u}+\overline{M}\right)  d\tau$%
\end{tabular}
\ \ \label{2.9}%
\end{equation}
for all $0\leq s\leq t\leq T,$ which implies clearly (\ref{2.7}).

Existence. If $Y$ is a Banach space and $\;g:\left[  0,T\right]  \rightarrow
Y$ \ is a continuous function we set%
\[
m_{Y}\left(  \delta;g\right)  =\sup\left\{  \left\Vert g\left(  t\right)
-g\left(  s\right)  \right\Vert _{Y}\;\;t,s\in\left[  0,T\right]
,\;\left\vert t-s\right\vert \leq\delta\right\}
\]
(modulus of continuity).

For $\;u_{0},f,M$ given as in (\ref{2.6}) these exist the sequences%
\[%
\begin{array}
[c]{l}%
\left\{  u_{0n}\right\}  \subset D\left(  A\right)  ,\;\left\{  f_{n}\right\}
\subset W^{1,1}\left(  \left[  0,T\right]  ;H\right)  ,\medskip\\
\left\{  M_{n}\right\}  \subset C\left(  \left[  0,T\right]  ;X\right)  \cap
W^{2,1}\left(  \left[  0,T\right]  ;H\right)  ,\;\;M_{n}\left(  0\right)  =0
\end{array}
\]
such that $m_{X}\left(  \delta,M_{n}\right)  \leq m_{X}\left(  \delta
,M\right)  ,\;\forall n\in N^{\ast},\;\forall\delta>0,$

$u_{0n}\rightarrow u_{0}\;\;$in \ \ H,$\smallskip$

$f_{n}\rightarrow f$ \ \ in \ \ $L^{1}\left(  0,T;H\right)  ,\smallskip$

$M_{n}\rightarrow M$ \ in \ $C\left(  \left[  0,T\right]  ;X\right)
.\smallskip$

The conditions on $M_{n}$ are satisfied setting, for example,
\[
M_{n}\left(  t\right)  =n\int\nolimits_{R}\rho\left(  n\left(  t-s\right)
-1\right)  \widetilde{M}\left(  s\right)  ds=\int\nolimits_{R}\rho\left(
r\right)  \widetilde{M}\left(  t-\frac{1+r}{n}\right)  dr,
\]
where$\;\rho\in C_{0}^{\infty}\left(  R\right)  ,\;\;\rho\left(  -r\right)
=\rho\left(  r\right)  \geq0\;\;\forall\,r\in R,\;\rho\left(  r\right)
=0\;\forall\,\left\vert r\right\vert \geq1,$ $\int\nolimits_{R}\rho\left(
r\right)  dr=1$ and%
\[
\widetilde{M}\left(  t\right)  =\left\{
\begin{tabular}
[c]{l}%
$M\left(  0\right)  ,\;\;t<0$\medskip\\
$M\left(  t\right)  ,\;\;t\in\left[  0,T\right]  $\medskip\\
$M\left(  T\right)  ,\;t>T.$%
\end{tabular}
\ \ \right.
\]
Let $K$ be an equiuniformly continuous subset of $C\left(  \left[  0,T\right]
;X\right)  $ which contains $M,$ and let $\delta=T/N_{0}$ sufficiently small
such that%
\[
m_{X}\left(  \frac{T}{N_{0}},M\right)  \leq\frac{r_{0}}{4}%
\]
The approximating problem:
\begin{equation}%
\begin{tabular}
[c]{l}%
$u_{n}^{^{\prime}}\left(  t\right)  +Au_{n}\left(  t\right)  \ni f_{n}\left(
t\right)  +M_{n}^{^{\prime}}\left(  t\right)  $\medskip\\
$u_{n}\left(  0\right)  =u_{0n}$%
\end{tabular}
\ \ \label{2.10}%
\end{equation}
has a unique strong solution $\;u_{n}\in W^{1,\infty}\left(  \left[
0,T\right]  ;H\right)  $ \ and the sequence $\eta_{n}$ defined by $\eta
_{n}^{\prime}\left(  t\right)  =f_{n}\left(  t\right)  +M_{n}^{\prime}\left(
t\right)  -u_{n}^{\prime}\left(  t\right)  \in Au_{n}\left(  t\right)
$,$\;\;\eta_{n}\left(  0\right)  =0,$ is in $\;W^{2,1}\left(  \left[
0,T\right]  ;H\right)  $. We multiply equation (\ref{2.10}) by $\;u_{n}%
-M_{n}-h_{0}$ and integrate on $\left[  0,t\right]  ;$ the equality%
\begin{equation}%
\begin{array}
[c]{c}%
\displaystyle\left\vert u_{n}\left(  t\right)  -M_{n}\left(  t\right)
-h_{0}\right\vert ^{2}+2\int\nolimits_{0}^{t}\left(  \eta_{n}^{^{\prime}%
}\left(  s\right)  ,u_{n}\left(  s\right)  -h_{0}\right)  ds=\left\vert
u_{0n}-h_{0}\right\vert ^{2}\medskip\\
\displaystyle\quad+2\int\nolimits_{0}^{t}\left(  f_{n}\left(  s\right)
,u_{n}\left(  s\right)  -M_{n}\left(  s\right)  -h_{0}\right)  ds+2\int%
\nolimits_{0}^{t}\left(  \eta_{n}^{^{\prime}}\left(  s\right)  ,M_{n}\left(
s\right)  \right)  ds
\end{array}
\label{2.11}%
\end{equation}
follows. Let$\;0=r_{0}<r_{1}<...<r_{m}=T$, $r_{i+1}-r_{i}=\frac{T}{N_{0}}$,
$i=\overline{0,m-1}$ \ and \ $t\in\lbrack r_{k},r_{k+1}].$ Denote $t_{i}%
=r_{i}$ if $i\in\overline{0,k},$ $t_{k+1}=t$ . Then%
\[%
\begin{array}
[c]{l}%
\displaystyle\int\nolimits_{0}^{t}\left(  \eta_{n}^{^{\prime}}\left(
s\right)  ,M_{n}\left(  s\right)  \right)  ds=\medskip\\
\displaystyle=\sum\limits_{i=0}^{k}\left[  \int\nolimits_{t_{i}}^{t_{i+1}%
}\left(  M_{n}\left(  s\right)  -M_{n}\left(  t_{i}\right)  ,d\eta_{n}\left(
s\right)  \right)  +\left(  M_{n}\left(  t_{i}\right)  ,\eta_{n}\left(
t_{i+1}\right)  -\eta_{n}\left(  t_{i}\right)  \right)  \right]  \medskip\\
\leq m_{X}\left(  \frac{T}{N_{0}},M_{n}\right)  \left\Vert \eta_{n}\right\Vert
_{BV\left(  \left[  0,T\right]  ;X^{\ast}\right)  }+2N_{0}\left\Vert
M_{n}\right\Vert _{C\left(  \left[  0,T\right]  ;H\right)  }\left\Vert
\eta_{n}\right\Vert _{C\left(  \left[  0,T\right]  ;H\right)  }\medskip\\
\displaystyle\leq\frac{r_{0}}{4}\left\Vert \eta_{n}\right\Vert _{BV\left(
\left[  0,T\right]  ;X^{\ast}\right)  }+2N_{0}\left\Vert M_{n}\right\Vert
_{C\left(  \left[  0,T\right]  ;H\right)  }[\int\nolimits_{0}^{T}\left\vert
f_{n}\left(  s\right)  \right\vert ds+\medskip\\
\quad+\sup\limits_{s\in\left[  0,t\right]  }\left\vert u_{n}\left(  s\right)
-M_{n}\left(  s\right)  -h_{0}\right\vert +\left\vert h_{0}\right\vert ]
\end{array}
\]
Hence by (H$_{1}$)%
\[%
\begin{array}
[c]{l}%
\displaystyle\sup\limits_{s\in\left[  0,t\right]  }\left\vert u_{n}\left(
s\right)  -M_{n}\left(  s\right)  -h_{0}\right\vert ^{2}+2r_{0}\int%
\nolimits_{0}^{t}\left\Vert \eta_{n}^{\prime}\left(  s\right)  \right\Vert
_{X^{\ast}}ds\leq2\left\vert u_{0n}-h_{0}\right\vert ^{2}\medskip\\
\displaystyle\quad+18\left(  \int\nolimits_{0}^{T}\left\vert f_{n}\left(
s\right)  \right\vert ds\right)  ^{2}\ +4\left\vert a_{2}\right\vert
T+\frac{1}{2}\sup\limits_{s\in\left[  0,t\right]  }\left\vert u_{n}\left(
s\right)  -M_{n}\left(  s\right)  -h_{0}\right\vert ^{2}\medskip\\
\displaystyle\quad+4\left\vert a_{1}\right\vert \int\nolimits_{0}^{t}%
\sup\limits_{\tau\in\left[  0,s\right]  }\left\vert u_{n}\left(  \tau\right)
\right\vert ^{2}d\tau+r_{0}\left\Vert \eta_{n}\right\Vert _{BV\left(  \left[
0,T\right]  ;X^{\ast}\right)  }\ +18N_{0}^{2}\left\Vert M_{n}\right\Vert
_{C\left(  \left[  0,T\right]  ;H\right)  }^{2}%
\end{array}
\]
We obtain%
\begin{equation}%
\begin{tabular}
[c]{l}%
$\left\Vert u_{n}\right\Vert _{C\left(  \left[  0,T\right]  ;H\right)  }%
^{2}+\left\Vert \eta_{n}\right\Vert _{BV\left(  \left[  0,T\right]  ;X^{\ast
}\right)  }\medskip$\\
$\leq C\left[  1+\left\vert u_{0n}\right\vert ^{2}+\left\Vert f_{n}\right\Vert
_{L^{1}\left(  0,T;H\right)  }^{2}+\left\Vert M_{n}\right\Vert _{C\left(
\left[  0,T\right]  ;H\right)  }^{2}\right]  \leq C_{1},$%
\end{tabular}
\ \ \label{2.12}%
\end{equation}
where $\;C=C\left(  r_{0,}h_{0},N_{0},T,a_{1},a_{2}\right)  $ \ and
\ $C_{1}=C_{1}\left(  A,T,u_{0},f,M\right)  $ are positive constants.

Since \ $\left(  u_{n},\eta_{n}\right)  =GD\left(  A;u_{0n},f_{n}%
,M_{n}\right)  ,$ $n\in\mathbb{N},\;$then by (\ref{2.7}) which we already
proved we have%
\[%
\begin{array}
[c]{l}%
\left\Vert u_{n}-u_{m}\right\Vert _{C\left(  \left[  0,T\right]  ;H\right)
}^{2}\leq C[\left\vert u_{0n}-u_{0m}\right\vert ^{2}+\left\Vert f_{n}%
-f_{m}\right\Vert _{L^{1}\left(  0,T;H\right)  }^{2}\medskip\\
\quad+\left\Vert M_{n}-M_{m}\right\Vert _{C\left(  \left[  0,T\right]
;H\right)  }^{2}+\left\Vert M_{n}-M_{m}\right\Vert _{C\left(  \left[
0,T\right]  ;X\right)  }\left\Vert \eta_{n}-\eta_{m}\right\Vert _{BV\left(
\left[  0,T\right]  :X^{\ast}\right)  }]\medskip\\
\leq C_{2}[\left\vert u_{0n}-u_{0m}\right\vert ^{2}+\left\Vert f_{n}%
-f_{m}\right\Vert _{L^{1}\left(  0,T;H\right)  }^{2}+\left\Vert M_{n}%
-M_{m}\right\Vert _{C(\left[  0,T\right]  ;H)}^{2}+\medskip\\
\quad+\left\Vert M_{n}-M_{m}\right\Vert _{C\left(  \left[  0,T\right]
;X\right)  }]
\end{array}
\]
where \ $C_{2}=C_{2}\left(  A,T,u_{0},f,M\right)  >0.\;$Hence there exists
$u\in C\left(  \left[  0,T\right]  ;H\right)  $ such that%
\[%
\begin{tabular}
[c]{l}%
$u_{n}\rightarrow u\;$\ in \ $C\left(  \left[  0,T\right]  ;H\right)
\medskip$\\
\multicolumn{1}{r}{$\displaystyle\eta_{n}=u_{0n}+\int\nolimits_{0}^{\cdot
}f_{n}\left(  s\right)  ds+M_{n}-u_{n}\rightarrow u_{0}+\int\nolimits_{0}%
^{\cdot}fds+M-u\medskip$}\\
\multicolumn{1}{r}{in $C\left(  \left[  0,T\right]  ;H\right)  \medskip$}\\
$\eta_{n}\rightarrow\eta$ \ weak star in \ $BV\left(  \left[  0,T\right]
;X^{\ast}\right)  $%
\end{tabular}
\ \
\]
an by (\ref{2.12}) the inequality (\ref{2.8}) follows. The proof is
complete.\hfill
\end{proof}

From (\ref{2.11}) and (H$_{1}$) as $n\rightarrow\infty,$ we have:

\begin{remark}
\label{r2.1}If $(u,\eta)=GD\left(  A;u_{0},f,M\right)  $, then:%
\[%
\begin{array}
[c]{l}%
\left\vert u\left(  t\right)  -M\left(  t\right)  -h_{0}\right\vert
^{2}+2r_{0}\left\Vert \eta\right\Vert _{BV\left(  \left[  0,t\right]
;X^{\ast}\right)  }\leq\left\vert u_{0}-h_{0}\right\vert ^{2}+2\left\vert
a_{2}\right\vert t\medskip\\
\displaystyle\quad+2\left\vert a_{1}\right\vert \int\limits_{0}^{t}\left\vert
u(s)\right\vert ^{2}ds+2\int\nolimits_{0}^{t}\left(  f\left(  s\right)
,u\left(  s\right)  -M\left(  s\right)  -h_{0}\right)  ds+2\int\nolimits_{0}%
^{t}\left(  M\left(  s\right)  ,d\eta\left(  s\right)  \right)  ,
\end{array}
\]
for all $t\in\left[  0,T\right]  .$
\end{remark}

In the next section this inequality will be used for some estimates of the
stochastic generalized solutions.

\begin{corollary}
\label{c2.4}Let the assumptions of Theorem \ref{t2.3} be satisfied
$A_{\varepsilon}^{\alpha}$ the Yosida approximation of the operator $A+\alpha
I$ and $u_{\varepsilon},\;$with $0<\varepsilon<\frac{1}{\left\vert
\alpha\right\vert +1},\;$the solution of the penalized equation%
\begin{equation}%
\begin{tabular}
[c]{l}%
$du_{\varepsilon}\left(  t\right)  +\left(  A_{\varepsilon}^{\alpha
}u_{\varepsilon}\left(  t\right)  -\alpha u_{\varepsilon}\left(  t\right)
\right)  dt=f\left(  t\right)  dt+dM\left(  t\right)  ,\medskip$\\
$u_{\varepsilon}\left(  0\right)  =u_{0},\medskip$\\
$\eta_{\varepsilon}\left(  t\right)  =\int\nolimits_{0}^{t}\left(
A_{\varepsilon}^{\alpha}u_{\varepsilon}\left(  s\right)  -\alpha
u_{\varepsilon}\left(  s\right)  \right)  ds.$%
\end{tabular}
\ \label{2.13}%
\end{equation}
Then as \ $\varepsilon\rightarrow0:$

\noindent$%
\begin{tabular}
[c]{l}%
$u_{\varepsilon}\rightarrow u\quad$\textit{in}$\quad C\left(  \left[
0,T\right]  ;H\right)  ,\smallskip$\\
$\eta_{\varepsilon}\rightarrow\eta\quad$in$\quad C\left(  \left[  0,T\right]
;H\right)  ,\smallskip$\\
$\eta_{\varepsilon}\overset{w^{\ast}}{\rightarrow}\eta\quad$\textit{(weak
star) in }$BV\left(  \left[  0,T\right]  ;X^{\ast}\right)  .$%
\end{tabular}
$
\end{corollary}

\begin{remark}
\label{r2.2}We remark that \ $u_{\varepsilon}=v_{\varepsilon}+M$,\ where
$v_{\varepsilon}\in C^{1}\left(  \left[  0,T\right]  ;H\right)  $ is the
strong solution of the equation:%
\[
\left\{
\begin{tabular}
[c]{l}%
$v_{\varepsilon}^{\prime}+A_{\varepsilon}^{\alpha}\left(  v_{\varepsilon
}+M\left(  t\right)  \right)  -\alpha\left(  v_{\varepsilon}+M\left(
t\right)  \right)  =f\left(  t\right)  \smallskip$\\
$v_{\varepsilon}\left(  0\right)  =u_{0}$%
\end{tabular}
\ \ \right.
\]

\end{remark}

\begin{proof}
[Proof of Corollary \ref{c2.4}]Let\ $u_{0n},f_{n},M_{n}$ as in the proof of
Theorem \ref{t2.3} and $\left(  u_{\varepsilon}^{n},\eta_{\varepsilon}%
^{n}\right)  $ the strong solution of the equation (\ref{2.13}) corresponding
to $\left(  u_{0n},f_{n},M_{n}\right)  .$ Then by Proposition \ref{2.1}
$u_{\varepsilon}^{n}\rightarrow u^{n}$ in $C\left(  \left[  0,T\right]
;H\right)  $ as $\varepsilon\rightarrow0.$

Also since \ $A_{\varepsilon}^{\alpha}x-\alpha J_{\varepsilon}^{\alpha}x\in
A\left(  J_{\varepsilon}^{\alpha}x\right)  ,\;J_{\varepsilon}^{\alpha
}x=x-\varepsilon A_{\varepsilon}^{\alpha}x$ and $\left\vert J_{\varepsilon
}^{\alpha}x\right\vert \leq\left\vert x\right\vert +\left(  1+\left\vert
\varepsilon-1\right\vert \right)  \left\vert J_{1}^{\alpha}0\right\vert $ then
by (H$_{1}$-ii) we have $r_{0}\left\Vert A_{\varepsilon}^{\alpha}x-\alpha
x\right\Vert _{X^{\ast}}\leq\left(  A_{\varepsilon}^{\alpha}x-\alpha
x,x-h_{0}\right)  +b_{1}\left\vert x\right\vert ^{2}+b_{2},$ where
\ $b_{i}=b_{i}\left(  \alpha,h_{0},a_{1},a_{2}\right)  >0,\;\;i=1,2.$

Since $\left(  u_{\varepsilon},\eta_{\varepsilon}\right)  =GD\left(
A_{\varepsilon}^{\alpha}-\alpha I;u_{0},f,M\right)  $ and $\left(
u_{\varepsilon}^{n},\eta_{\varepsilon}^{n}\right)  =GD\left(  A_{\varepsilon
}^{\alpha}-\alpha I;u_{0n},f_{n},M_{n}\right)  $, then by (\ref{2.7}) and
(\ref{2.8})%
\[%
\begin{tabular}
[c]{l}%
$\left\Vert u_{\mathbb{\varepsilon}}\right\Vert _{C\left(  \left[  0,T\right]
;H\right)  }^{2}+\left\Vert \eta_{\varepsilon}\right\Vert _{BV\left(  \left[
0,T\right]  ;X^{\ast}\right)  }\leq C_{3}[1+\left\vert u_{0}\right\vert
^{2}+\left\vert f\right\vert _{L^{1}\left(  0,T;H\right)  }^{2}+\left\Vert
M\right\Vert _{C\left(  \left[  0,T\right]  ;H\right)  }^{2}],$\medskip\\
$\left\Vert u_{\mathbb{\varepsilon}}^{n}\right\Vert _{C\left(  \left[
0,T\right]  ;H\right)  }^{2}+\left\Vert \eta_{\varepsilon}^{n}\right\Vert
_{BV\left(  \left[  0,T\right]  ;X^{\ast}\right)  }\leq C_{3}[1+\left\vert
u_{on}\right\vert ^{2}+\left\vert f_{n}\right\vert _{L^{1}\left(
0,T;H\right)  }^{2}+\left\Vert M_{n}\right\Vert _{C\left(  \left[  0,T\right]
;H\right)  }^{2}]$%
\end{tabular}
\
\]
and%
\[%
\begin{tabular}
[c]{l}%
$\left\vert u_{\varepsilon}-u_{\varepsilon}^{n}\right\vert _{C\left(  \left[
0,T\right]  ;H\right)  }^{2}\leq C_{4}[\left\vert u_{0}-u_{0n}\right\vert
^{2}+\left\Vert f-f_{n}\right\Vert _{L^{1}\left(  0,T;H\right)  }^{2}$%
\medskip\\
$\quad+\left\Vert M-M_{n}\right\Vert _{C\left(  \left[  0,T\right]  ;H\right)
}^{2}+\left\Vert M-M_{n}\right\Vert _{C\left(  \left[  0,T\right]  ;X\right)
}]=\alpha_{n}$%
\end{tabular}
\
\]
where $C_{3},C_{4}$ are constants independent of $\varepsilon$ and $n,$ and
$\lim_{n\rightarrow\infty}\alpha_{n}=0.$ Also from the proof of Theorem
\ref{t2.3} we have%
\[
\left\Vert u-u_{n}\right\Vert _{C\left(  \left[  0,T\right]  ;H\right)  }%
^{2}\leq\alpha_{n}%
\]
Now from%
\[
\left\Vert u_{\varepsilon}-u\right\Vert _{C\left(  \left[  0,T\right]
;H\right)  }^{2}\leq3[\left\Vert u_{\varepsilon}-u_{\varepsilon}%
^{n}\right\Vert _{C\left(  \left[  0,T\right]  ;H\right)  }^{2}+\left\Vert
u_{\varepsilon}^{n}-u_{n}\right\Vert _{C\left(  \left[  0,T\right]  ;H\right)
}^{2}+\left\Vert u_{n}-u\right\Vert _{C\left(  \left[  0,T\right]  ;H\right)
}^{2}]
\]
we have%
\[
\limsup_{\varepsilon\rightarrow0}\left\Vert u_{\varepsilon}-u\right\Vert
_{C\left(  \left[  0,T\right]  ;H\right)  }\leq6\alpha_{n}\quad\text{for all
}n\in\mathbb{N}^{\ast}%
\]
and the conclusions of Corollary \ref{c2.4} follows easily.\hfill
\end{proof}

\begin{corollary}
\label{c2.5}If we substitute the assumption (\ref{2.6}-i) of Theorem
\ref{t2.3} by%
\begin{equation}
\text{\textit{operator }}A\;\text{\textit{satisfies }(}H_{1}\text{-III),}
\label{2.14}%
\end{equation}
then $\left(  u,n\right)  =GD\left(  A;u_{0},f,M\right)  $ if and only if
$\left(  u,\eta\right)  $ is the solution of the following problem:%
\begin{equation}%
\begin{tabular}
[c]{l}%
$i)\quad u\in C\left(  \left[  0,T\right]  ;H\right)  ,\;\;u\left(  t\right)
\in\overline{D\left(  A\right)  }\;\;\forall t\in\left[  0,T\right]
,\;u\left(  0\right)  =u_{0},$\medskip\\
$ii)\quad\eta\in C\left(  \left[  0,T\right]  ;H\right)  \cap BV\left(
\left[  0,T\right]  ;H\right)  ,\eta\left(  0\right)  =0,$\medskip\\
$\displaystyle iii)\quad u\left(  t\right)  +\eta\left(  t\right)  =u_{0}%
+\int\nolimits_{0}^{t}f\left(  s\right)  ds+M\left(  t\right)  ,\;\forall
t\in\left[  0,T\right]  ,$\medskip\\
$\displaystyle iv)\quad\int\nolimits_{s}^{t}\left(  u\left(  \tau\right)
-x,d\eta\left(  \tau\right)  -yd\tau\right)  +\alpha\int\nolimits_{s}%
^{t}\left\vert u\left(  \tau\right)  -x\right\vert ^{2}d\tau\geq0$\medskip\\
\quad\quad\quad\quad\quad$\forall0\leq s\leq t\leq T,$ and $\forall\left[
x,y\right]  \in A$%
\end{tabular}
\ \label{2.15}%
\end{equation}

\end{corollary}

\begin{proof}
If $\left(  u,\eta\right)  =GD\left(  A;u_{0},f,M\right)  $ then $\left(
u,\eta\right)  $ satisfies (\ref{2.15}-i,ii,iii). Also, since $\eta
_{n}^{\prime}+\alpha u_{n}\in\left(  A+\alpha I\right)  (u_{n})\,\;$and
$y+\alpha x\in\left(  A+\alpha I\right)  \left(  x\right)  $ for $\left[
x,y\right]  \in A,$ then by monotony of $A+\alpha I$ the inequality
(\ref{2.15}-iv) follows as $n\rightarrow\infty.$

To finish the proof of Corollary \ref{2.3} we have to prove only the
uniqueness of the solution of the problem (\ref{2.15}). It is clearly that
(\ref{2.15}-iv) gives%
\begin{equation}
\int\nolimits_{s}^{t}\left(  u\left(  \tau\right)  -a\left(  \tau\right)
,d\eta\left(  \tau\right)  -b\left(  \tau\right)  d\tau\right)  +\alpha
\int\nolimits_{s}^{t}\left\vert u\left(  \tau\right)  -a\left(  \tau\right)
\right\vert ^{2}d\tau\geq0 \label{2.16}%
\end{equation}
for all $a\left(  \tau\right)  ,b\left(  \tau\right)  $ step functions on
$\left[  0,T\right]  $ such that $\left[  a\left(  \tau\right)  ,b\left(
\tau\right)  \right]  \in A$, $\forall\tau\in\left[  0,T\right]  $, and then
for all $a,b\in C\left(  \left[  0,T\right]  ;H\right)  $, $[a\left(
\tau\right)  $,$b\left(  \tau\right)  ]\in A\;\forall\tau\in\left[
0,T\right]  $. Now the uniqueness follows by a standard argument. We write
(\ref{2.16}) for $u=u_{1}$ and $u=u_{2},$ $\;a=J_{\varepsilon}^{\alpha}\left(
\frac{u_{1}+u_{2}}{2}\right)  ,\;b=A_{\varepsilon}^{\alpha}\left(  \frac
{u_{1}+u_{2}}{2}\right)  -\alpha J_{\varepsilon}^{\alpha}\left(  \frac
{u_{1}+u_{2}}{2}\right)  $; by addition of the two inequalities and by passing
to limit for $\varepsilon\rightarrow0$ the following inequality follows%
\begin{equation}%
\begin{tabular}
[c]{r}%
$\displaystyle\frac{1}{2}\int\nolimits_{s}^{t}\left(  u_{1}\left(
\tau\right)  -u_{2}\left(  \tau\right)  ,d\eta_{1}\left(  \tau\right)
-d\eta_{2}\left(  \tau\right)  \right)  +\alpha\int\nolimits_{s}^{t}\left\vert
u_{1}\left(  \tau\right)  -u_{2}\left(  \tau\right)  \right\vert ^{2}d\tau
\geq0,$\medskip\\
for all $0\leq s\leq t\leq T.$%
\end{tabular}
\ \label{2.17}%
\end{equation}
From (\ref{2.17}-iii) we have $\;u_{1}\left(  t\right)  -u_{2}\left(
t\right)  =\eta_{2}\left(  t\right)  -\eta_{1}\left(  t\right)  .$ \ Hence by
(\ref{2.17}):%
\[
\left\vert u_{1}\left(  t\right)  -u_{2}\left(  t\right)  \right\vert ^{2}%
\leq4\alpha\int\nolimits_{0}^{t}\left\vert u_{1}\left(  \tau\right)
-u_{2}\left(  \tau\right)  \right\vert ^{2}d\tau\;\;\forall t\in\left[
0,T\right]  ,
\]
which implies $u_{1}=u_{2}.$\hfill
\end{proof}

\begin{corollary}
\label{c2.6} If we substitute the assumption (\ref{2.6}-i) of Theorem
\ref{t2.3} by
\begin{equation}
\text{\textit{operator }}A\text{\textit{\ satisfies (}}H_{1}%
-\text{II\textbf{),}} \label{2.18}%
\end{equation}
then moreover the generalized solution $u\in L^{2}\left(  0,T;V\right)  .$
\end{corollary}

\begin{proof}
By the assumption (H$_{1}$-II-c) we have that $u_{n}$ is a Cauchy sequence in
$L^{2}\left(  0,T;V\right)  .$\hfill
\end{proof}

\subsection{An example}

Let $D$ be an open bounded subset of $\mathbb{R}^{d}$with a sufficiently
smooth boundary $\Gamma,$ and let $\beta\subset\mathbb{R}\times\mathbb{R}$ be
a maximal monotone graph or equivalent $\beta=\partial j,$ where
$j:\mathbb{R}\rightarrow\left]  -\infty,+\infty\right]  $ is a convex
lower-semicontinuous function. We assume that $\exists\,b_{0}\in
\mathbb{R},\;\exists\,\varepsilon_{0}>0$ such that $\left[  b_{0}%
-\varepsilon_{0},b_{0}+\varepsilon_{0}\right]  \subset$\ Dom$\left(  j\right)
.$ Also let $g:\mathbb{R}\rightarrow\mathbb{R}$ be a continuous function such
that
\[%
\begin{tabular}
[c]{l}%
$\left(  g\left(  r\right)  -g\left(  q\right)  \right)  \left(  r-q\right)
+\alpha\left\vert r-q\right\vert ^{2}\geq0\;\;\forall r,q\in R$\medskip\\
$\left\vert g\left(  r\right)  \right\vert \leq b\left(  1+\left\vert
r\right\vert \right)  $%
\end{tabular}
\ \
\]
($\alpha,b$ are some given positive constants). Consider the following
problem
\begin{equation}
\left\{
\begin{tabular}
[c]{l}%
$du\left(  t\right)  -\Delta u\left(  t\right)  dt+g\left(  u\left(  t\right)
\right)  dt=f\left(  t\right)  dt+dM\left(  t\right)  $ on $\left]
0,T\right[  \times D,$\medskip\\
$-\dfrac{\partial u\left(  t,x\right)  }{\partial n}\in\beta\left(  u\left(
t,x\right)  \right)  $ on $\left]  0,T\right[  \times\Gamma,$\medskip\\
$u\left(  0,x\right)  =u_{0}\left(  x\right)  ,$ on $D$%
\end{tabular}
\ \ \right.  \label{2.19}%
\end{equation}
or equivalent%
\[
\left\{
\begin{tabular}
[c]{l}%
$du\left(  t\right)  +\partial\varphi\left(  u\left(  t\right)  \right)
\left(  dt\right)  +g\left(  u\left(  t\right)  \right)  dt\ni f\left(
t\right)  dt+dM\left(  t\right)  $\medskip\\
$u\left(  0\right)  =u_{0}$%
\end{tabular}
\ \ \right.
\]
where \ $\varphi:H=L^{2}\left(  D\right)  \rightarrow\left]  -\infty
,+\infty\right]  $ \ is the convex l.s.c. function given by
\[
\varphi\left(  u\right)  =\left\{
\begin{tabular}
[c]{l}%
$\displaystyle\frac{1}{2}\int\nolimits_{D}\left\vert \text{grad }u\right\vert
^{2}dx+\int\nolimits_{\Gamma}j\left(  u\right)  d\sigma,$ if $u\in
H^{1}\left(  D\right)  ,\;j\left(  u\right)  \in L^{1}\left(  \Gamma\right)
$\medskip\\
$+\infty,$ otherwise
\end{tabular}
\ \ \right.
\]
The assumptions from (H$_{1}$-I) are satisfied for $\left(  A_{0}u\right)
\left(  x\right)  =g\left(  u\left(  x\right)  \right)  $ and the Sobolev
space $H^{k}\left(  D\right)  =X,\;k>\frac{d}{2}.$ Hence if $f\in L^{1}\left(
0,T;L^{2}\left(  D\right)  \right)  $, $u_{0}\in H^{1}\left(  D\right)  $,
$j\left(  u_{0}\right)  \in L^{1}\left(  \Gamma\right)  $, $M\in C\left(
\left[  0,T\right]  ;H^{k}\left(  D\right)  \right)  $ then the equation
(\ref{2.19}) has a unique generalized solution $u\in C\left(  \left[
0,T\right]  ;L^{2}\left(  D\right)  \right)  $ in the sense of Definition
\ref{d2.1}.

\section{Stochastic evolution equations}

\subsection{Preliminaries}

\textbf{A.} We assume as given a filtered complete probability space
$(\Omega,\mathcal{F},\mathbb{P},\left\{  \mathcal{F}_{t}\right\}  _{t\geq0})$
which satisfy the usual hypotheses i.e. $\left(  \Omega,\mathcal{F}%
,\mathbb{P}\right)  $ is a complete probability space and $\left\{
\mathcal{F}_{t},t\geq0\right\}  $ is a increasing right continuous
sub-$\sigma$-algebras of $\mathcal{F}$. We shall say that $(\Omega
,\mathcal{F},\mathbb{P},\left\{  \mathcal{F}_{t}\right\}  _{t\geq0})$. is a
stochastic base.

If $Y$ is a real separable Banach space we denote by $L_{ad}^{r}\left(
\Omega;C\left(  \left[  0,T\right]  ;Y\right)  \right)  $, $r\geq0$, the
closed linear subspace (Banach space for $r\in\left[  1,\infty\right)  $ and
metric space for $0\leq r<1;$ the metric of convergence in probability for
$r=0$) of adapted stochastic processes $f\in L^{r}\left(  \Omega
,\mathcal{F},\mathbb{P};C\left(  \left[  0,T\right]  ;Y\right)  \right)  .$
Similarly $L_{ad}^{r}\left(  \Omega;L^{q}\left(  0,T;X\right)  \right)  $,
$r\geq0$, $q\in\lbrack1,\infty)$, denoted the Banach space for $r\geq1$ or the
metric space for $0\leq r<1$ of measurable stochastic processes $f\in
L^{r}\left(  \Omega;L^{q}\left(  0,T;Y\right)  \right)  $ such that $f\left(
\cdot,t\right)  $ is $\mathcal{F}_{t}$-measurable a.e. $t\in\left(
0,T\right)  $.

\textbf{B.} If $H$ is a real separable Hilbert space we denote by
$\mathcal{M}^{p}\left(  0,T;H\right)  $, $p\in\left[  1,\infty\right)  $, the
space of continuous $p$-martingales $M$, that is
\begin{equation}%
\begin{array}
[c]{rl}%
\left(  i\right)  & M\in L_{ad}^{0}\left(  \Omega;C\left(  \left[  0,T\right]
;H\right)  \right)  ,\medskip\\
\left(  ii\right)  & \mathbb{E}\left\vert M_{t}\right\vert ^{p}<\infty
,\quad\text{for all }t\geq0,\medskip\\
\left(  iii\right)  & M\left(  \omega,0\right)  =0,~a.s.~\omega\in
\Omega,\medskip\\
\left(  iv\right)  & \mathbb{E}\left(  M\left(  t\right)  |\mathcal{F}%
_{s}\right)  =M\left(  s\right)  ,~a.s.,~\text{if\ }0\leq s\leq t\leq T.
\end{array}
\label{3.1}%
\end{equation}
If $M\in\mathcal{M}^{2}\left(  0,T;H\right)  $ then by Doob-Meyer
decomposition, there exists a unique stochastic process $\left\langle
M\right\rangle \in L_{ad}^{1}\left(  \Omega;C\left(  \left[  0,T\right]
;R\right)  \right)  $ such that%
\begin{equation}%
\begin{tabular}
[c]{l}%
$i)\quad t\longmapsto\left\langle M\right\rangle \left(  \omega,t\right)  $ is
increasing $\mathbb{P}$-a.s.,$\medskip$\\
$ii)\quad\left\vert M\right\vert _{H}^{2}-\left\langle M\right\rangle
\in\mathcal{M}^{1}\left(  0,T;R\right)  .$%
\end{tabular}
\ \ \label{3.2}%
\end{equation}
Moreover%
\begin{equation}%
\begin{tabular}
[c]{l}%
$a)\quad\mathbb{E}\sup\limits_{s\in\left[  0,T\right]  }\left\vert M\left(
s\right)  \right\vert \leq3\mathbb{E}\sqrt{\left\langle M\right\rangle \left(
T\right)  },\medskip$\\
$b)\quad\mathcal{M}^{p}\left(  0,T;H\right)  $ is a closed linear subspace
of$\medskip$\\
\quad\quad\quad$L_{ad}^{p}\left(  \Omega;C\left(  \left[  0,T\right]
;H\right)  \right)  ,\;$for all $p\in(1,\infty).$%
\end{tabular}
\ \ \label{3.3}%
\end{equation}
On $\mathcal{M}^{p}\left(  0,T;H\right)  ,\,p>1$, it is defined the norm
$\left\Vert M\right\Vert _{\mathcal{M}^{p}}=\left(  \mathbb{E}\left\vert
M\left(  T\right)  \right\vert _{H}^{p}\right)  ^{1/p}$ ; in the case $p>1$
this norm is equivalent on $M^{p}\left(  0,T;H\right)  $ to usual norm from
$L^{p}\left(  \Omega;C\left(  \left[  0,T\right]  ;H\right)  \right)  $. The
space $\left(  \mathcal{M}^{2}\left(  0,T;H\right)  ,\left\Vert \cdot
\right\Vert _{\mathcal{M}^{2}}\right)  $ is a Hilbert space.

For any $f\in L_{ad}^{r}\left(  \Omega;C\left(  \left[  0,T\right]  ;H\right)
\right)  $, $0\leq r<\infty$ and $M\in\mathcal{M}^{2}\left(  0,T;H\right)  $
the stochastic integral $I\left(  f;M\right)  \left(  t\right)
=\displaystyle\int\nolimits_{0}^{t}\left(  f\left(  s\right)  ,dM\left(
s\right)  \right)  $ is well defined and has the properties:%
\begin{equation}%
\begin{tabular}
[c]{l}%
$a)\quad I:L_{ad}^{2}\left(  \Omega;C\left(  \left[  0,T\right]  ;H\right)
\right)  \times\mathcal{M}^{2}\left(  0,T;H\right)  \rightarrow\mathcal{M}%
^{1}\left(  0,T;\mathbb{R}\right)  \medskip$\\
\quad\quad is continuous, and also$\medskip$\\
\quad\quad$I:L_{ad}^{0}\left(  \Omega;C\left(  \left[  0,T\right]  ;H\right)
\right)  \times\mathcal{M}^{2}\left(  0,T;H\right)  \rightarrow L_{ad}%
^{0}\left(  \Omega;C\left(  \left[  0,T\right]  ;\mathbb{R}\right)  \right)
\medskip$\\
$\quad$\quad is continuous;$\medskip$\\
$b)\quad\left\vert M\left(  t\right)  \right\vert ^{2}=2\displaystyle\int%
\nolimits_{0}^{t}\left(  M\left(  t\right)  ,dM\left(  t\right)  \right)
+\left\langle M\right\rangle \left(  t\right)  ,\;\forall t\in\left[
0,T\right]  ,$ a.s. $\omega\in\Omega;\medskip$\\
$c)\quad$if $f\in L_{ad}^{2}\left(  \Omega;C\left(  \left[  0,T\right]
;H\right)  \right)  $ then$\medskip$\\
\quad\quad$\displaystyle\mathbb{E}\sup\limits_{t\in\left[  0,T\right]
}\left\vert \int\nolimits_{0}^{t}\left(  f\left(  s\right)  ,dM\left(
s\right)  \right)  \right\vert \leq3\mathbb{E}\left[  \sup\limits_{t\in\left[
0,T\right]  }\left\vert f\left(  t\right)  \right\vert _{H}\sqrt{\left\langle
M\right\rangle \left(  T\right)  }\right]  \medskip$\\
$d)\quad$if $\quad0=t_{0}<t_{1}<...<t_{n}=T,\;\delta_{n}=\max
\limits_{i=\overline{0,n-1}}\left\vert t_{i+1}-t_{i}\right\vert $ \quad
and$\medskip$\\
\quad\quad$\displaystyle I_{n}\left(  f\right)  \left(  t\right)
=\sum\limits_{i=0}^{n-1}\left(  f\left(  t_{i}\right)  ,M\left(  t_{i+1}\wedge
t\right)  -M\left(  t_{i}\wedge t\right)  \right)  $ \quad then$\medskip$\\
\quad\quad$I\left(  f;M\right)  =\lim\limits_{\delta_{n}\rightarrow0}%
I_{n}\left(  f\right)  \left\{
\begin{tabular}
[c]{l}%
in $\mathcal{M}^{1}\left(  0,T;\mathbb{R}\right)  $, \ for$\medskip$\\
$\,\,\quad f\in L_{ad}^{2}\left(  \Omega;C\left(  \left[  0,T\right]
;H\right)  \right)  \medskip$\\
in$\;L_{ad}^{0}\left(  \Omega;C\left(  \left[  0,T\right]  ;\mathbb{R}\right)
\right)  ,$ \quad for$\medskip$\\
$\quad f\in L_{ad}^{0}\left(  \Omega;C\left(  \left[  0,T\right]  ;H\right)
\right)  $%
\end{tabular}
\ \ \ \ \ \right.  $%
\end{tabular}
\ \ \ \ \label{3.4}%
\end{equation}
\textbf{C.} Let $\left(  U\;\left\vert \cdot\right\vert _{U}\right)  $ be a
real separable Hilbert space and $Q:U\rightarrow U$ be a linear operator. We
shall assume%
\begin{equation}%
\begin{tabular}
[c]{l}%
$i)\quad Qu=\sum\limits_{i}\lambda_{i}\left(  u,e_{i}\right)  _{U}e_{i},$
where\medskip\\
\quad\quad$\lambda_{i}\geq0,\;\sum\limits_{i}\lambda_{i}<\infty,\;\;\left\{
e_{i}\right\}  $ orthonormal basis in $U,$\medskip\\
$ii)\quad W\in\mathcal{M}^{2}\left(  0,T;U\right)  $ is a $U$-valued Wiener
process with\medskip\\
\quad\quad covariance operator $Q.$%
\end{tabular}
\ \ \ \ \label{3.5}%
\end{equation}
Hence%
\begin{equation}%
\begin{tabular}
[c]{l}%
$a)\quad W\left(  \omega,t\right)  =\sum\limits_{i}\sqrt{\lambda_{i}}\beta
_{i}\left(  \omega,t\right)  e_{i},$\medskip\\
$b)\quad\beta_{i}\in L_{ad}^{2}\left(  \Omega;C\left(  \left[  0,T\right]
;\mathbb{R}\right)  \right)  $ are independent real Brownian\medskip\\
\quad\quad motions with respect to the stochastic basis $(\Omega
,\mathcal{F},\mathbb{P},\left\{  \mathcal{F}_{t}\right\}  _{t\geq0})$%
\medskip\\
$c)\quad W\in\mathcal{M}^{p}\left(  0,T;U\right)  ,\;\;\forall p\geq1.$%
\end{tabular}
\ \ \ \ \label{3.6}%
\end{equation}
To define the stochastic integral with respect to a $Q$-Wiener process one
introduces the Hilbert space $U_{0}=Q^{1/2}\left(  U\right)  $ endowed with
the inner product%
\[
\left(  u,v\right)  _{U_{0}}=\sum\limits_{\lambda_{k}\neq0}\frac{1}%
{\lambda_{k}}\left(  u,e_{k}\right)  _{U}\left(  v,e_{k}\right)  _{U}%
\]
and the space of all Hilbert-Schmidt operators $L_{0}^{2}=L$ $^{2}\left(
U_{0},H\right)  $. The space $L_{0}^{2}$ is a separable Hilbert space,
equipped with the Hilbert-Schmidt norm:%
\[
\left\vert f\right\vert _{Q}^{2}\overset{\Delta}{=}\left\Vert f\right\Vert
_{\mathcal{L}_{0}^{2}}^{2}=\sum\limits_{i}\lambda_{i}\left\vert fe_{i}%
\right\vert ^{2}=\sum\limits_{i}\left\vert fQ^{1/2}e_{i}\right\vert
^{2}=\left\Vert fQ^{1/2}\right\Vert _{HS}^{2}=tr\;fQf^{\ast}.
\]
The stochastic integral \ $I\left(  f\right)  \left(  \omega,t\right)
=\displaystyle\int\nolimits_{0}^{t}f\left(  \omega,s\right)  dW\left(
\omega,s\right)  $ is well defined for $\;f\in L_{ad}^{r}(\Omega
;L^{2}(0,T;\mathcal{L}_{0}^{2}))  $,\ $0\leq r<\infty$
and%
\begin{equation}%
\begin{tabular}
[c]{l}%
$a)\quad I\left(  f\right)  \in L_{ad}^{r}\left(  \Omega;C\left(  \left[
0,T\right]  ;H\right)  \right)  ,$\medskip\\
$b)\quad\mathbb{E}I\left(  f\right)  \left(  t\right)  =0,$ for $r\geq
1,$\medskip\\
$c)\quad\displaystyle\mathbb{E}\left\vert I\left(  f\right)  \left(  t\right)
\right\vert ^{2}=\mathbb{E}\int\nolimits_{0}^{t}\left\vert f\left(  s\right)
\right\vert _{Q}^{2}ds,$ for $r\geq2,$\medskip\\
$d)\quad$(Burkh\"{o}lder-Davis-Gundy inequality): $\forall r>0,\;\exists
~c_{r},C_{r}>0$ $\left(  \footnotemark\right)  :$\medskip\\
\quad\quad$\displaystyle c_r\mathbb{E~}\left(  \int\nolimits_0^T\left\vert
f\right\vert _Q^2ds\right)  ^r/2\leq\mathbb{E}\sup\limits_t\in\left[
0,T\right]  \left\vert I\left(  f\right)  \left(  t\right)  \right\vert ^r\leq
C_r\mathbb{E~}\left(  \int\nolimits_0^T\left\vert f\right\vert _Q^2ds\right)
^r/2,$\medskip\\
$e)\quad I\left(  f\right)  \in\mathcal{M}^r\left(  0,T;H\right)  ,$ for $%
r\in\left[  1,\infty\right)  .$
\end{tabular}
\ \ \ \ \label{3.7}%
\end{equation}
\footnotetext{$C_{2}=4$ (Doob inequality), $C_{r}\leq3$ if $0<r\leq1$and
$C_{r}\leq9\left(  2r\right)  ^{r}$ \ if $r>1.$}\textbf{D.} Let\ $X\subset
H\cong H^{\ast}\subset X^{\ast}$ be the spaces defined as in Section 2
(Subsection \ref{2.1}: Preliminaries).

\begin{proposition}
\label{p3.1} (Integration by parts)%
\[%
\begin{tabular}
[c]{l}%
$a)\quad$if $\;m\in L^{0}\left(  \Omega;C\left(  \left[  0,T\right]
;X\right)  \right)  \cap\mathcal{M}^{2}\left(  0,T;H\right)  $\medskip\\
\quad$\,\,\,\,\,\,\eta\in L_{ad}^{0}\left(  \Omega;C\left(  \left[
0,T\right]  ;H\right)  \right)  \cap L^{0}\left(  \Omega;BV\left(
0,T;X^{\ast}\right)  \right)  ,\;$\thinspace then\medskip\\
\quad\quad$\displaystyle\int\nolimits_{0}^{t}(m\left(  s\right)  ,d\eta\left(
s\right)  )=\left(  m\left(  t\right)  ,\eta\left(  t\right)  \right)
-\int\nolimits_{0}^{t}\left(  \eta\left(  s\right)  ,dm\left(  s\right)
\right)  $\medskip\\
\quad\quad\quad\quad\quad\quad\quad\quad\quad for all\ \ $t\in\left[
0,T\right]  ,$\ a.s.\ $\omega\in\Omega.$\medskip\\
$b)\quad$Moreover if\medskip\\
\quad\quad$\displaystyle u\left(  t\right)  +\eta\left(  t\right)  =u_{0}%
+\int\nolimits_{0}^{t}f\left(  s\right)  ds+m\left(  t\right)  ,\forall
t\in\left[  0,T\right]  ,$a.s.,\medskip\\
\quad\quad where\quad\quad$u\in L_{ad}^{0}\left(  \Omega;C\left(  \left[
0,T\right]  ;H\right)  \right)  ,$\medskip\\
\quad\quad$u_{0}\in L^{0}\left(  \Omega,\mathcal{F}_{0},\mathbb{P};H\right)
,\,\,f\in L_{ad}^{0}\left(  \Omega;L^{1}\left(  0,T;H\right)  \right)  $%
\end{tabular}
\]
then%
\begin{equation}%
\begin{tabular}
[c]{l}%
$\displaystyle\left\vert u\left(  t\right)  -m\left(  t\right)  \right\vert
^{2}=\left\vert u\left(  t\right)  \right\vert ^{2}-2\int\nolimits_{0}%
^{t}\left(  u\left(  s\right)  ,dm\left(  s\right)  \right)  -2\int%
\nolimits_{0}^{t}\left(  m\left(  s\right)  ,f\left(  s\right)  \right)
ds+$\medskip\\
$\;\;\displaystyle+2\int\nolimits_{0}^{t}\left(  m\left(  s\right)
,d\eta\left(  s\right)  \right)  -\left\langle m\right\rangle \left(
t\right)  $\ \ for all $\;t\in\left[  0,T\right]  ,$\ a.s. $\omega\in\Omega.$%
\end{tabular}
\ \label{3.8}%
\end{equation}

\end{proposition}

\begin{proof}
Denote $F\left(  t\right)  =\int\nolimits_{0}^{t}f\left(  s\right)  ds.$ Let
$0=t_{0}<t_{1}<...<t_{n}=t,\;\frac{t}{n}=t_{i+1}-t_{i}$ and $g\left(
t_{i}\right)  =g_{i}.$ From the definition of Riemann-Stieltjes and the
properties of stochastic integral (\ref{3.4}) we have for $\mathbb{P}-$a.s.
$\omega\in\Omega$ (on a subsequence $n_{k}$ denoted also $n$)
\[%
\begin{tabular}
[c]{l}%
$\displaystyle\int_{0}^{t}\left(  m\left(  s\right)  ,d\eta\left(  s\right)
\right)  =\lim\limits_{n\rightarrow\infty}\sum\limits_{i=0}^{n-1}\left(
m_{i+1},\eta_{i+1}-\eta_{i}\right)  $\medskip\\
$\displaystyle=\lim\limits_{n\rightarrow\infty}[\left(  m_{n},\eta_{n}\right)
-\left(  m_{0},\eta_{0}\right)  -\sum\limits_{i=0}^{n-1}\left(  \eta
_{i},m_{i+1}-m_{i}\right)  ]$\medskip\\
$\displaystyle=\left(  m\left(  t\right)  ,\eta\left(  t\right)  \right)
-\int\nolimits_{0}^{t}\left(  \eta\left(  s\right)  ,dm\left(  s\right)
\right)  $\medskip\\
$\displaystyle=\left(  m\left(  t\right)  ,u_{0}+F\left(  t\right)  +m\left(
t\right)  -u\left(  t\right)  \right)  $\medskip\\
$\quad\displaystyle-\int\nolimits_{0}^{t}\left(  u_{0}+F\left(  s\right)
+m\left(  s\right)  -u\left(  s\right)  ,dm\left(  s\right)  \right)
$\medskip\\
$=\displaystyle[\left(  m\left(  t\right)  ,F\left(  t\right)  \right)
-\int\nolimits_{0}^{t}\left(  F\left(  s\right)  ,dm\left(  s\right)  \right)
]$\medskip\\
$\quad\displaystyle+[\int\nolimits_{0}^{t}\left(  u\left(  s\right)
,dm\left(  s\right)  \right)  -\left(  m\left(  t\right)  ,u\left(  t\right)
\right)  ]\,$\medskip\\
$\quad\displaystyle+[\left\vert m\left(  t\right)  \right\vert ^{2}%
-\int\nolimits_{0}^{t}\left(  m\left(  s\right)  ,dm\left(  s\right)  \right)
]$\medskip\\
$=\displaystyle\int\nolimits_{0}^{t}\left(  f\left(  s\right)  ,m\left(
s\right)  \right)  ds+[\int\nolimits_{0}^{t}\left(  u\left(  s\right)
,dm\left(  s\right)  \right)  -\left(  m\left(  t\right)  ,u\left(  t\right)
\right)  ]$\medskip\\
$\quad\displaystyle+[\frac{1}{2}\left\vert m\left(  t\right)  \right\vert
^{2}+\frac{1}{2}\left\langle m\right\rangle \left(  t\right)  ]$\medskip\\
$=\displaystyle\int\nolimits_{0}^{t}\left(  f\left(  s\right)  ,m\left(
s\right)  \right)  ds+\int\nolimits_{0}^{t}\left(  u\left(  s\right)
,dm\left(  s\right)  \right)  $\medskip\\
$\quad+\displaystyle\frac{1}{2}\left\vert u\left(  t\right)  -m\left(
t\right)  \right\vert ^{2}-\frac{1}{2}\left\vert u\left(  t\right)
\right\vert ^{2}+\frac{1}{2}\left\langle m\right\rangle \left(  t\right)  .$%
\end{tabular}
\ \ \ \
\]
\hfill
\end{proof}

\textbf{E.} Finally it is interesting to recall a general Pardoux's result
from \cite{PaEA}. We shall see that our results means a generalization, also,
for unbounded operator but in the multivalued case.

Let $\left(  H,\left(  \cdot,\cdot\right)  ,\left\vert \cdot\right\vert
\right)  $ be a real separable Hilbert space and let $\left(  V,\left\Vert
\cdot\right\Vert \right)  $ be a real separable reflexive Banach space with
the dual $\left(  V^{\ast},\left\Vert \cdot\right\Vert _{V^{\ast}}\right)  $
strictly convex space. Assume $V\subset H\cong H^{\ast}\subset V^{\ast}$ with
continuous densely embeddings.

Consider the equation
\begin{equation}%
\begin{tabular}
[c]{l}%
$dy\left(  t\right)  +A\left(  t,y\left(  t\right)  \right)  dt=f\left(
t\right)  dt+B\left(  t,y\left(  t\right)  \right)  dW\left(  t\right)
+dM\left(  t\right)  $\medskip\\
$y\left(  0\right)  =y_{0},\;\;t\in\left[  0,T\right]  ,$%
\end{tabular}
\ \ \label{3.9}%
\end{equation}
where $A\left(  t,\cdot\right)  :V\rightarrow V^{\ast}$ defined for a.e.
$t\in(0,T)$ satisfies: there exist the constants $p>1,\;\,\,a,\alpha
>0,\;\,\,d,d_{0}\in\mathbb{R}$ such that for all $u,v,z\in V$ the following
properties hold a.e. $t\in(0,T):$%
\begin{equation}%
\begin{tabular}
[c]{l}%
$i)\quad2\left\langle A\left(  t,u\right)  ,u\right\rangle +d\left\vert
u\right\vert ^{2}+d_{0}\geq\alpha\left\Vert u\right\Vert ^{p},$\medskip\\
$ii)\quad2\left\langle A\left(  t,u\right)  -A\left(  t,v\right)
,u-v\right\rangle +d\left\vert u-v\right\vert ^{2}\geq0,$\medskip\\
$iii)\quad\left\Vert A\left(  t,u\right)  \right\Vert _{\ast}^{p^{\prime}}\leq
a\left(  1+\left\Vert u\right\Vert ^{p}\right)  ,\;1/p+1/p^{\prime}%
=1,$\medskip\\
$iv)\quad s\longmapsto\left\langle A\left(  t,u+sv\right)  ,z\right\rangle
:\mathbb{R}\rightarrow\mathbb{R}$ is a continuous function,\medskip\\
$v)\quad s\longmapsto\left\langle A\left(  s,u\right)  ,v\right\rangle
:\left]  0,T\right[  \rightarrow\mathbb{R}$ is Lebesgue measurable,
\end{tabular}
\ \ \label{3.10}%
\end{equation}
and $B\left(  t,\cdot\right)  :H\rightarrow L^{2}\left(  U_{0},H\right)  $
defined for a.e. $t\in\left]  0,T\right[  $ satisfies: $\exists L,b>0$ such
that for all $u,v\in H,\;w\in U_{0}$ we have a.e. $t\in\left]  0,T\right[  :$%
\begin{equation}%
\begin{tabular}
[c]{l}%
$i)\quad\left\vert B\left(  t,u\right)  -B\left(  t,v\right)  \right\vert
_{Q}^{2}\leq L\left\vert u-v\right\vert ^{2},$\medskip\\
$ii)\quad\left\vert B\left(  t,u\right)  \right\vert _{Q}^{2}\leq b\left(
1+\left\vert u\right\vert ^{2}\right)  ,$\medskip\\
$iii)\quad s\longmapsto\left(  B\left(  s,u\right)  w,v\right)  :\left]
0,T\right[  \rightarrow\mathbb{R}$ is Lebesgue measurable
\end{tabular}
\ \ \label{3.11}%
\end{equation}
Also it is assumed:
\begin{equation}%
\begin{tabular}
[c]{l}%
$i)\quad y_{0}\in L^{2}\left(  \Omega,\mathcal{F}_{0},\mathbb{P};H\right)
,$\medskip\\
$ii)\quad f=f_{1}+f_{2},\;f_{1}\in L_{ad}^{2}\left(  \Omega;L^{1}\left(
0,T;H\right)  \right)  ,\;f_{2}\in L_{ad}^{p^{\prime}}\left(  \Omega
\times\left]  0,T\right[  ;V^{\ast}\right)  ,$\medskip\\
$iii)\quad M\in\mathcal{M}^{2}\left(  0,T;H\right)  ,$\medskip\\
$iv)\quad\left\{  W\left(  t\right)  ,t\geq0\right\}  $ is a $Q$-Wiener
process.
\end{tabular}
\ \ \label{3.12}%
\end{equation}

\begin{theorem}
\label{t3.2}(E. Pardoux \cite{PaEA}). Under the assumptions (\ref{3.10})
\ (\ref{3.11})\ and (\ref{3.12})\ the equation (\ref{3.9})\ has a unique
solution $y\in L_{ad}^{p}\left(  \Omega\times\left]  0,T\right[  ;V\right)
\cap L_{ad}^{2}\left(  \Omega;C\left(  \left[  0,T\right]  ;H\right)  \right)
$. The solution $y$\ satisfies

\noindent a) (Energy Equality)%
\[%
\begin{array}
[c]{l}%
\displaystyle\left\vert y\left(  t\right)  \right\vert ^{2}+2\int%
\nolimits_{0}^{t}\left\langle A\left(  s,y\left(  s\right)  \right)  ,y\left(
s\right)  \right\rangle ds=\left\vert y_{0}\right\vert ^{2}+2\int%
\nolimits_{0}^{t}\left\langle f\left(  s\right)  ,y\left(  s\right)
\right\rangle ds+2\int\nolimits_{0}^{t}\left(  y\left(  s\right)  ,B\left(
s,y\left(  s\right)  \right)  dW\left(  s\right)  \right)  \medskip\\
\displaystyle+2\int\nolimits_{0}^{t}\left(  y\left(  s\right)  ,dM\left(
s\right)  \right)  +\langle M-\int\nolimits_{0}^{\cdot}B\left(  s,y\left(
s\right)  \right)  dW\left(  s\right)  \rangle\left(  t\right)  \ forallt\in
\left[  0,T\right]  ,~a.e.~\omega\in\Omega.
\end{array}
\]
\noindent b)\quad$\displaystyle m_{1}\left(  \cdot\right)  =\int%
\nolimits_{0}^{\cdot}\left(  y\left(  s\right)  ,B\left(  s,y\left(  s\right)
\right)  dW\left(  s\right)  \right)  ,\;m_{2}\left(  \cdot\right)
=\int\nolimits_{0}^{\cdot}\left(  y\left(  s\right)  ,dM\left(  s\right)
\right)  $\ are martingales from $M^{1}\left(  0,T;R\right)  .$
\end{theorem}

\subsection{$\alpha$-Monotone SDE with additive noise}

Let \ $X\subset H\cong H^{\ast}\subset X^{\ast}$ be the spaces defined at the
beginning of Section 2 that is: $\left(  H,\left\vert \cdot\right\vert
\right)  $ is a real separable Hilbert space and $\left(  X,\left\Vert
\cdot\right\Vert \right)  $ is a real separable Banach space with the dual
$\left(  X^{\ast},\left\Vert \cdot\right\Vert _{\ast}\right)  $ separable too.
The inclusion mapping $X\subset H$ is continuous and $X$ is dense in $H.$

We shall assume given a stochastic basis $\;(\Omega,\mathcal{F},\mathbb{P}%
,\left\{  \mathcal{F}_{t}\right\}  _{t\geq0}).$ Consider the multivalued
stochastic differential equation:%
\begin{equation}
\left\{
\begin{tabular}
[c]{l}%
$du\left(  t\right)  +Au\left(  t\right)  dt\ni f\left(  t\right)
dt+dM\left(  t\right)  ,$\medskip\\
$u\left(  0\right)  =u_{0},\quad\quad\quad\quad\;\;t\in\left[  0,T\right]  ,$%
\end{tabular}
\ \ \ \ \right.  \label{3.13}%
\end{equation}
where we put the assumptions:%
\[%
\begin{tabular}
[c]{l}%
(A$_{1}$)$\left\{
\begin{array}
[c]{c}%
A:H\rightarrow2^{H}\,\;\text{ satisfies (H}_{1}\text{), that is\ \quad
\quad\quad\quad\ \ \ \ \ \ \ \ \ \ \ \ \ \ \ \ \ \ \ \ ~}\medskip\\%
\begin{tabular}
[c]{l}%
$i)\;A:H\rightarrow2^{H}$ is $\alpha$-maximal monotone operator $\left(
\alpha\in\mathbb{R}\right)  ,\medskip$\\
$ii)\;\exists\;h_{0}\in H,\;\exists\,\,r_{0},a_{1},a_{2}>0$ \quad such
that$\medskip$\\
$r_{0}\left\Vert y\right\Vert _{X^{\ast}}\leq\left(  y,x-h_{0}\right)
+a_{1}\left\vert x\right\vert ^{2}+a_{2},\;\forall\left[  x,y\right]  \in
A.\left(  \footnotemark\right)  $
\end{tabular}
\end{array}
\right.  $\\
and\\
(A$_{2}$)$\left\{
\begin{tabular}
[c]{l}%
$i)\quad u_{0}\in L^{2}\left(  \Omega,\mathcal{F}_{0},\mathbb{P};H\right)
,\;u_{0}\left(  \omega\right)  \in\overline{D\left(  A\right)  }%
,\;\mathbb{P}-$a.s. $\omega\in\Omega$,$\quad\quad\quad\medskip$\\
$ii)\quad f\in L_{ad}^{2}\left(  \Omega;L^{1}\left(  0,T;H\right)  \right)
,\medskip$\\
$iii)\quad M\in\mathcal{M}^{2}\left(  0,T;H\right)  .$%
\end{tabular}
\ \ \ \ \right.  $%
\end{tabular}
\ \ \
\]
\footnotetext{see also Theorem 2.2}We mention that by Theorem \ref{t2.3} if
$M\in L^{0}\left(  \Omega;C\left(  \left[  0,T\right]  ;X\right)  \right)  $
then equation (\ref{3.13}) has a unique (determinist) solution%
\[
u\left(  \omega,\cdot\right)  =GD\left(  A;u_{0}\left(  \omega\right)
,f\left(  \omega,\cdot\right)  ,M\left(  \omega,\cdot\right)  \right)
\;\text{a.s. }\omega\in\Omega.
\]
In the sequel using the martingale properties we shall extend our determinist
results to the case when $M$ has not $X$-valued continuous trajectories. The
corresponding solution will be denoted by $GS\left(  A;u_{0},f,M\right)  .$
Certainly we shall use the determinist result approximating $M\in
\mathcal{M}^{2}\left(  0,T;H\right)  $ by%
\begin{equation}%
\begin{tabular}
[c]{l}%
$\overline{M}_{n}\in L_{ad}^{2}\left(  \Omega;C\left(  \left[  0,T\right]
;X\right)  \right)  \cap\mathcal{M}^{2}\left(  0,T;H\right)  $\medskip\\
$\overline{M}_{n}\rightarrow M$ in $\mathcal{M}^{2}\left(  0,T;H\right)  .$%
\end{tabular}
\ \label{3.14}%
\end{equation}
Remark that putting%
\begin{equation}
\overline{M}_{n}\left(  t\right)  =\sum\limits_{i=1}^{n}\left(  M\left(
t\right)  ,h_{i}\right)  h_{i}, \label{3.15}%
\end{equation}
where $\left\{  h_{i},i\in\mathbb{N}^{\ast}\right\}  \subset X$ is an
orthonormal basis in $H,$then $\overline{M}_{n}$ satisfies (\ref{3.14}).

\begin{definition}
\label{d3.1} A pair of stochastic processes $\left(  u,\eta\right)  $\ is the
generalized stochastic solution of the stochastic evolution equation
(\ref{3.13}), $\left(  u,\eta\right)  =GS\left(  A;u_{0},f,M\right)  ,$\ if
\[%
\begin{tabular}
[c]{l}%
$s_{1})\quad u\in L_{ad}^{0}\left(  \Omega;C\left(  \left[  0,T\right]
;H\right)  \right)  ;\;u\left(  0\right)  =u_{0}$\ a.s. and\medskip\\
\quad\quad\ $u\left(  t\right)  \in\overline{D\left(  A\right)  },\;\forall
t\in\left[  0,T\right]  $\ a.s.,\medskip\\
$s_{2})\quad\eta\in L_{ad}^{0}\left(  \Omega;C\left(  \left[  0,T\right]
;H\right)  \right)  \cap L^{0}\left(  \Omega;BV\left(  0,T;X^{\ast}\right)
\right)  ,$\medskip\\
\quad\quad\ $\eta\left(  0\right)  =0$ a.s.\medskip\\
$s_{3})\quad\displaystyle u\left(  t\right)  +\eta\left(  t\right)
=u_{0}+\int\nolimits_{0}^{t}f\left(  s\right)  ds+M\left(  t\right)
,\;\forall t\in\left[  0,T\right]  ,$ a.s.\medskip\\
$s_{4})\quad$there exists $\overline{M}_{n}$\ a stochastic process satisfying
(\ref{3.14})\medskip\\
\quad\quad\ such that denoting for a.s. \ $\omega\in\Omega:$\medskip\\
$\quad\quad\,(\overline{u}_{n}\left(  \omega,\cdot\right)  ,\overline{\eta
}_{n}\left(  \omega,\cdot\right)  )=GD\left(  A;u_{0}\left(  \omega\right)
,f\left(  \omega,\cdot\right)  ,\overline{M}_{n}\left(  \omega,\cdot\right)
\right)  $\medskip\\
\quad\quad\ then $\;\overline{u}_{n}\rightarrow u,\;\overline{\eta}%
_{n}\rightarrow\eta$\ in\ $L_{ad}^{0}\left(  \Omega;C\left(  \left[
0,T\right]  ;H\right)  \right)  $ as $n\rightarrow\infty$ and\medskip\\
\quad\quad\ $\sup\limits_{n}\mathbb{E}\left\Vert \overline{\eta}%
_{n}\right\Vert _{BV\left(  \left[  0,T\right]  ;X^{\ast}\right)  }%
<+\infty.\;$%
\end{tabular}
\ \
\]

\end{definition}

Remark that the adaptability of the stochastic processes $u$ and $\eta
=u_{0}+\int\nolimits_{0}^{t}f\left(  s\right)  ds+M-u$ is obtained from
$s_{4}) $ and the continuity of the Skorohod mapping%
\[
\left(  u_{0}\left(  \omega\right)  ,f\left(  \omega,\cdot\right)
,\overline{M}_{n}\left(  \omega,\cdot\right)  \right)  \overset{\Gamma
}{\longmapsto}\overline{u}_{n}\left(  \omega,\cdot\right)
\]
(see Theorem \ref{t2.3}). We have not confusion if we denote, also
$u=GS\left(  A;u_{0},f,M\right)  $ since $\eta$ uniquely defined by s$_{3}).$

\begin{proposition}
\label{p3.2}\ a) \ If $\left(  u,\eta\right)  =GS\left(  A;u_{0},f,M\right)
$\ and $\left(  v,\eta\right)  =GS\left(  A;v_{0},g,N\right)  $\ are two
generalized solution of the Cauchy problem (\ref{3.13})\ then%
\begin{equation}%
\begin{tabular}
[c]{l}%
$\displaystyle\left\vert u\left(  t\right)  -v\left(  t\right)  \right\vert
^{2}\leq\left\vert u\left(  s\right)  -v\left(  s\right)  \right\vert
^{2}+2\alpha\int_{s}^{t}\left\vert u-v\right\vert ^{2}d\tau$\medskip\\
$\displaystyle\quad+2\int\nolimits_{s}^{t}\left(  u-v,f-g\right)  d\tau
+2\int\nolimits_{s}^{t}\left(  u\left(  \tau\right)  -v\left(  \tau\right)
,dM\left(  \tau\right)  -dN\left(  \tau\right)  \right)  $\medskip\\
$\displaystyle\quad+\left\langle M-N\right\rangle \left(  t\right)
-\left\langle M-N\right\rangle \left(  s\right)  ,$%
\end{tabular}
\ \ \label{3.16}%
\end{equation}
for all $0\leq s\leq t\leq T,$a.s. $\omega\in\Omega.$

b) \ The equation (\ref{3.13})\ has at most one generalized solution.
\end{proposition}

\begin{proof}
From Definition \ref{d3.1} and the inequality (\ref{2.18}) we have for all
$0\leq s\leq t\leq T:$%
\[%
\begin{tabular}
[c]{l}%
$\left\vert \overline{u}_{n}\left(  \omega,t\right)  -\overline{M}_{n}\left(
\omega,t\right)  -\overline{v}_{n}\left(  \omega,t\right)  -\overline{N}%
_{n}\left(  \omega,t\right)  \right\vert ^{2}\leq$\medskip\\
$\displaystyle\leq\left\vert \overline{u}_{n}\left(  \omega,s\right)
-\overline{M}_{n}\left(  \omega,s\right)  -\overline{v}_{n}\left(
\omega,s\right)  +\overline{N}_{n}\left(  \omega,s\right)  \right\vert
^{2}+2\alpha\int\nolimits_{s}^{t}\left\vert \overline{u}_{n}-\overline{v}%
_{n}\right\vert ^{2}ds$\medskip\\
$\displaystyle\quad+2\int\nolimits_{s}^{t}\left(  \overline{u}_{n}\left(
\omega,\tau\right)  -\overline{M}_{n}\left(  \omega,\tau\right)  -\overline
{v}_{n}\left(  \omega,\tau\right)  +\overline{N}_{n}\left(  \tau\right)
,f\left(  \omega,\tau\right)  -g\left(  \omega,\tau\right)  \right)  d\tau
$\medskip\\
$\displaystyle\quad+2\int\nolimits_{s}^{t}\left\langle \overline{M}_{n}\left(
\omega,\tau\right)  -\overline{N}_{n}\left(  \omega,\tau\right)
,d\overline{\eta}_{n}\left(  \omega,\tau\right)  -d\overline{\zeta}_{n}\left(
\omega,\tau\right)  \right\rangle $ a.s. $\omega\in\Omega.$%
\end{tabular}
\ \ \
\]
The continuity with respect to $t$ and $s$ involved that this inequality holds
for all $0\leq s\leq t\leq T$ and $\omega\in\Omega_{0},\;\mathbb{P}\left(
\Omega_{0}\right)  =1.$ Using Proposition \ref{p3.1} (equality \ref{3.8}) with
$u=\overline{u}_{n}-\overline{v}_{n}$, $m=\overline{M}_{n}-\overline{N}_{n}$,
$f:=f-g$, $\eta=\overline{\eta}_{n}-\overline{\zeta}_{n}$ from this last
inequality we have%
\begin{equation}%
\begin{tabular}
[c]{l}%
$\displaystyle\left\vert \overline{u}_{n}\left(  t\right)  -\overline{v}%
_{n}\left(  t\right)  \right\vert ^{2}\leq\left\vert \overline{u}_{n}\left(
s\right)  -\overline{v}_{n}\left(  s\right)  \right\vert ^{2}+2\alpha
\int\nolimits_{s}^{t}\left\vert \overline{u}_{n}-\overline{v}_{n}\right\vert
^{2}d\tau+$\medskip\\
$\displaystyle\quad+2\int\nolimits_{s}^{t}\left(  \overline{u}_{n}%
-\overline{v}_{n},f-g\right)  d\tau+\left\langle \overline{M}_{n}-\overline
{N}_{n}\right\rangle \left(  t\right)  -\left\langle \overline{M}%
_{n}-\overline{N}_{n}\right\rangle \left(  s\right)  $\medskip\\
$\displaystyle\quad+2\int_{s}^{t}\left(  \overline{u}_{n}\left(  \tau\right)
-\overline{v}_{n}\left(  \tau\right)  ,d\overline{M}_{n}\left(  \tau\right)
-d\overline{N}_{n}\left(  \tau\right)  \right)  ,$%
\end{tabular}
\ \ \ \label{3.17}%
\end{equation}
for all $0\leq s\leq t\leq T;$ a.s. $\omega\in\Omega$, which yields
(\ref{3.16}) passing to limit as $n\rightarrow\infty$. The uniqueness is,
obviously, the consequence of (\ref{3.16}).\hfill
\end{proof}

\begin{theorem}
\label{t3.4}Under the assumptions (A$_{1}$)\ and (A$_{2}$) the initial problem
(\ref{3.13}) has a unique generalized (stochastic) solution $\left(
u,\eta\right)  $, $\left(  u,\eta\right)  =GS\left(  A;u_{0},f,M\right)  .$
Moreover this solution satisfies:%
\begin{equation}%
\begin{tabular}
[c]{l}%
$a)\quad u\in L_{ad}^{2}\left(  \Omega;C\left(  \left[  0,T\right]  ;H\right)
\right)  $\medskip\\
$b)\quad\eta\in L_{ad}^{2}\left(  \Omega;C\left(  \left[  0,T\right]
;H\right)  \right)  \cap L^{1}\left(  \Omega;BV\left(  0,T;X^{\ast}\right)
\right)  $\medskip\\
$c)\quad\mathbb{E}\sup\limits_{t\in\left[  0,T\right]  }\left\vert u\left(
t\right)  \right\vert ^{2}+\mathbb{E}\left\Vert \eta\right\Vert _{BV\left(
\left[  0,T\right]  ;X^{\ast}\right)  }\leq C_{0}[1+\mathbb{E}\left\vert
u_{0}\right\vert ^{2}+$\medskip\\
\quad\quad$\displaystyle+\mathbb{E}\left(  \int\nolimits_{0}^{T}\left\vert
f\left(  s\right)  \right\vert ds\right)  ^{2}+\mathbb{E}\left\vert M\left(
T\right)  \right\vert ^{2}]\left(  \footnotemark\right)  $
\end{tabular}
\ \ \label{3.18}%
\end{equation}
\footnotetext{$C_{0}=C_{0}\left(  T,r_{0},h_{0},a_{1},a_{2}\right)  >0$}and if
$\left(  u,\eta\right)  =GS\left(  A;u_{0},f,M\right)  ,\;\left(
v,\zeta\right)  =GS\left(  A;v_{0},g,N\right)  $\ then for all $t\in\left[
0,T\right]  $%
\begin{equation}%
\begin{tabular}
[c]{r}%
$\displaystyle\mathbb{E}\sup\limits_{s\in\left[  0,t\right]  }\left\vert
u\left(  s\right)  -v\left(  s\right)  \right\vert ^{2}\leq C\left(
\alpha,T\right)  [\mathbb{E}\left\vert u_{0}-v_{0}\right\vert ^{2}%
+\mathbb{E}\left(  \int_{0}^{t}\left\vert f-g\right\vert ds\right)  ^{2}%
$\medskip\\
$+\mathbb{E}\left\vert M\left(  t\right)  -N\left(  t\right)  \right\vert
^{2}$%
\end{tabular}
\label{3.19}%
\end{equation}

\end{theorem}

\begin{proof}
The inequality (\ref{3.19}) is obtained easily from (\ref{3.16}). Let $\left[
x,y\right]  \in A.$ Then $v_{n}\left(  t\right)  =x$,\ $\zeta_{n}\left(
t\right)  =yt$ is a generalized (determinist) solution corresponding to
$u_{0}=x$, $f\left(  t\right)  =y$ and $M=0.$ Let $\tau_{n,R}\left(
\omega\right)  =\inf\left\{  t\in\left[  0,T\right]  :\left\vert \overline
{u}_{n}\left(  \omega,t\right)  \right\vert \geq R\right\}  ,$ and $\tau
_{n,R}\left(  \omega\right)  =T$ if the set from inf is empty; $\tau_{n,R}$ is
a stopping time since $\overline{u}_{n}\in L_{ad}^{0}\left(  \Omega;C\left(
\left[  0,T\right]  ;H\right)  \right)  .$ Substituting in (\ref{3.17})
$t=\tau_{n,R}\left(  \omega\right)  $ and using the properties of the
stochastic integral, by elementary calculus we obtain:%
\[%
\begin{array}
[c]{r}%
\displaystyle\mathbb{E}\sup\limits_{t\in\left[  0,T\right]  }\left\vert
\overline{u}_{n}\left(  t\wedge\tau_{n,R}\right)  -x\right\vert ^{2}\leq
C\left(  \alpha,T\right)  [\mathbb{E}\left\vert u_{0}-x\right\vert
^{2}+\mathbb{E}\left\Vert f-y\right\Vert _{L^{1}\left(  0,T;H\right)  }%
^{2}\medskip\\
+\mathbb{E}\left\vert \overline{M}_{n}\left(  T\right)  \right\vert
^{2}],\quad\forall\left[  x,y\right]  \in A.
\end{array}
\]
Which implies for $R\nearrow+\infty$ that%
\begin{equation}%
\begin{tabular}
[c]{r}%
$\mathbb{E}\sup\limits_{t\in\left[  0,T\right]  }\left\vert \overline{u}%
_{n}\left(  t\right)  -x\right\vert ^{2}\leq C\left(  \alpha,T\right)
[\mathbb{E}\left\vert u_{0}-x\right\vert ^{2}+\mathbb{E}\left\Vert
f-y\right\Vert _{L^{1}\left(  0,T;H\right)  }^{2}+\medskip$\\
$+\mathbb{E}\left\vert \overline{M}_{n}\left(  T\right)  \right\vert
^{2}],\quad\forall\left[  x,y\right]  \in A.$%
\end{tabular}
\label{3.20}%
\end{equation}
Hence\ $\overline{u}_{n}\in L_{ad}^{2}\left(  \Omega;C\left(  \left[
0,T\right]  ;H\right)  \right)  $ and also%
\[
\overline{\eta}_{n}=u_{0}+\int\nolimits_{0}^{\cdot}fds+\overline{M}%
_{n}-\overline{u}_{n}\in L_{ad}^{2}\left(  \Omega;C\left(  \left[  0,T\right]
;H\right)  \right)
\]
Now from (\ref{3.17}) for $v_{n}=u_{n+k}$ we have%
\[%
\begin{tabular}
[c]{l}%
$\displaystyle\left\vert \overline{u}_{n}\left(  t\right)  -\overline{u}%
_{n+k}\left(  t\right)  \right\vert ^{2}\leq2\left\vert \alpha\right\vert
\int\nolimits_{0}^{t}\left\vert \overline{u}_{n}\left(  s\right)
-\overline{u}_{n+k}\left(  s\right)  \right\vert ^{2}ds+\medskip$\\
$\displaystyle\quad+2\int\nolimits_{0}^{t}\left(  \overline{u}_{n}\left(
s\right)  -\overline{u}_{n+k}\left(  s\right)  ,d\overline{M}_{n}\left(
s\right)  -d\overline{M}_{n+k}\left(  s\right)  \right)  +\left\langle
\overline{M}_{n}-\overline{M}_{n+k}\right\rangle \left(  t\right)  $%
\end{tabular}
\ \ \
\]
for all $t\in\left[  0,T\right]  ,$ a.s. $\omega\in\Omega,$ which implies by
(\ref{3.4}) and Gronwall's inequality:%
\[
\mathbb{E}\sup\limits_{t\in\left[  0,T\right]  }\left\vert \overline{u}%
_{n}\left(  t\right)  -\overline{u}_{n+k}\left(  t\right)  \right\vert
^{2}\leq C\left(  \alpha,T\right)  \mathbb{E}\left\vert \overline{M}%
_{n}\left(  T\right)  -\overline{M}_{n+k}\left(  T\right)  \right\vert ^{2}.
\]
Hence $\exists\,u,\eta\in L_{ad}^{2}\left(  \Omega;C\left(  \left[
0,T\right]  ;H\right)  \right)  $ such that as $n\rightarrow\infty:$%
\begin{equation}%
\begin{tabular}
[c]{l}%
$\overline{u}_{n}\rightarrow u$\medskip\\
$\displaystyle\overline{\eta}_{n}=u_{0}+\int\nolimits_{0}^{\cdot}%
fds+\overline{M}_{n}-\overline{u}_{n}\rightarrow\eta=u_{0}+\int\nolimits_{0}%
^{\cdot}f+M-u$\medskip\\
in $\;L_{ad}^{2}\left(  \Omega;C\left(  \left[  0,T\right]  ;H\right)
\right)  $ \ and\medskip\\
$\mathbb{E}\sup\limits_{t\in\left[  0,T\right]  }\left\vert \overline{u}%
_{n}\left(  t\right)  -u\left(  t\right)  \right\vert ^{2}\leq C\left(
\alpha,T\right)  \mathbb{E}\left\vert \overline{M}_{n}\left(  T\right)
-M\left(  T\right)  \right\vert ^{2}$%
\end{tabular}
\label{3.21}%
\end{equation}
and
\begin{equation}%
\begin{tabular}
[c]{r}%
$\mathbb{E}(\sup\limits_{t\in\left[  0,T\right]  }\left\vert u\left(
t\right)  -x\right\vert ^{2})\leq C\left(  \alpha,T\right)  [\mathbb{E}%
\left\vert u_{0}-x\right\vert ^{2}+\mathbb{E}\left\Vert f-y\right\Vert
_{L^{1}\left(  0,T;H\right)  }^{2}+$\medskip\\
$\quad\quad\quad\quad\quad\quad\quad\quad\quad+\mathbb{E}\left\vert M\left(
T\right)  \right\vert ^{2}]\quad\quad\forall\left[  x,y\right]  \in A$%
\end{tabular}
\label{3.22}%
\end{equation}
From (\ref{3.16}), (\ref{3.17}) follows by a standard calculus (using
(\ref{3.4}-c) and Gronwall's inequality) . Of course $u\left(  t\right)
\in\overline{D\left(  A\right)  }$,\ $\forall t\in\left[  0,T\right]  $, a.s.
$u\left(  0\right)  =u_{0}$, $\eta\left(  0\right)  =0$ since $\overline
{u}_{n},\overline{\eta}_{n}$ satisfy these conditions.

From Remark \ref{r2.2} we have a.s.$\,\omega\in\Omega:$%
\[%
\begin{array}
[c]{l}%
\left\vert \overline{u}_{n}\left(  t\right)  -\overline{M}_{n}\left(
t\right)  -h_{0}\right\vert ^{2}+2r_{0}\left\Vert \overline{\eta}%
_{n}\right\Vert _{BV\left(  \left[  0,T\right]  ;X^{\ast}\right)  }%
\leq\left\vert u_{0}-h_{0}\right\vert ^{2}+\medskip\\
\displaystyle\quad+2a_{1}\int\nolimits_{0}^{t}\left\vert \overline{u}%
_{n}\left(  s\right)  \right\vert ^{2}ds+2a_{2}t+2\int\nolimits_{0}^{t}\left(
f\left(  s\right)  ,\overline{u}_{n}\left(  s\right)  -\overline{M}_{n}\left(
s\right)  -h_{0}\right)  ds+\medskip\\
\displaystyle\quad+2\int\nolimits_{0}^{t}\left(  \overline{M}_{n}\left(
s\right)  ,d\overline{\eta}_{n}\left(  s\right)  \right)
\end{array}
\]
But by Proposition \ref{p3.1}-a)%
\[%
\begin{array}
[c]{l}%
\displaystyle2\int\limits_{0}^{t}\left(  \overline{M}_{n}\left(  s\right)
,d\overline{\eta}_{n}\left(  s\right)  \right)  ds=2\left(  \overline{M}%
_{n}\left(  t\right)  ,\overline{\eta}_{n}\left(  t\right)  \right)
-2\int\limits_{0}^{t}\left(  \overline{\eta}_{n}\left(  s\right)
,d\overline{M}_{n}\left(  s\right)  \right)  \medskip\\
\displaystyle=2\left(  \overline{M}_{n}\left(  t\right)  ,u_{0}+\int%
\limits_{0}^{t}f\left(  s\right)  ds+\overline{M}_{n}\left(  t\right)
-\overline{u}_{n}\left(  t\right)  \right)  -2\int\limits_{0}^{t}(u_{0}%
+\int\nolimits_{0}^{s}f\left(  \tau\right)  d\tau+\medskip\\
\displaystyle\quad+\overline{M}_{n}\left(  s\right)  -u_{n}\left(  s\right)
,d\overline{M}_{n}\left(  s\right)  )\medskip\\
\displaystyle=2\int\limits_{0}^{t}\left(  f\left(  s\right)  ,\overline{M}%
_{n}\left(  s\right)  \right)  ds+\left\vert \overline{M}_{n}\left(  t\right)
\right\vert ^{2}+\left\langle \overline{M}_{n}\right\rangle \left(  t\right)
-2\left(  \overline{M}_{n}\left(  t\right)  ,\overline{u}_{n}\left(  t\right)
\right)  \medskip\\
\displaystyle\quad+2\int\limits_{0}^{t}\left(  u_{n}\left(  s\right)
,d\overline{M}_{n}\left(  s\right)  \right)  .
\end{array}
\]
Hence%
\[%
\begin{array}
[c]{l}%
\left\vert \overline{u}_{n}\left(  t\right)  \right\vert ^{2}+2r_{0}\left\Vert
\overline{\eta}_{n}\right\Vert _{BV\left(  \left[  0,t\right]  ;X^{\ast
}\right)  }\leq2\left\vert u_{0}\right\vert ^{2}+10\left\vert h_{0}\right\vert
^{2}+2\left\vert a_{2}\right\vert T+\medskip\\
\displaystyle+9\left(  \int\nolimits_{0}^{T}\left\vert f\left(  s\right)
\right\vert ds\right)  ^{2}+\left\langle \overline{M}_{n}\right\rangle \left(
T\right)  +\frac{1}{4}\sup\limits_{s\in\left[  0,t\right]  }\left\vert
\overline{u}_{n}\left(  s\right)  \right\vert ^{2}+2\left\vert a_{1}%
\right\vert \int\nolimits_{0}^{t}\left\vert \overline{u}_{n}\left(  s\right)
\right\vert ^{2}ds+\medskip\\
\displaystyle+2\sup\limits_{s\in\left[  0,t\right]  }\left\vert \int%
\nolimits_{0}^{s}\left(  \overline{u}_{n}\left(  \tau\right)  ,d\overline
{M}_{n}\left(  \tau\right)  \right)  d\tau\right\vert
\end{array}
\]
and by (\ref{3.4}-c) we conclude: there exists a positive constant
$C_{0}=C_{0}\left(  T,r_{0},h_{0},a_{1},a_{2}\right)  $ such that
\begin{equation}%
\begin{tabular}
[c]{l}%
$\mathbb{E}\sup\limits_{t\in\left[  0,T\right]  }\left\vert \overline{u}%
_{n}\left(  t\right)  \right\vert ^{2}+\mathbb{E}(\left\Vert \overline{\eta
}_{n}\right\Vert _{BV\left(  \left[  0,T\right]  ;X^{\ast}\right)  }%
)\leq\medskip$\\
$\leq C_{0}[1+\mathbb{E}\left\vert u_{0}\right\vert ^{2}+\mathbb{E}\left\Vert
f\right\Vert _{L^{1}\left(  0,T;H\right)  }^{2}+\mathbb{E}\left\vert
\overline{M}_{n}\left(  T\right)  \right\vert ]$%
\end{tabular}
\ \ \ \label{3.23}%
\end{equation}
Finally the inequality (\ref{3.23}) and the following lemma complete the proof
of Theorem \ref{t3.4}.\hfill
\end{proof}

\begin{lemma}
\label{l3.5} If%
\[%
\begin{tabular}
[c]{l}%
$g,g_{k}\in L^{1}\left(  \Omega;C\left(  \left[  0,T\right]  ;H\right)
\right)  ,\;k\in\mathbb{N}^{\ast}\medskip$\\
$\mathbb{E}\left\Vert g_{k}\right\Vert _{BV\left(  \left[  0,T\right]
;X^{\ast}\right)  }\leq D\equiv$const,$\;\forall k\in\mathbb{N}^{\ast}%
\medskip$\\
$g_{k}\rightarrow g$\ in $C\left(  \left[  0,T\right]  ;H\right)
,\;$a.s.\ $\omega\in\Omega,$%
\end{tabular}
\]
then%
\[%
\begin{tabular}
[c]{l}%
$g\in L^{1}\left(  \Omega;BV\left(  \left[  0,T\right]  ;X^{\ast}\right)
\right)  ,\;$and$\medskip$\\
$\mathbb{E}\left\Vert g\right\Vert _{BV\left(  \left[  0,T\right]  ;X^{\ast
}\right)  }\leq D.$%
\end{tabular}
\
\]

\end{lemma}

\begin{proof}
Let $\;\Delta_{N}:0=t_{0}^{\left(  N\right)  }<t_{1}^{\left(  N\right)
}<...<t_{k_{N}}^{\left(  N\right)  }=T$ with
\[%
\begin{tabular}
[c]{l}%
$\nu\left(  \Delta_{N}\right)  =\max\limits_{i}\left\vert t_{i+1}^{\left(
N\right)  }-t_{i}^{\left(  N\right)  }\right\vert \rightarrow0,\;$as
$N\rightarrow\infty\medskip$\\
$S\left(  g;\Delta_{N}\right)  =\sum\limits_{i=0}^{k_{N}-1}\left\Vert
g(t_{i+1}^{\left(  N\right)  })-g(t_{i}^{\left(  N\right)  }\right\Vert
_{X^{\ast}}\nearrow\left\Vert g\right\Vert _{BV\left(  \left[  0,T\right]
;X^{\ast}\right)  }$%
\end{tabular}
\ \
\]
Then we have%
\[
\mathbb{E}S\left(  g_{k};\Delta_{N}\right)  \leq\mathbb{E}\left\Vert
g_{k}\right\Vert _{BV\left(  \left[  0,T\right]  ;X^{\ast}\right)  }\leq
D,\;\forall k\in\mathbb{N}^{\ast},
\]
Passing to $\liminf_{k\rightarrow\infty},$ by Fatou Lemma, we obtain
\ $\mathbb{E}S\left(  g;\Delta_{N}\right)  \leq D.$

Now the monotone convergence theorem (Beppv-L\'{e}vy Lemma) as $N\rightarrow
\infty$ yields:
\[
\mathbb{E}\left\Vert g\right\Vert _{BV\left(  \left[  0,T\right]  ;X^{\ast
}\right)  }\leq D.
\]
\hfill
\end{proof}

\begin{corollary}
\label{c3.6}Under the assumptions of Theorem \ref{t3.4} if $\left(
u_{\varepsilon},\eta_{\varepsilon}\right)  $ is the solution of the
approximating problem
\begin{equation}%
\begin{tabular}
[c]{l}%
$du_{\varepsilon}\left(  t\right)  +\left(  A_{\varepsilon}^{\alpha
}u_{\varepsilon}\left(  t\right)  -\alpha u_{\varepsilon}\left(  t\right)
\right)  dt=f\left(  t\right)  dt+dM\left(  t\right)  \medskip$\\
$u_{\varepsilon}\left(  0\right)  =u_{0}\medskip$\\
$\displaystyle\eta_{\varepsilon}\left(  t\right)  =\int\nolimits_{0}%
^{t}\left(  A_{\varepsilon}^{\alpha}\left(  u_{\varepsilon}\left(  s\right)
\right)  -\alpha u_{\varepsilon}\left(  s\right)  \right)  ds$%
\end{tabular}
\ \label{3.24}%
\end{equation}
where $A_{\varepsilon}^{\alpha}$ is the Yosida approximation of the maximal
monotone operator $A+\alpha I$ and $0<\varepsilon<\frac{1}{\left\vert
\alpha\right\vert +1},$then there exists a constant $C_{0}=C_{0}\left(
T,r_{0},h_{0},a_{1},a_{2}\right)  >0$\ such that%
\begin{equation}%
\begin{tabular}
[c]{l}%
$\mathbb{E}\sup\limits_{t\in\left[  0,T\right]  }\left\vert u_{\varepsilon
}\left(  t\right)  \right\vert ^{2}+\mathbb{E}\left\Vert \eta_{\varepsilon
}\right\Vert _{BV\left(  \left[  0,T\right]  ;X^{\ast}\right)  }\leq
C_{0}(1+\mathbb{E}\left\vert u_{0}\right\vert ^{2}+\mathbb{E}\left\Vert
f\right\Vert _{L^{1}\left(  0,T;H\right)  }+\medskip$\\
$\quad\quad\quad\quad\quad\quad\quad\quad\quad\quad\quad\quad\quad\quad
\quad\quad\quad\quad+\mathbb{E}\left\vert M\left(  T\right)  \right\vert
^{2})$%
\end{tabular}
\label{3.25}%
\end{equation}
and \ $\lim\limits_{\varepsilon\searrow0}u_{\varepsilon}=u,\;\lim
\limits_{\varepsilon\searrow0}\eta_{\varepsilon}=\eta$ \ in $L_{ad}^{2}\left(
\Omega;C\left(  \left[  0,T\right]  ;H\right)  \right)  .$
\end{corollary}

\begin{proof}
As we could see in the proof of Corollary \ref{c2.4}%
\begin{equation}
r_{0}\left\Vert (A_{\varepsilon}^{\alpha}-\alpha I)x\right\Vert _{X^{\ast}%
}\leq\left(  \left(  A_{\varepsilon}^{\alpha}-\alpha I\right)  x,x-h_{0}%
\right)  +b_{1}\left\vert x\right\vert ^{2}+b_{2} \label{3.26}%
\end{equation}
with $b_{i}=b_{i}\left(  \alpha,h_{0},a_{1},a_{2}\right)  >0.$ By Energy
Equality for $u_{\varepsilon}-h_{0}$ we have:%
\[%
\begin{array}
[c]{l}%
\displaystyle\left\vert u_{\varepsilon}\left(  t\right)  -h_{0}\right\vert
^{2}+2\int\nolimits_{0}^{t}\left(  u_{\varepsilon}\left(  s\right)
-h_{0},d\eta_{\varepsilon}\left(  s\right)  \right)  =\left\vert u_{0}%
-h_{0}\right\vert ^{2}+\medskip\\
\displaystyle\quad+2\int\nolimits_{0}^{t}\left(  f\left(  s\right)
,u_{\varepsilon}\left(  s\right)  -h_{0}\right)  ds+2\int\nolimits_{0}%
^{t}\left(  u_{\varepsilon}\left(  s\right)  -h_{0},dM\left(  s\right)
\right)  +\left\langle M\right\rangle \left(  t\right)
\end{array}
\]
which implies, by (\ref{3.26}) and a standard calculus the inequality
(\ref{3.25}). If we approximate the martingale $M$ by $\overline{M}_{n}$ as in
Theorem \ref{t3.4} then for the corresponding solution $\overline
{u}_{\varepsilon n},\overline{\eta}_{\varepsilon n}$ of (\ref{3.25}) the
inequality (\ref{3.25}) is fulfilled . We know, by Corollary \ref{c2.4}, that
for every $n$ fixed $\;y_{\varepsilon,n}\left(  \omega\right)  =\sup
\limits_{t\in\left[  0,T\right]  }\left\vert \overline{u}_{\varepsilon
n}\left(  \omega,t\right)  -\overline{u}_{n}\left(  \omega,t\right)
\right\vert $ satisfies \ $\lim\limits_{\varepsilon\searrow0}y_{\varepsilon
,n}\left(  \omega\right)  =0$ \ a.s. $\omega\in\Omega$, and by Proposition
\ref{p2.1} and the proof of Corollary \ref{c2.4} $\left\vert y_{\varepsilon
,n}\left(  \omega\right)  \right\vert ^{2}\leq D_{n}\left(  \omega\right)
,\;\mathbb{E}D_{n}<\infty$. Hence $\exists\lim\limits_{\varepsilon\searrow
0}u_{\varepsilon,n}=\overline{u}_{n}$ in $L_{ad}^{2}\left(  \Omega;C\left(
\left[  0,T\right]  ;H\right)  \right)  .$ Also by Ito's formula we have%
\[%
\begin{array}
[c]{l}%
\displaystyle\left\vert u_{\varepsilon}\left(  t\right)  -\overline
{u}_{\varepsilon n}\left(  t\right)  \right\vert ^{2}\leq2\left\vert
\alpha\right\vert \int\nolimits_{0}^{t}\left\vert u_{\varepsilon}\left(
s\right)  -\overline{u}_{\varepsilon n}\left(  s\right)  \right\vert
^{2}ds+\medskip\\
\displaystyle\quad+2\int\nolimits_{0}^{t}\left(  u_{\varepsilon}\left(
s\right)  -\overline{u}_{\varepsilon,n}\left(  s\right)  ,dM\left(  s\right)
-\overline{M}_{n}\left(  s\right)  \right)  +\left\langle M-\overline{M}%
_{n}\right\rangle \left(  t\right)  ,
\end{array}
\]
which implies%
\[
\mathbb{E}\sup\limits_{t\in\left[  0,T\right]  }\left\vert u_{\varepsilon
}\left(  t\right)  -\overline{u}_{\varepsilon,n}\left(  t\right)  \right\vert
^{2}\leq C\left(  \alpha,T\right)  \mathbb{E}\left\vert M\left(  T\right)
-\overline{M}_{n}\left(  T\right)  \right\vert ^{2}%
\]
Finally since%
\[
\left\vert u_{\varepsilon}-u\right\vert ^{2}\leq3\left(  \left\vert
u_{\varepsilon}-\overline{u}_{\varepsilon n}\right\vert ^{2}+\left\vert
\overline{u}_{\varepsilon n}-\overline{u}_{n}\right\vert ^{2}+\left\vert
\overline{u}_{n}-u\right\vert ^{2}\right)
\]
then%
\[
\limsup_{\varepsilon\searrow0}\mathbb{E}\sup\limits_{t\in\left[  0,T\right]
}\left\vert u_{\varepsilon}\left(  t\right)  -u\left(  t\right)  \right\vert
^{2}\leq C_{1}\left(  r,\alpha,T\right)  \left[  \mathbb{E}\left\vert M\left(
t\right)  -\overline{M}_{n}\left(  t\right)  \right\vert ^{2}\right]
\]
for all $n\in N^{\ast},$ that is $\exists\lim\limits_{\varepsilon\searrow
0}u_{\varepsilon}=u$ and $\exists\lim\limits_{\varepsilon\searrow0}%
\eta_{\varepsilon}=\lim\limits_{\varepsilon\searrow0}\left(  u_{0}%
+\int\nolimits_{0}^{\cdot}fds+M-u_{\varepsilon}\right)  =\eta$ in $L_{ad}%
^{2}\left(  \Omega;C\left(  \left[  0,T\right]  ;H\right)  \right)  .$\hfill
\end{proof}

\subsection{Monotone SDE with state depending diffusion}

We shall work in the context of the spaces $X\subset H\subset X^{\ast}$
introduced in Subsection 2.1, and in the context of the stochastic elements
defined in Subsection 3.1 as the stochastic basis $(\Omega,\mathcal{F}%
,\mathbb{P},(\mathcal{F}_{t})_{t\geq0})$, the $U$-Hilbert valued Wiener
process $\left\{  W\left(  t\right)  ,t\geq0\right\}  $ with the covariance
operator $Q\in L\left(  U,U\right)  $, the stochastic integral etc. Consider
the multivalued stochastic differential equation:%
\begin{equation}
\left\{
\begin{tabular}
[c]{l}%
$du\left(  t\right)  +Au\left(  t\right)  dt\ni f\left(  t,u\left(  t\right)
\right)  dt+B\left(  t,u\left(  t\right)  \right)  dW\left(  t\right)
$\medskip\\
$u\left(  0\right)  =u_{0},\quad\quad t\in\left[  0,T\right]  ,$%
\end{tabular}
\ \right.  \label{3.27}%
\end{equation}
where we put the assumptions:%

\[
\left(  \widetilde{H}_{0}\right)
\begin{tabular}
[c]{l}%
$\quad u_{0}\in L^{2}\left(  \Omega,\mathcal{F}_{0},\mathbb{P};\overline
{D\left(  A\right)  }\right)  $%
\end{tabular}
\
\]%
\[
\left(  \widetilde{H}_{1}\right)  \left\{
\begin{tabular}
[c]{l}%
$i)\quad A:H\rightarrow2^{H}$ is a maximal monotone operator,\medskip\\
$ii)\quad\exists r_{0},a_{1},a_{2}>0,\;\exists h_{0}\in H$ such that\medskip\\
\quad\quad\ $r_{0}\left\Vert y\right\Vert _{\ast}\leq\left(  y,x-h_{0}\right)
+a_{1}\left\vert x\right\vert ^{2}+a_{2}$ \ for all $\forall\left[
x,y\right]  \in A\;\left(  \footnotemark\right)  $
\end{tabular}
\ \ \right.
\]
\footnotetext{see Theorem 2.3}%
\[
\left(  \widetilde{H}_{2}\right)  \left\{
\begin{tabular}
[c]{l}%
$i)\,f:\Omega\times\left[  0,T\right]  \times H\longmapsto H$ is progressively
measurable i.e.\medskip\\
\thinspace\quad\thinspace$\forall\,t\in\left[  0,T\right]  ,\;\;f\cdot
1_{\left[  0,t\right]  }$ is $\left(  \mathcal{F}_{t}\otimes\mathcal{B}%
_{\left[  0,t\right]  }\otimes\mathcal{B}_{H},\mathcal{B}_{H}\right)
$-measurable,\medskip\\
$ii)\,\exists\,L_{1},b_{1}>0$ such that a.s. $\omega\in\Omega:$\medskip\\
\quad\quad\thinspace$\left\vert f(t,u)-f\left(  t,v\right)  \right\vert \leq
L_{1}\left\vert u-v\right\vert ,$\medskip\\
\quad\quad\thinspace$\left\vert f\left(  t,v\right)  \right\vert ^{2}\leq
b_{1}\left(  1+\left\vert u\right\vert ^{2}\right)  ,\;$for all $u,v\in H,$
a.e. $t\in\left]  0,T\right[  .$%
\end{tabular}
\ \ \right.
\]%
\[
\left(  \widetilde{H}_{3}\right)  \left\{
\begin{tabular}
[c]{l}%
$i)\,B:\Omega\times\left[  0,T\right]  \times H\rightarrow L^{2}\left(
U_{0},H\right)  $ is progressively measurable,\medskip\\
$ii)\,\exists L,b>0$ such that a.s. $\omega\in\Omega:$\medskip\\
\quad\quad$\left\vert B\left(  t,u\right)  -B\left(  t,v\right)  \right\vert
_{Q}^{2}\leq L\left\vert u-v\right\vert ^{2}$\medskip\\
\quad\quad$\left\vert B\left(  t,u\right)  \right\vert _{Q}^{2}\leq b\left(
1+\left\vert u\right\vert ^{2}\right)  $ \quad for all $u,v\in H,$ a.e.
$t\in\left]  0,T\right[  $%
\end{tabular}
\ \ \right.
\]
We shall try to find a solution $u\in L_{ad}^{2}\left(  \Omega;C\left(
\left[  0,T\right]  ;H\right)  \right)  $ for the equation (\ref{3.27}). First
we remark that for a such stochastic process $u$ we have:%
\begin{align*}
f\left(  \cdot,u\left(  \cdot\right)  \right)   &  \in L_{ad}^{2}\left(
\Omega\times\left]  0,T\right[  ;H\right)  ,\\
B\left(  \cdot,u\left(  \cdot\right)  \right)   &  \in L_{ad}^{2}\left(
\Omega\times\left]  0,T\right[  ;\mathcal{L}^{2}\left(  U_{0},H\right)
\right)
\end{align*}
and if we denote%
\[
F\left(  t;u\right)  =\int_{0}^{t}f\left(  s,u\left(  s\right)  \right)
ds\quad and\quad M\left(  t;u\right)  =\int_{0}^{t}B\left(  s,u\left(
s\right)  \right)  dW\left(  s\right)
\]
then%
\[
F\left(  \cdot;u\right)  \in L_{ad}^{2}\left(  \Omega;C\left(  \left[
0,T\right]  ;H\right)  \right)  \quad and\quad M\left(  \cdot;u\right)  \in
M^{2}\left(  0,T;H\right)  .
\]

\begin{definition}
\label{d3.2}A stochastic process $u$\ is a (generalized) solution of
multivalued SDE (\ref{3.27}) if%
\begin{equation}%
\begin{tabular}
[c]{l}%
$a)\quad u\in L_{ad}^{2}\left(  \Omega;C\left(  \left[  0,T\right]  ;H\right)
\right)  ,\;\;u\left(  0\right)  =u_{0},$\medskip\\
\quad\quad\quad\quad$\quad u\left(  t\right)  $ $\in\overline{D\left(
A\right)  },\;\forall\,t\in\left[  0,T\right]  ,$ a.s. $\omega\in\Omega
,$\medskip\\
$b)\quad\eta=u_{0}+F\left(  \cdot,u\right)  +M\left(  \cdot,u\right)  -u\in
L^{1}\left(  \Omega;BV\left(  \left[  0,T\right]  ;X^{\ast}\right)  \right)
$\medskip\\
$c)\quad\mathbb{E}\left\vert u\left(  t\right)  -z\left(  t\right)
\right\vert ^{2}\leq\mathbb{E}\left\vert u\left(  s\right)  -z\left(
s\right)  \right\vert ^{2}+$\medskip\\
$\displaystyle+2\mathbb{E}\int\nolimits_{s}^{t}\left(  f\left(  \tau,u\right)
-g\left(  \tau\right)  ,u\left(  \tau\right)  -z\left(  \tau\right)  \right)
d\tau+\mathbb{E}\int\nolimits_{s}^{t}\left\vert B\left(  \tau,u\right)
-D\left(  \tau\right)  \right\vert _{Q}^{2}d\tau,$\medskip\\
$\;\quad$for all\ \ $0\leq s\leq t\leq T,\;\forall\,g\in L_{ad}^{2}\left(
\Omega\times\left]  0,T\right[  ;H\right)  ,$\medskip\\
$\;\;\forall\,D\in L_{ad}^{2}\left(  \Omega\times\left]  0,T\right[
;\mathcal{L}^{2}\left(  U_{0},H\right)  \right)  ,\;\forall\;z\in L_{ad}%
^{2}\left(  \Omega;C\left(  \left[  0,T\right]  ;\overline{D\left(  A\right)
}\right)  \right)  $\medskip\\
\quad\quad$\displaystyle$such that $z=GS\left(  A;z\left(  0\right)
,g,\int\nolimits_{0}^{\cdot}D\left(  s\right)  dW\left(  s\right)  \right)  .$%
\end{tabular}
\ \label{3.28}%
\end{equation}

\end{definition}

The existence and uniqueness result will be obtained from the following two propositions

\begin{proposition}
\label{p3.7}Let ($\widetilde{H}_{0}$-$\widetilde{H}_{3}$) be satisfied. Then
the following problem has a unique solution:%
\begin{equation}%
\begin{tabular}
[c]{l}%
$a)\quad u\in L_{ad}^{2}\left(  \Omega;C\left(  \left[  0,T\right]  ;H\right)
\right)  ,\;\quad\quad u\left(  0\right)  =u_{0},$\ \medskip\\
\quad\quad\quad$u\left(  t\right)  \in\overline{D\left(  A\right)  },\;\forall
t\in\left[  0,T\right]  ,$ a.s. $\omega\in\Omega,$\medskip\\
$b)\quad\eta=u_{0}+F\left(  \cdot;u\right)  +M\left(  \cdot;u\right)  -u\in
L^{1}\left(  \Omega;BV\left(  \left[  0,T\right]  ;X^{\ast}\right)  \right)
,$\medskip\\
$c)\quad u=GS\left(  u_{0},f\left(  \cdot,u\right)  ,M\left(  \cdot;u\right)
\right)  .$%
\end{tabular}
\ \label{3.29}%
\end{equation}

\end{proposition}

\begin{proof}
Under the assumptions ($\widetilde{H}_{0}$-$\widetilde{H}_{3}$) it follows
that for every $v$ from the space $L_{ad}^{2}\left(  \Omega;C\left(  \left[
0,T\right]  ;H\right)  \right)  $ the triplet $\left(  u_{0},f\left(
\cdot,v\right)  ,M\left(  \cdot;v\right)  \right)  $ satisfies the hypotheses
(A$_{2}$) of Theorem \ref{t3.4}. Hence there exists a corresponding unique
generalized solution $\left(  \overline{v},\overline{\eta}\right)  .$ This
solution satisfies:%
\[%
\begin{tabular}
[c]{l}%
$\overline{v}\in L_{ad}^{2}\left(  \Omega;C\left(  \left[  0,T\right]
;H\right)  \right)  ,\;\overline{v}\left(  0\right)  =u_{0},$\medskip\\
$\overline{v}\left(  t\right)  \in\overline{D\left(  A\right)  },\;\forall
t\in\left[  0,T\right]  ,$ a.s. $\omega\in\Omega,$\medskip\\
$\overline{\eta}=u_{0}+F\left(  \cdot;v\right)  +M\left(  \cdot;v\right)
-\overline{u}\in L_{ad}^{2}\left(  \Omega;C\left(  \left[  0,T\right]
;H\right)  \right)  \cap$\medskip\\
$\quad\quad\quad\quad\quad\quad\quad\quad\quad\quad\quad\quad\quad\quad\quad
L^{1}\left(  \Omega;BV\left(  \left[  0,T\right]  ;X^{\ast}\right)  \right)
.$%
\end{tabular}
\ \
\]
Denote the mapping \ $\Lambda:L_{ad}^{2}\left(  \Omega;C\left(  \left[
0,T\right]  ;H\right)  \right)  \rightarrow L_{ad}^{2}\left(  \Omega;C\left(
\left[  0,T\right]  ;H\right)  \right)  $, $\Lambda\left(  v\right)
=\overline{v}$. We show that $\Lambda$ has a unique fix point. Let $a>0.$ For
$v\in L_{ad}^{2}\left(  \Omega;C\left(  \left[  0,T\right]  ;H\right)
\right)  $ denote \ $\left\Vert v\right\Vert _{t}=\left(  \mathbb{E}%
\sup\limits_{s\in\left[  0,t\right]  }\left\vert v\left(  s\right)
\right\vert ^{2}\right)  ^{1/2}$ and \ $|||v|||_{a}=\sup\limits_{t\in\left[
0,T\right]  }\left(  e^{-at}\left\Vert v\right\Vert _{t}\right)  .$ The usual
norm, $\left\Vert \,\cdot\,\right\Vert _{T},$ on $L_{ad}^{2}\left(
\Omega;C\left(  \left[  0,T\right]  ;H\right)  \right)  $ is equivalent to
$|||\,\cdot\,|||_{a}$ since \ $e^{-aT}\left\Vert v\right\Vert _{T}%
\leq|||v|||_{a}\leq\left\Vert v\right\Vert _{T}.$ From Theorem \ref{t3.4} (the
inequality (\ref{3.19})) we have:%
\begin{align*}
\left\Vert \overline{v}_{1}-\overline{v}_{2}\right\Vert _{t}^{2}  &  \leq
C\left(  \alpha,T\right)  [\mathbb{E}\int_{0}^{t}\left\vert f\left(
v_{1}\right)  -f\left(  v_{2}\right)  \right\vert ^{2}ds\\
&  \;\;\;\;\;+\mathbb{E}\int_{0}^{t}\left\vert B\left(  v_{1}\right)
-B\left(  v_{2}\right)  \right\vert _{Q}^{2}ds]\\
&  \leq C(L_{1}^{2}+L)\int_{0}^{t}\left\Vert v_{1}-v_{2}\right\Vert _{s}%
^{2}ds\\
&  =C_{1}\int_{0}^{t}e^{-2as}\left\Vert v_{1}-v_{2}\right\Vert _{s}^{2}%
e^{2as}ds\\
&  \leq C_{1}|||v_{1}-v_{2}|||_{a}^{2}\frac{e^{2at}-1}{2a}%
\end{align*}
which gets, as $a\geq2\left(  C_{1}+1\right)  $:%
\[
|||\overline{v}_{1}-\overline{v}_{2}|||_{a}\leq\frac{1}{2}|||v_{1}%
-v_{2}|||_{a}%
\]
that is $\Lambda$ is a contraction mapping in $L_{ad}^{2}\left(
\Omega;C\left(  \left[  0,T\right]  ;H\right)  \right)  .$ By Banach fixed
point theorem a unique $v\in L_{ad}^{2}\left(  \Omega;C\left(  \left[
0,T\right]  ;H\right)  \right)  $ exists such that $\Lambda v=v$. Hence the
problem (\ref{3.29}) has a unique solution.\hfill
\end{proof}

\begin{proposition}
\label{p3.8}Let the hypotheses ($\widetilde{H}_{0},...,\widetilde{H}_{3}$) be
satisfied. Then $u$ is the solution of the equation (\ref{3.27}) in the sense
of Definition \ref{d3.2} if and only if $u$ is solution of the problem
(\ref{3.29}).
\end{proposition}

\begin{proof}
Let $u$ be a stochastic process satisfying (\ref{3.28}). Then (under the
assumptions ($\widetilde{H}_{0},...,\widetilde{H}_{3}$), by Theorem \ref{t3.4}
there exists a unique generalized solution $\left(  \overline{u}%
,\overline{\eta}\right)  $ corresponding to $(u_{0}$,$f\left(  \cdot,u\left(
\cdot\right)  \right)  $,$M\left(  \cdot;u\right)  )$, $\overline{u}=GS\left(
A;u_{0},f\left(  \cdot,u\right)  ,M\left(  \cdot;u\right)  \right)  ,$ and
this solution satisfies (\ref{3.29}-a,b). If we put in (\ref{3.28}):
$s=0,\;z\left(  0\right)  =u_{0},\;g=f\left(  \cdot,u\left(  \cdot\right)
\right)  ,\;D=B\left(  \cdot,u\left(  \cdot\right)  \right)  ,\;z=\overline
{u}\;$we have $\mathbb{E}\left\vert u\left(  t\right)  -\overline{u}\left(
t\right)  \right\vert ^{2}\leq0,$ which yields, by the continuity of the
trajectories, that $u=\overline{u}$ in $L_{ad}^{2}\left(  \Omega;C\left(
\left[  0,T\right]  ;H\right)  \right)  .$ Hence $u$ is solution of
(\ref{3.29}).

Converse if $u$ is solution of (\ref{3.29}), then $u$ satisfies the conditions
$a)$ and $b)$ of Definition \ref{d3.2} and, also, the condition $c)$ via
Proposition \ref{p3.2}.\hfill
\end{proof}

Propositions \ref{p3.7} and \ref{p3.8} yield clearly:

\begin{theorem}
\label{t3.9}Under the assumptions ($\widetilde{H}_{0},...,\widetilde{H}_{3}$)
the equation (\ref{3.27}) has a unique solution in the sense of Definition
\ref{d3.2}. Moreover if $u_{01},u_{02}\in L_{ad}^{2p}\left(  \Omega
,\mathcal{F}_{0},\mathbb{P},;H\right)  $, $p\in\lbrack1,\infty),$ then
$u_{1}=u\left(  \cdot;u_{01}\right)  ,\;u_{2}=u\left(  \cdot;u_{02}\right)
\in L_{ad}^{2p}\left(  \Omega;C\left(  \left[  0,T\right]  ;H\right)  \right)
,$ and
\begin{equation}
\left\{
\begin{tabular}
[c]{l}%
$a)\;\;\mathbb{E}\sup\limits_{t\in\left[  0,T\right]  }\left\vert u\left(
t;u_{01}\right)  \right\vert ^{2p}\leq C_{1}(1+\mathbb{E}\left\vert
u_{01}\right\vert ^{2p}),$\medskip\\
$b)\;\;\mathbb{E}\sup\limits_{t\in\left[  0,T\right]  }\left\vert u\left(
t;u_{01}\right)  -u\left(  t;u_{02}\right)  \right\vert ^{2p}\leq
C_{2}\mathbb{E}\left\vert u_{01}-u_{02}\right\vert ^{2p},$%
\end{tabular}
\ \right.  \label{3.30}%
\end{equation}
where $C_{1}=C_{1}\left(  b,b_{1},p,T,x,y\right)  >0$ and $C_{2}=C_{2}\left(
L_{1},L,T,p\right)  >0$, $\left[  x,y\right]  \in A$ arbitrary fixed.
\end{theorem}

\begin{proof}
Now we have to do is proving (\ref{3.30}). By Propositions \ref{p3.7} and
\ref{p3.2} for all $\left[  x,y\right]  \in A$ we have:
\begin{equation}%
\begin{tabular}
[c]{l}%
$\displaystyle\left\vert u\left(  t\right)  -x\right\vert ^{2}\leq\left\vert
u_{0}-x\right\vert ^{2}+2\int_{0}^{t}\left(  u-x,f\left(  u\right)  -y\right)
ds$\medskip\\
$\displaystyle+\int_{0}^{t}\left\vert B\left(  u\right)  -0\right\vert
_{Q}^{2}ds+2\int_{0}^{t}\left(  u-x,B\left(  u\right)  -0\right)  dW\left(
s\right)  $%
\end{tabular}
\ \ \label{3.31}%
\end{equation}
$\forall\,t\in\left[  0,T\right]  ,$ $a.s.\;\omega\in\Omega,$ since
$x=GS\left(  x,y,\int_{0}^{\cdot}0dW\right)  .$

Let the stopping time:%
\[
\tau_{n}\left(  \omega\right)  =\left\{
\begin{tabular}
[c]{l}%
inf$\{t\in\left[  0,T\right]  :\left\vert u\left(  \omega,t\right)
\right\vert \geq n\},$\medskip\\
$T,\quad$if $\left\vert u\left(  \omega,t\right)  \right\vert <n,\;\forall
\,t\in\left[  0,T\right]  .$%
\end{tabular}
\ \ \ \right.
\]
We substitute in (\ref{3.31}) $t$ by $\tau_{n}\left(  \omega\right)  \wedge
t.$ Using Burkholder-Davis-Gundy inequality we have:%
\[%
\begin{array}
[c]{l}%
\displaystyle\mathbb{E}\sup_{s\in\left[  0,t\right]  }\left\vert \int%
_{0}^{s\wedge\tau_{n}}\left(  u-x,B\left(  u\right)  dW\right)  \right\vert
^{p}\leq\medskip\\
\displaystyle\leq9\left(  2p\right)  ^{p}\mathbb{E}\left(  \int_{0}%
^{t\wedge\tau_{n}}\left\vert u-x\right\vert ^{2}\left\vert B\left(  u\right)
\right\vert _{Q}^{2}ds\right)  ^{p/2}\medskip\\
\displaystyle\leq\frac{1}{2}\mathbb{E}\sup_{s\in\left[  0,t\right]
}\left\vert u\left(  s\wedge\tau_{n}\right)  \right\vert ^{2p}+C_{1}\left(
1+\mathbb{E}\int_{0}^{t\wedge\tau_{n}}\left\vert u\left(  s\right)
\right\vert ^{2p}ds\right)  ,
\end{array}
\]
where $C_{1}=C_{1}\left(  b,T,p,x\right)  >0$ is independent of $n$. Then from
(\ref{3.31}), after some elementary calculus, we obtain:%
\[
\mathbb{E}\sup_{s\in\left[  0,t\right]  }\left\vert u\left(  s\wedge\tau
_{n}\right)  \right\vert ^{2p}\leq C\left(  1+\mathbb{E}\left\vert
u_{0}\right\vert ^{2p}+\mathbb{E}\int_{0}^{t\wedge\tau_{n}}\left\vert u\left(
s\right)  \right\vert ^{2p}ds\right)
\]
which yields:%
\[
\mathbb{E}\sup_{s\in\left[  0,t\right]  }\left\vert u\left(  s\wedge\tau
_{n}\right)  \right\vert ^{2p}\leq C\left(  1+\mathbb{E}\left\vert
u_{0}\right\vert ^{2p}\right)
\]
and passing to limit as $n\rightarrow\infty,$ (\ref{3.30}-a) follows. The
inequality (\ref{3.30}-b) is obtained in the same manner. From Propositions
\ref{p3.7} and \ref{p3.2} we have%
\begin{align*}
\left\vert u_{1}\left(  t\right)  -u_{2}\left(  t\right)  \right\vert ^{2}  &
\leq\left\vert u_{01}-u_{02}\right\vert ^{2}+2\int_{0}^{t}\left(  u_{1}%
-u_{2},f\left(  u_{1}\right)  -f\left(  u_{2}\right)  \right)  ds\\
&  +\int_{0}^{t}\left\vert B\left(  u_{1}\right)  -B\left(  u_{2}\right)
\right\vert _{Q}^{2}ds+2\int_{0}^{t}\left(  u_{1}-u_{2},B\left(  u_{1}\right)
-B\left(  u_{2}\right)  \right)  dW\left(  s\right)
\end{align*}
As above using Burkholder-Davis-Gundy inequality and Gronwall inequality it is
clear that this last inequality yields (\ref{3.30}-b).\hfill
\end{proof}

\begin{corollary}
\label{c3.10}Let the assumptions ($\widetilde{H}_{0},...,\widetilde{H}_{3}$)
be satisfied, $p\in\lbrack1,\infty)$ and $u_{0}\in L^{p}(\Omega,F_{0}%
,\mathbb{P};H)$. If $\left(  u_{\varepsilon},\eta_{\varepsilon}\right)  $ is
the solution of the approximating problem:%
\begin{equation}%
\begin{array}
[c]{l}%
du_{\varepsilon}\left(  t\right)  +A_{\varepsilon}\left(  u_{\varepsilon
}\left(  t\right)  \right)  dt=f\left(  t,u_{\varepsilon}\left(  t\right)
\right)  dt+B\left(  t,u_{\varepsilon}\left(  t\right)  \right)  dW\left(
t\right)  \medskip\\
u_{\varepsilon}\left(  0\right)  =u_{0}\medskip\\
\displaystyle\eta_{\varepsilon}\left(  t\right)  =\int_{0}^{t}A_{\varepsilon
}\left(  u_{\varepsilon}\left(  s\right)  \right)  ds,
\end{array}
\label{3.32}%
\end{equation}
$\varepsilon\in(0,1]$,\ where $A_{\varepsilon}$\ is the Yosida approximation
of the maximal monotone operator $A$. Then there exists a positive constant
$C=C(T$,$r_{0}$,$h_{0}$,$a_{1}$,$a_{2}$,$b_{1}$,$b)$ such that:%
\begin{equation}
\mathbb{E}\sup_{t\in\left[  0,T\right]  }\left\vert u_{\varepsilon}\left(
t\right)  \right\vert ^{2p}+\mathbb{E}\left\Vert \eta_{\varepsilon}\right\Vert
_{BV\left(  \left[  0,T\right]  ;X^{\ast}\right)  }^{p}\leq C(1+\mathbb{E}%
\left\vert u_{0}\right\vert ^{2p}) \label{3.33}%
\end{equation}
and%
\begin{equation}
\lim_{\varepsilon\searrow0}u_{\varepsilon}=u,\quad\lim_{\varepsilon\searrow
0}\eta_{\varepsilon}=\eta\quad\quad\text{in }L_{ad}^{2}\left(  \Omega;C\left(
\left[  0,T\right]  ;H\right)  \right)  \label{3.34}%
\end{equation}

\end{corollary}

\begin{proof}
We know, by Pardoux result (Theorem \ref{t3.2}) or directly by a fixed point
argument, that,under the assumptions $\left(  \widetilde{H}_{0}%
,...,\widetilde{H}_{3}\right)  ,$ the equation (\ref{3.32}) has a unique
solution $u_{\varepsilon}\in L_{ad}^{2}\left(  \Omega;C\left(  \left[
0,T\right]  ;H\right)  \right)  .$ By Energy Equality:%
\begin{equation}%
\begin{array}
[c]{l}%
\left\vert u_{\varepsilon}\left(  t\right)  -h_{0}\right\vert ^{2}+2\int%
_{0}^{t}\left(  A_{\varepsilon}\left(  u_{\varepsilon}\right)  ,u_{\varepsilon
}-h_{0}\right)  ds\medskip\\
=\left\vert u_{0}-h_{0}\right\vert ^{2}+2\int_{0}^{t}\left(  u_{\varepsilon
}-h_{0},f\left(  u_{\varepsilon}\right)  \right)  ds\medskip\\
\quad\displaystyle+2\int_{0}^{t}\left(  u_{\varepsilon}-h_{0},B\left(
u_{\varepsilon}\right)  \right)  dW\left(  s\right)  +\int_{0}^{t}\left\vert
B\left(  u_{\varepsilon}\right)  \right\vert _{Q}^{2}ds.
\end{array}
\label{3.35}%
\end{equation}
But by $\left(  \widetilde{H}_{1}-ii\right)  $ and (\ref{2.1}-a,e) we have:%
\begin{align*}
r_{0}\left\Vert A_{\varepsilon}x\right\Vert _{X^{\ast}}  &  \leq\left(
A_{\varepsilon}x,J_{\varepsilon}x-h_{0}\right)  +a_{1}\left\vert
J_{\varepsilon}x\right\vert ^{2}+a_{2}\\
&  \leq\left(  A_{\varepsilon}x,x-h_{0}\right)  +2a_{1}\left\vert x\right\vert
^{2}+\left(  8a_{1}\left\vert J_{1}0\right\vert ^{2}+a_{2}\right)  ,
\end{align*}
for all $x\in H,\;\varepsilon\in(0,1]$, and by (\ref{3.7}-d) $\forall
\,\delta>0:$%
\[%
\begin{array}
[c]{l}%
\displaystyle\mathbb{E}\sup_{s\in\left[  0,t\right]  }\left\vert \int_{0}%
^{s}\left(  u_{\varepsilon}-h_{0},B\left(  u_{\varepsilon}\right)  \right)
dW\left(  \tau\right)  \right\vert ^{p}\leq\medskip\\
\displaystyle\leq9\left(  2p\right)  ^{p}\mathbb{E}\left(  \int_{0}%
^{t}\left\vert u_{\varepsilon}-h_{0}\right\vert ^{2}\left\vert B\left(
u_{\varepsilon}\right)  \right\vert _{Q}^{2}ds\right)  ^{p/2}\medskip\\
\displaystyle\leq\delta\,\mathbb{E}\sup_{s\in\left[  0,t\right]  }\left\vert
u\left(  s\right)  \right\vert ^{2p}+\frac{C\left(  p,T,h_{0},b\right)
}{\delta}\,\mathbb{E}\int_{0}^{t}\left\vert u_{\varepsilon}\left(  s\right)
\right\vert ^{2p}ds.
\end{array}
\]
These estimates used in (\ref{3.35}) produce clearly, by a standard calculus,
(\ref{3.33}).

From Proposition \ref{p3.8} we know that $u=GS(A;u_{0},f\left(  \cdot
,u\right)  ,M\left(  \cdot;u\right)  )$. Hence%
\[
u=\lim_{\varepsilon\searrow0}\widetilde{u}_{\varepsilon}\;\text{in}%
\;L_{ad}^{2}\left(  \Omega;C\left(  \left[  0,T\right]  ;H\right)  \right)  ,
\]
where $\widetilde{u}_{\varepsilon}$ is the solution of the equation%
\[%
\begin{array}
[c]{l}%
d\widetilde{u}_{\varepsilon}\left(  t\right)  +A_{\varepsilon}\left(
\widetilde{u}_{\varepsilon}\left(  t\right)  \right)  dt=f\left(  t,u\left(
t\right)  \right)  dt+B\left(  t,u\left(  t\right)  \right)  dW\left(
t\right)  \medskip\\
u_{\varepsilon}\left(  0\right)  =u_{0},
\end{array}
\]
$\varepsilon\in(0,1].$

We write the Energy Equality for $u_{\varepsilon}-\widetilde{u}_{\varepsilon
}:$%
\[%
\begin{array}
[c]{l}%
\displaystyle|u_{\varepsilon}\left(  t\right)  -\widetilde{u}_{\varepsilon
}\left(  t\right)  |^{2}+2\int_{0}^{t}\left(  A_{\varepsilon}\left(
u_{\varepsilon}\right)  -A_{\varepsilon}\left(  \widetilde{u}_{\varepsilon
}\right)  ,u_{\varepsilon}-\widetilde{u}_{\varepsilon}\right)  ds\medskip\\
\displaystyle=2\int_{0}^{t}\left(  u_{\varepsilon}-\widetilde{u}_{\varepsilon
},f\left(  u_{\varepsilon}\right)  -f\left(  u\right)  \right)  ds+2\int%
_{0}^{t}\left(  u_{\varepsilon}-\widetilde{u}_{\varepsilon},B\left(
u_{\varepsilon}\right)  -B\left(  u\right)  \right)  dW\left(  s\right)
\medskip\\
\displaystyle\quad+\int_{0}^{t}\left\vert B\left(  u_{\varepsilon}\right)
-B\left(  u\right)  \right\vert _{Q}^{2}ds.
\end{array}
\]
By $\left(  \widetilde{H}_{2}\right)  $ and $\left(  \widetilde{H}_{3}\right)
$we have%
\begin{align*}
2\left(  u_{\varepsilon}-\widetilde{u}_{\varepsilon},f\left(  u_{\varepsilon
}\right)  -f\left(  u\right)  \right)   &  \leq3L_{1}\left\vert u_{\varepsilon
}-\widetilde{u}_{\varepsilon}\right\vert ^{2}+L_{1}\left\vert \widetilde{u}%
_{\varepsilon}-u\right\vert ^{2},\\
\left\vert B\left(  u_{\varepsilon}\right)  -B\left(  u\right)  \right\vert
_{Q}^{2}  &  \leq2L\left\vert u_{\varepsilon}-\widetilde{u}_{\varepsilon
}\right\vert ^{2}+2L\left\vert \widetilde{u}_{\varepsilon}-u\right\vert ^{2}.
\end{align*}
Hence%
\[%
\begin{array}
[c]{l}%
\mathbb{E}\sup_{s\in\left[  0,t\right]  }\left\vert u_{\varepsilon}\left(
s\right)  -\widetilde{u}_{\varepsilon}\left(  s\right)  \right\vert ^{2}%
\leq(3L_{1}+2L)\mathbb{E}\int_{0}^{T}\left\vert \widetilde{u}_{\varepsilon
}\left(  s\right)  -u\left(  s\right)  \right\vert ^{2}ds\medskip\\
\;\;\;\;\;\;\;\;\;\;\;\;\;+(3L_{1}+2L)\int_{0}^{t}\mathbb{E}\sup_{s\in\left[
0,\tau\right]  }\left\vert u_{\varepsilon}\left(  s\right)  -\widetilde{u}%
_{\varepsilon}\left(  s\right)  \right\vert ^{2}d\tau\medskip\\
\;\;\;\;\;\;\;\;\;\;\;\;\;+6\mathbb{E}\left(  \int_{0}^{t}\left\vert
u_{\varepsilon}-\widetilde{u}_{\varepsilon}\right\vert ^{2}\left\vert B\left(
u_{\varepsilon}\right)  -B\left(  u\right)  \right\vert _{Q}^{2}ds\right)
^{1/2}%
\end{array}
\]
which easily yields%
\[
\mathbb{E}\sup_{s\in\left[  0,T\right]  }\left\vert u_{\varepsilon}\left(
s\right)  -\widetilde{u}_{\varepsilon}\left(  s\right)  \right\vert ^{2}\leq
C\,\mathbb{E}\sup_{s\in\left[  0,T\right]  }\left\vert \widetilde{u}%
_{\varepsilon}\left(  s\right)  -u\left(  s\right)  \right\vert ^{2},
\]
where $C=C\left(  T,L,L_{1}\right)  >0$ and (\ref{3.34}) follows.\hfill
\end{proof}

\section{Large time behaviour}

\subsection{Exponentially stability}

In this section we shall study some asymptotic properties of the generalized
(stochastic) solution of the equation (\ref{3.27}). We shall consider the same
context of the spaces, $X\subset H\subset X^{\ast},$ as in Section 3.3 and,
the assumptions, $\left(  \widetilde{H}_{0},...,\widetilde{H}_{3}\right)  $ be
satisfied. Assume also that%
\begin{equation}%
\begin{tabular}
[c]{l}%
$\exists\,a\geq0$ \ such that $\forall\,\left[  x_{1},y_{1}\right]  ,\,\left[
x_{2},y_{2}\right]  \in A\medskip$\\
$\left(  y_{1}-y_{2},x_{1}-x_{2}\right)  \geq a\left\vert x_{1}-x_{2}%
\right\vert ^{2}$%
\end{tabular}
\ \label{4.1}%
\end{equation}

\begin{lemma}
\label{l4.1} If $\left(  \widetilde{H}_{1}-i\right)  $ and (\ref{4.1}) are
satisfied then $\forall\theta\in\left[  0,1\right)  $, $\forall\varepsilon>0$,
such that $\varepsilon a\theta\leq1-\theta$, the inequality $\left(
A_{\varepsilon}u-A_{\varepsilon}v,u-v\right)  \geq a\theta\left\vert
u-v\right\vert ^{2}\;$\ holds $\forall\,u,v\in H.$
\end{lemma}

\begin{proof}
We have%
\begin{align*}
\left(  A_{\varepsilon}u-A_{\varepsilon}v,u-v\right)   &  =\left(
A_{\varepsilon}u-A_{\varepsilon}v,J_{\varepsilon}u+\varepsilon A_{\varepsilon
}u-J_{\varepsilon}v-\lambda A_{\varepsilon}v\right) \\
&  \geq a\left\vert J_{\varepsilon}u-J_{\varepsilon}v\right\vert ^{2}%
+\lambda\left\vert A_{\varepsilon}u-A_{\varepsilon}v\right\vert ^{2}\\
&  =a\left\vert u-v-\varepsilon\left(  A_{\varepsilon}u-A_{\varepsilon
}v\right)  \right\vert ^{2}++\varepsilon\left\vert A_{\varepsilon
}u-A_{\varepsilon}v\right\vert ^{2}.
\end{align*}
But \ $\left\vert u-v\right\vert ^{2}\geq\theta\left\vert u\right\vert
^{2}-\frac{\theta}{1-\theta}\left\vert v\right\vert ^{2},\;\;\forall\,u,v\in
H.$ Hence
\begin{align*}
\left(  A_{\varepsilon}u-A_{\varepsilon}v,u-v\right)   &  \geq a\theta
\left\vert u-v\right\vert ^{2}+\varepsilon(1-\frac{a\theta\varepsilon
}{1-\theta})\left\vert A_{\varepsilon}u-A_{\varepsilon}v\right\vert ^{2}\\
&  \geq a\theta\left\vert u-v\right\vert ^{2}.
\end{align*}
\hfill
\end{proof}

\begin{proposition}
\label{p4.2}Let ($\widetilde{H}_{1},\widetilde{H}_{2},\widetilde{H}_{3}$) and
(\ref{4.1}) be satisfied and $u,v\;$be two generalized solutions of equation
(\ref{3.27}) corresponding to the initial dates $u_{0},v_{0}$, respectively
($u_{0},v_{0}\in L^{2}(\Omega,\mathcal{F}_{0},\mathbb{P};\overline{D\left(
A\right)  })$). Let $\theta\in\left[  0,1\right)  $ and $\beta=2a\theta
-2L-L_{1}$. Then%
\begin{equation}%
\begin{tabular}
[c]{l}%
$\mathbb{E}\left(  \left\vert u\left(  t\right)  -v\left(  t\right)
\right\vert ^{2}\left\vert \mathcal{F}_{s}\right.  \right)  \leq
e^{-\beta\left(  t-s\right)  }\left\vert u\left(  s\right)  -v\left(
s\right)  \right\vert ^{2}$\medskip\\
\quad\quad\quad\quad\ for all\ $0\leq s\leq t\leq T$, a.s. $\omega\in\Omega$.
\end{tabular}
\ \label{4.2}%
\end{equation}

\end{proposition}

\begin{proof}
By Corollary \ref{c3.10} we have: \ $u=\lim\limits_{\varepsilon\rightarrow
0}u_{\varepsilon},\;v=\lim\limits_{\varepsilon\rightarrow0}v_{\varepsilon}$ in
$L_{ad}^{2}\left(  \Omega;C\left(  \left[  0,T\right]  ;H\right)  \right)
$\ where $u_{\varepsilon}$ and $v_{\varepsilon}$ are the solutions of the
approximating equations:%
\[
\left\{
\begin{tabular}
[c]{l}%
$du_{\varepsilon}+A_{\varepsilon}u_{\varepsilon}dt=f\left(  t,u_{\varepsilon
}\left(  t\right)  \right)  dt+B\left(  t,u_{\varepsilon}\left(  t\right)
\right)  dW\left(  t\right)  $\medskip\\
$u_{\varepsilon}\left(  0\right)  =u_{0}$%
\end{tabular}
\ \ \ \right.
\]
and%
\[
\left\{
\begin{tabular}
[c]{l}%
$dv_{\varepsilon}\left(  t\right)  +A_{\varepsilon}v_{\varepsilon}dt=f\left(
t,v_{\varepsilon}\left(  t\right)  \right)  dt+B\left(  t,v_{\varepsilon
}\left(  t\right)  \right)  dW\left(  t\right)  $\medskip\\
$v_{\varepsilon}\left(  0\right)  =v_{0}$%
\end{tabular}
\ \ \ \right.
\]
respectively.

We write the Energy Equality for $\left\vert u_{\varepsilon}-v_{\varepsilon
}\right\vert ^{2}:$%
\[%
\begin{tabular}
[c]{l}%
$\displaystyle\left\vert u_{\varepsilon}\left(  t\right)  -v_{\varepsilon
}\left(  t\right)  \right\vert ^{2}+2\int\nolimits_{s}^{t}\left(
A_{\varepsilon}u_{\varepsilon}-A_{\varepsilon}v_{\varepsilon},u_{\varepsilon
}-v_{\varepsilon}\right)  d\tau$\medskip\\
$\displaystyle=\left\vert u_{\varepsilon}\left(  s\right)  -v_{\varepsilon
}\left(  s\right)  \right\vert ^{2}+2\int\nolimits_{s}^{t}\left(
u_{\varepsilon}-v_{\varepsilon},f\left(  u_{\varepsilon}\right)  -f\left(
v_{\varepsilon}\right)  \right)  ds+$\medskip\\
$\quad\displaystyle+2\int\nolimits_{s}^{t}\left(  u_{\varepsilon
}-v_{\varepsilon},\left(  B\left(  u_{\varepsilon}\right)  -B\left(
v_{\varepsilon}\right)  \right)  dW\right)  +\int\nolimits_{s}^{t}\left\vert
B\left(  u_{\varepsilon}\right)  -B\left(  v_{\varepsilon}\right)  \right\vert
_{Q}^{2}ds.$%
\end{tabular}
\ \ \
\]
Using Lemma \ref{l4.1} and the hypotheses ($\widetilde{H}_{2}$) and
($\widetilde{H}_{3}$) we have for $\varepsilon>0$, $\varepsilon a\theta
<1-\theta$:%
\[%
\begin{tabular}
[c]{l}%
$\displaystyle\left\vert u_{\varepsilon}\left(  t\right)  -v_{\varepsilon
}\left(  t\right)  \right\vert ^{2}+2a\theta\int\nolimits_{s}^{t}\left\vert
u_{\varepsilon}-v_{\varepsilon}\right\vert ^{2}d\tau\leq\left\vert
u_{\varepsilon}\left(  s\right)  -v_{\varepsilon}\left(  s\right)  \right\vert
^{2}+$\medskip\\
$\displaystyle\quad+\left(  L+2L_{1}\right)  \int\nolimits_{s}^{t}\left\vert
u_{\varepsilon}-v_{\varepsilon}\right\vert ^{2}d\tau+2\int\nolimits_{s}%
^{t}\left(  u_{\varepsilon}-v_{\varepsilon},\left(  B\left(  u_{\varepsilon
}\right)  -B\left(  v_{\varepsilon}\right)  \right)  dW\right)  $%
\end{tabular}
\]
for all \ $0\leq s\leq t\leq T,$ a.s. $\omega\in\Omega.$ Let $\varepsilon
=\varepsilon_{n}\rightarrow0$ such that%
\[
u_{\varepsilon_{n}}\rightarrow u,\;v_{\varepsilon_{n}}\rightarrow
v\quad\text{in }C\left(  \left[  0,T\right]  ;H\right)  ,\;\text{a.s. }%
\omega\in\Omega.
\]
Passing to limit in the last inequality for $\varepsilon=\varepsilon
_{n}\rightarrow0$ we obtain:%
\[%
\begin{tabular}
[c]{l}%
$\displaystyle\left\vert u\left(  t\right)  -v\left(  t\right)  \right\vert
^{2}+2a\theta\int\nolimits_{s}^{t}\left\vert u-v\right\vert ^{2}d\tau
\leq\left\vert u\left(  s\right)  -v\left(  s\right)  \right\vert ^{2}%
+$\medskip\\
$\displaystyle\quad+\left(  2L_{1}+L\right)  \int\nolimits_{s}^{t}\left\vert
u-v\right\vert ^{2}d\tau+2\int\nolimits_{s}^{t}\left(  u-v,\left(  B\left(
u\right)  -B\left(  v\right)  \right)  dW\right)  $%
\end{tabular}
\]
for all $0\leq s\leq t\leq T,$ a.s. $\omega\in\Omega.$

Let $s_{0}\in\left[  0,T\right]  $ be fixed, and $\varphi\left(  \tau\right)
=\mathbb{E}\left(  \left\vert u\left(  \tau\right)  -v\left(  \tau\right)
\right\vert ^{2}\left\vert \mathcal{F}_{s_{0}}\right.  \right)  $. Then
$\varphi\left(  t\right)  +\beta\int\nolimits_{s}^{t}\varphi\left(
\tau\right)  d\tau\leq\varphi\left(  s\right)  $ for all $s_{0}\leq s\leq
t\leq T$, a.s. $\omega\in\Omega$, which implies $\varphi\left(  t\right)  \leq
e^{-\beta\left(  t-s\right)  }\varphi\left(  s\right)  $ that is%
\begin{equation}%
\begin{tabular}
[c]{r}%
$\mathbb{E}\left(  \left\vert u\left(  t\right)  -v\left(  t\right)
\right\vert ^{2}\left\vert \mathcal{F}_{s_{0}}\right.  \right)  \leq
e^{-\beta\left(  t-s\right)  }\mathbb{E}\left(  \left\vert u\left(  s\right)
-v\left(  s\right)  \right\vert ^{2}\left\vert \mathcal{F}_{s_{0}}\right.
\right)  $\medskip\\
\thinspace\thinspace\thinspace\thinspace\thinspace\quad\quad\quad\quad
\quad\quad\quad\quad\thinspace for all $s_{0}\leq s\leq t\leq T$, a.s.
$\omega\in\Omega$%
\end{tabular}
\label{4.3}%
\end{equation}
The inequality (\ref{4.3}) yields (\ref{4.2}) for $s=s_{0}.$\hfill
\end{proof}

\begin{corollary}
\label{c4.3}Under the assumptions of Proposition \ref{p4.2} $\;\Delta\left(
t\right)  =e^{\beta t}\left\vert u\left(  t\right)  -v\left(  t\right)
\right\vert ^{2}$ is a supermartingale.
\end{corollary}

\begin{proof}
It is evidently that: $\Delta\in L_{ad}^{1}\left(  \Omega;C\left(  \left[
0,T\right]  ;R\right)  \right)  ,\;\mathbb{E}\left(  \Delta\left(  t\right)
\left\vert \mathcal{F}_{s}\right.  \right)  \leq\Delta\left(  s\right)  $ for
$s\leq t$.\hfill
\end{proof}

We shall consider from now on $f$ and $B$ that:
\begin{equation}%
\begin{tabular}
[c]{l}%
$f:\Omega\times\left[  0,\infty\right[  \times H\rightarrow H,$ $B:\Omega
\times\left[  0,\infty\right[  \times H\rightarrow\mathcal{L}^{2}\left(
U_{0},H\right)  $\medskip\\
satisfy ($\widetilde{H}_{2}$) and ($\widetilde{H}_{3}$) respectively on each
interval $\left[  0,T\right]  $\medskip\\
with the constants $L,b_{1},L_{1}$ independent of $T$%
\end{tabular}
\label{4.4}%
\end{equation}
If $u_{\left[  0,T\right]  }\left(  t;0,u_{0}\right)  $ is the solution on
$\left[  0,T\right]  $ then by uniqueness%
\[
u_{\left[  0,T_{1}\right]  }(t;0,u_{0})=u_{\left[  0,T_{2}\right]  }\left(
t;0,u_{0}\right)  ,\forall\,t\in\lbrack0,T_{1}\wedge T_{2}]\text{, a.s.}%
\]
Thus we can define $u\left(  \cdot;0,u_{0}\right)  :\Omega\times\left[
0,\infty\right[  \rightarrow H\;$ by $\;u\left(  t;0,u_{0}\right)  =u_{\left[
0,n\right]  }\left(  t;0,u_{0}\right)  $ for $t\in\left[  0,n\right]  ,\;n\in
N^{\ast}.$

\begin{theorem}
\label{t4.4} Assume that ($\widetilde{H}_{1}$), (\ref{4.4}) hold, and the
operator $A$\ satisfies, moreover, (\ref{4.1}) with $a>0.$ If $\beta
_{0}=2a-2L-L_{1}>0$\ then for all two solutions $u,v$ corresponding to initial
dates $u_{0},v_{0}\in L^{2}\left(  \Omega,\mathcal{F}_{0},\mathbb{P}%
;\overline{D\left(  A\right)  }\right)  $, respectively, we have:

$a)\quad\mathbb{E}\left\vert u\left(  t\right)  -v\left(  t\right)
\right\vert ^{2}\leq e^{-\beta_{0}t}\mathbb{E}\left\vert u_{0}-v_{0}%
\right\vert ^{2}$\medskip

$b)\quad$for all $\gamma\in\left[  0,\beta_{0}\right)  \;\lim
\limits_{t\rightarrow\infty}e^{\gamma t}\left\vert u\left(  t\right)
-v\left(  t\right)  \right\vert ^{2}=0$ in $L^{1}\left(  \Omega,\mathcal{F}%
,\mathbb{P};R\right)  $ and $\mathbb{P}$-a.s..\medskip

$c)\quad\displaystyle\mathbb{E}\int\nolimits_{0}^{\infty}\left\vert u\left(
t\right)  -v\left(  t\right)  \right\vert ^{2}dt\leq\frac{1}{\beta_{0}%
}\mathbb{E}\left\vert u_{0}-v_{0}\right\vert ^{2}$\medskip

Besides if $u_{0}\equiv x\in\overline{D\left(  A\right)  }$ and $v_{0}%
\equiv\widetilde{x}\in\overline{D\left(  A\right)  }$ then for all $\gamma
\in\left[  0,\beta_{\cdot}\right)  $\ there exists a random variable
$\tau\left(  \omega\right)  =\tau\left(  \omega;a,x-\widetilde{x}\right)
>0$\ a.s. such that$\smallskip$

$d)\quad\left\vert u\left(  \omega,t\right)  -v\left(  \omega,t\right)
\right\vert ^{2}\leq e^{-\gamma t}\left\vert x-\widetilde{x}\right\vert ^{2}$
for all $t\geq\tau\left(  \omega\right)  $, a.s. $\omega\in\Omega.$
\end{theorem}

\begin{proof}
The results are easily obtained from (\ref{4.2}) as $\theta\nearrow1$ and
$s=0,$ and Corollary \ref{c4.3}.\hfill
\end{proof}

\subsection{Invariant measure}

Let $D=Dom\left(  A\right)  ,\;B_{\overline{D}}$ the $\sigma$-algebras Borel
on $\overline{D},\;\mathcal{M}\left(  \overline{D}\right)  $ the space of
bounded measure on $B_{\overline{D}}$, $\mathcal{M}_{1}^{+}\left(
\overline{D}\right)  $ the space of probability measure on $B_{\overline{D}}$,
$B_{b}\left(  \overline{D}\right)  $ the Banach space of the bounded Borel
measurable functions $g:\overline{D}\rightarrow\mathbb{R}$ with the sup-norm,
and $C_{b}\left(  \overline{D}\right)  $ the Banach space of the bounded
continuous functions $g:\overline{D}\rightarrow\mathbb{R}$. We shall assume
that $\left(  \widetilde{H}_{1}\right)  $ and (\ref{4.4}) are satisfied. We
denote $u\left(  t;s,\xi\right)  $ the generalized (stochastic) solution
$u\left(  t\right)  $ of the problem%
\begin{equation}%
\begin{tabular}
[c]{l}%
$du\left(  t\right)  +\left(  Au\right)  \left(  dt\right)  \ni f\left(
t,u\left(  t\right)  \right)  dt+B\left(  t,u\left(  t\right)  \right)
dW_{s}\left(  t\right)  $\medskip\\
$u\left(  s\right)  =\xi,\;\;t\in\left[  s,\infty\right)  ,$%
\end{tabular}
\ \ \label{4.5}%
\end{equation}
where $\xi\in L^{2}\left(  \Omega,\mathcal{F}_{s},\mathbb{P};\overline
{D}\right)  $ and $W_{s}\left(  t\right)  =W\left(  t\right)  -W\left(
s\right)  .$

We define for $s\leq t:\;\left(  P_{st}g\right)  \left(  x\right)
=\mathbb{E}g\left(  u\left(  t;s,x\right)  \right)  ,$ where $g\in
B_{b}\left(  \overline{D}\right)  $ and$\;x\in\overline{D}.$ With a similar
proof as in \cite{DZSE} (Theorem 9.8, Corollaries 9.9 and 9.10, p.250-252) we
have:%
\begin{equation}%
\begin{tabular}
[c]{l}%
$a)\quad P_{st}\in\mathcal{M}_{1}^{+}\left(  \overline{D}\right)  ,$\medskip\\
$b)\quad u\left(  t;s,x\right)  $ is a Markov process with transition
probability $P_{st}$ i.e.\medskip\\
\quad\quad$\mathbb{E}\left(  g\left(  u\left(  t;s,x\right)  \right)
\left\vert \mathcal{F}_{\tau}\right.  \right)  =\left(  P_{\tau t}g\right)
\left(  u\left(  \tau;s,x\right)  \right)  ,\;\;$\medskip\\
$\quad\quad\quad\quad\quad\quad\quad\forall\,0\leq s\leq\tau\leq t,\;\forall
g\in B_{b}\left(  \overline{D}\right)  ,$\medskip\\
$c)\quad u\left(  t;s,x\right)  $ has the Feller property that is
$P_{st}\left(  C_{b}\left(  \overline{D}\right)  \right)  \subset C_{b}\left(
\overline{D}\right)  ,$\medskip\\
$d)\quad P_{s\tau}\left(  P_{\tau t}g\right)  =P_{st}g,\;\forall g\in
B_{b}\left(  D\right)  ,\;\;\forall\,0\leq s\leq\tau\leq t,$\medskip\\
$e)\quad$if $f(t,u)\equiv f\left(  u\right)  ,\;B\left(  t,u\right)  \equiv
B\left(  u\right)  $ then $P_{st}=P_{0,t-s}$%
\end{tabular}
\label{4.6}%
\end{equation}
In this subsection we shall study the behaviour as $t\rightarrow\infty$ of the
laws $L\left(  u\left(  t;u_{0}\right)  \right)  $ of the random variables
$u\left(  t;u_{0}\right)  =u\left(  t;0,u_{0}\right)  $ under the assumptions
\begin{equation}%
\begin{tabular}
[c]{l}%
$i)\quad f\left(  t,u\right)  \equiv f\left(  u\right)  ,\;B\left(
t,u\right)  \equiv B\left(  u\right)  $ ($f$ and $B$ are independent of
$t$)\medskip\\
$ii)\quad(\widetilde{H}_{1}),\;$(\ref{4.1}) with $a>0,$ and (\ref{4.4}) are
satisfied\medskip\\
$iii)\quad\beta_{0}=a-L_{1}-\frac{1}{2}L>0$%
\end{tabular}
\ \ \ \label{4.7}%
\end{equation}
Let $(P_{t}g)\left(  x\right)  =\mathbb{E}g\left(  u\left(  t;x\right)
\right)  =\left(  P_{st}g\right)  \left(  x\right)  $ where $t\geq0,\;g\in
B_{b}\left(  \overline{D}\right)  .$ For all $\mu\in M\left(  \overline
{D}\right)  $ and $g\in B_{b}\left(  \overline{D}\right)  $ we define
$\left\langle \mu,g\right\rangle =\int\nolimits_{\overline{D}}g\left(
x\right)  d\mu\left(  x\right)  $ and the dual semigroup $P_{t}^{\ast
}:\mathcal{M}\left(  \overline{D}\right)  \rightarrow\mathcal{M}\left(
\overline{D}\right)  $ by\ $\left\langle P_{t}^{\ast}\mu,g\right\rangle
=\left\langle \mu,P_{t}g\right\rangle $

\begin{definition}
\label{d4.1}A probability measure $\mu\in\mathcal{M}_{1}^{+}\left(
\overline{D}\right)  $\ is an invariant measure for the generalized
(stochastic) solution $u\left(  t;u_{0}\right)  $\ if $\forall t>0:\;P_{t}%
^{\ast}\mu=\mu$\ \ or equivalent%
\[
\int\nolimits_{\overline{D}}\mathbb{E}g\left(  u\left(  t;x\right)  \right)
d\mu\left(  x\right)  =\int\nolimits_{\overline{D}}g\left(  x\right)
d\mu\left(  x\right)  \text{, for all }g\in B_{b}\left(  \overline{D}\right)
\]
(or $\forall g\in C_{b}\left(  \overline{D}\right)  $, or $\forall
g=1_{\Gamma}$, $\Gamma\in B_{\overline{D}}$).
\end{definition}

As in \cite{DZSE} (Proposition 11.1 and 11.2, p.303-304), based only
(\ref{4.6}), we have:

\begin{proposition}
\label{p4.5}a) If the law of $\xi$ is $\nu$ then the law of $\xi$ is $\nu$
then the law of $u\left(  t;\xi\right)  $\ is $P_{t}^{\ast}\nu.$

b) If \ $\exists\,\xi\in L^{2}\left(  \Omega,\mathcal{F}_{0},\mathbb{P}%
;\overline{D}\right)  $\ such that \ $\lim\limits_{t\rightarrow\infty
}\mathbb{E}g\left(  u\left(  t;\xi\right)  \right)  =\left\langle
\mu,g\right\rangle ,\;\forall g\in C_{b}\left(  \overline{D}\right)  $ then
$\mu$\ is an invariant measure.
\end{proposition}

The first result of this subsection will give an information on the behaviour
of $u\left(  t;x\right)  ,$ $x\in D\left(  A\right)  $ as $t\rightarrow\infty$
in the sense of $L^{2}\left(  \Omega,\mathcal{F},\mathbb{P};H\right)  $-convergence.

\begin{proposition}
\label{p4.6}Under the assumptions (\ref{4.7}), for all $\left[  x_{0}%
,y_{0}\right]  \in A,$\ $\forall0\leq s\leq t:$%
\begin{equation}%
\begin{tabular}
[c]{l}%
$a)\quad\mathbb{E}\left\vert u\left(  t;x_{0}\right)  -x_{0}\right\vert
^{2}\leq C_{0}M_{0},$\medskip\\
\multicolumn{1}{r}{$b)\quad\mathbb{E}\left\vert u\left(  t;x_{0}\right)
-u\left(  s;x_{0}\right)  \right\vert ^{2}\leq C_{0}M_{0}e^{-\beta_{0}%
s}\left(  e^{-\beta_{0}s}-e^{-\beta_{0}t}\right)  $\medskip}\\
\multicolumn{1}{r}{$\quad\quad\quad\quad\quad\quad\quad\quad\quad\quad
\quad\quad\leq C_{0}M_{0}\beta_{0}\left(  t-s\right)  ,$}%
\end{tabular}
\label{4.8}%
\end{equation}
where $C_{0}=C_{0}\left(  a,L,L_{1}\right)  $\ is a positive constant (we can
put $C_{0}=\left(  \beta_{0}+2+L\right)  /\beta_{0}^{2}$) and%
\begin{equation}%
\begin{array}
[c]{l}%
M_{0}=\left\vert y_{0}\right\vert ^{2}+\left\vert f\left(  x_{0}\right)
\right\vert ^{2}+\left\vert B\left(  x_{0}\right)  \right\vert _{Q}^{2}%
\end{array}
\label{4.9}%
\end{equation}

\end{proposition}

\begin{proof}
We shall use the idea from [\cite{DZSE}] (Theorem 11.21, p.327). Let
$\widetilde{W}$ a $U$-valued $Q$-Wiener process independent of $W$. We extend
$W$ for $t<0$ by $W\left(  t\right)  =\widetilde{W}\left(  -t\right)
,\;t\leq0,$ and consider $u_{\varepsilon,\tau}\left(  t\right)
=u_{\varepsilon}\left(  t;-\tau,x_{0}\right)  $ the solution of the equation%
\[
\left\{
\begin{tabular}
[c]{l}%
$du_{\varepsilon,\tau}\left(  t\right)  +A_{\varepsilon}u_{\varepsilon,\tau
}\left(  t\right)  dt=f\left(  u_{\varepsilon,\tau}\left(  t\right)  \right)
dt+B\left(  u_{\varepsilon,\tau}\left(  t\right)  \right)  dW\left(  t\right)
$\medskip\\
$u_{\varepsilon,\tau}\left(  -\tau\right)  =x_{0},\;t\geq-\tau$%
\end{tabular}
\ \ \right.
\]
It is easy to see that $u_{\varepsilon,\tau}\left(  0\right)  $ and
$u_{\varepsilon}\left(  \tau;0,x_{0}\right)  $ have the same law. By Energy
Equality we have%
\begin{equation}
\frac{d}{dt}\mathbb{E}\left\vert u_{\varepsilon,\tau}\left(  t\right)
-x_{0}\right\vert ^{2}+\mathbb{E}G_{\varepsilon}\left(  u_{\varepsilon,\tau
}\left(  t\right)  \right)  =0,\;t>-\tau, \label{4.10}%
\end{equation}
where%
\[%
\begin{tabular}
[c]{l}%
$G_{\varepsilon}\left(  v\right)  =2\left(  A_{\varepsilon}v,v-x_{0}\right)
-2\left(  f\left(  v\right)  ,v-x_{0}\right)  -\left\vert B\left(  v\right)
\right\vert ^{2}$\medskip\\
$=2\left(  A_{\varepsilon}v-A_{\varepsilon}x_{0},v-x_{0}\right)  -2\left(
f\left(  v\right)  -f\left(  x_{0}\right)  ,v-x_{0}\right)  $\medskip\\
$\quad-\left\vert B\left(  v\right)  -B\left(  x_{0}\right)  \right\vert
_{Q}^{2}\,+2\left(  A_{\varepsilon}x_{0},v-x_{0}\right)  -2\left(  f\left(
x_{0}\right)  ,v-x_{0}\right)  $\medskip\\
$\quad-2\left(  B\left(  v\right)  -B\left(  x_{0}\right)  ,B\left(
x_{0}\right)  \right)  _{Q}-\left\vert B\left(  x_{0}\right)  \right\vert
_{Q}^{2}$%
\end{tabular}
\]
By Lemma \ref{l4.1} for $\theta\in\left(  0,1\right)  $ fixed arbitrary and
$\varepsilon\in\left(  0,\frac{1-\theta}{a\theta}\right)  $ we have%
\begin{align*}
G_{\varepsilon}\left(  v\right)   &  \geq\left(  2a\theta-2L_{1}%
-L-2\delta-L\delta\right)  \left\vert v-x_{0}\right\vert ^{2}\\
&  -\left(  1+\frac{1}{\delta}\right)  \left(  \left\vert y_{0}\right\vert
^{2}+\left\vert f\left(  x_{0}\right)  \right\vert ^{2}+\left\vert B\left(
x_{0}\right)  \right\vert _{Q}^{2}\right)
\end{align*}
for all $\delta>0.$ Using this last inequality in (\ref{4.10}) and integrating
from $-\tau$ to $t$ we obtain for $\delta=\delta_{0}=\beta_{0}/\left(
2+L\right)  :$%
\begin{equation}
\mathbb{E}\left\vert u_{\varepsilon,\tau}\left(  t\right)  -x_{0}\right\vert
^{2}\leq\left(  1+\frac{1}{\delta_{0}}\right)  M_{0}\int\nolimits_{-\tau}%
^{t}e^{\left(  \beta_{0}-2a\left(  1-\theta\right)  \right)  \left(
r-t\right)  }dr \label{4.11}%
\end{equation}
and more for $t=0:$%
\begin{equation}
\mathbb{E}\left\vert u_{\varepsilon}\left(  \tau;0,x_{0}\right)
-x_{0}\right\vert ^{2}\leq\left(  1+\frac{1}{\delta_{0}}\right)  M_{0}%
\int\nolimits_{0}^{\tau}e^{\left(  2a\left(  1-\theta\right)  -\beta
_{0}\right)  r}dr \label{4.12}%
\end{equation}
We pass to limit in (\ref{4.12}) as $\varepsilon\searrow0$ and then for
$\theta\nearrow1$; the inequality (\ref{4.8}-a) is yielded.

Applying Ito's formula to $\left\vert u_{\varepsilon,\sigma}\left(  t\right)
-u_{\varepsilon,\tau}\left(  t\right)  \right\vert ^{2},$ for $t\geq
-\sigma\geq-\tau$ one has:%
\begin{equation}
\frac{d}{dt}\mathbb{E}\left\vert u_{\varepsilon,\sigma}\left(  t\right)
-u_{\varepsilon,\tau}\left(  t\right)  \right\vert ^{2}+\mathbb{E}%
\widetilde{G}_{\varepsilon}\left(  u_{\varepsilon,\sigma}\left(  t\right)
,u_{\varepsilon,\tau}\left(  t\right)  \right)  =0, \label{4.13}%
\end{equation}
where
\begin{align*}
\widetilde{G}_{\varepsilon}\left(  u,v\right)   &  =2\left(  A_{\varepsilon
}u-A_{\varepsilon}v,u-v\right)  -2\left(  f\left(  u\right)  -f\left(
v\right)  ,u-v\right)  -\left\vert B\left(  u\right)  -B\left(  v\right)
\right\vert _{Q}^{2}\\
&  \geq\left(  2a\theta-2L_{1}-L\right)  \left\vert u-v\right\vert ^{2}%
\end{align*}
for $\theta\in\left(  0,1\right)  $ fixed arbitrary and $\varepsilon\in\left(
0,\frac{1-\theta}{a\theta}\right)  .$ We integrate (\ref{4.13}) from $-\sigma$
to $0$ and we obtain%
\[%
\begin{array}
[c]{l}%
\mathbb{E}\left\vert u_{\varepsilon}\left(  \sigma;0,x_{0}\right)
-u_{\varepsilon}\left(  \tau;0,x_{0}\right)  \right\vert ^{2}=\mathbb{E}%
\left\vert u_{\varepsilon,\sigma}\left(  0\right)  -u_{\varepsilon,\tau
}\left(  0\right)  \right\vert ^{2}\medskip\\
\displaystyle\leq\mathbb{E}\left(  \left\vert x_{0}-u_{\varepsilon,\tau
}\left(  -\sigma\right)  \right\vert ^{2}\right)  e^{-\left(  2a\left(
\theta-1\right)  +2\beta_{0}\right)  \sigma}\medskip\\
\displaystyle\leq\left(  1+\frac{1}{\delta_{0}}\right)  M_{0}e^{-\beta
_{0}\sigma}\int\nolimits_{-\tau}^{-\sigma}e^{\left(  \beta_{0}-2a\left(
1-\theta\right)  \right)  r}dr
\end{array}
\]
which yields (\ref{4.8}-b) as $\varepsilon\searrow0$ and then $\theta
\nearrow1.$\hfill
\end{proof}

\begin{theorem}
\label{t4.7} Let the assumptions (\ref{4.7}) be satisfied. Then there exists
$\eta\in L^{2}\left(  \Omega,\mathcal{F},\mathbb{P};\overline{D}\right)
$\ such that%
\[%
\begin{tabular}
[c]{l}%
$c_{1})\quad\lim\limits_{t\rightarrow\infty}u\left(  t;u_{0}\right)  =\eta
$\ in $L^{2}\left(  \Omega,\mathcal{F},\mathbb{P};H\right)  $,$\forall
\,u_{0}\in L^{2}\left(  \Omega,\mathcal{F}_{0},\mathbb{P};\overline{D}\right)
$\medskip\\
$c_{2})\quad u\left(  t;u_{0}\right)  $\ has a unique invariant measure
$\mu=L\left(  \eta\right)  $%
\end{tabular}
\
\]

\end{theorem}

\begin{proof}
Let $\left[  x_{0},y_{0}\right]  \in A.$ From Proposition \ref{p4.6} one
follows that $\exists\,\eta\in L^{2}\left(  \Omega,\mathcal{F},\mathbb{P}%
;\overline{D}\right)  $ such that $\exists\lim\limits_{t\rightarrow\infty
}u\left(  t;x_{0}\right)  =\eta$ in $L^{2}\left(  \Omega,\mathcal{F}%
,\mathbb{P};H\right)  $ and from Theorem \ref{t4.4} we have that
$\lim\limits_{t\rightarrow\infty}u\left(  t;u_{0}\right)  =\eta$ in
$L^{2}\left(  \Omega,\mathcal{F},\mathbb{P};H\right)  $ for all $u_{0}\in
L^{2}\left(  \Omega,\mathcal{F}_{0},\mathbb{P};\overline{D}\right)  .$ Let
$\mu=L\left(  \eta\right)  $ the law of $\eta$. Since the convergence in
$L^{2}$ implies the convergence in law then $\lim\limits_{t\rightarrow\infty
}\mathbb{E}g\left(  u\left(  t;u_{0}\right)  \right)  =\mathbb{E}g\left(
\eta\right)  =\int\nolimits_{\overline{D}}g\left(  x\right)  d\mu\left(
x\right)  $ for all $g\in C_{b}\left(  \overline{D}\right)  .$ Hence by
Proposition \ref{p4.5} $\mu$ is an invariant measure. The invariant measure
$\mu$ is unique since if $\mu_{1},\mu_{2}$ are two invariant measure then
$\forall t\geq0:$%
\[
\left\langle \mu_{1},g\right\rangle =\int\nolimits_{\overline{D}}%
\mathbb{E}g\left(  u\left(  t;x\right)  \right)  d\mu_{1}\left(  x\right)
,\text{ and }\left\langle \mu_{2},g\right\rangle =\int\nolimits_{\overline{D}%
}\mathbb{E}g\left(  u\left[  t;x\right)  \right)  d\mu_{2}\left(  x\right)  ,
\]
which implies, as $t\rightarrow\infty,\;\left\langle \mu_{1},g\right\rangle
=\mathbb{E}g\left(  \eta\right)  =\left\langle \mu_{2},g\right\rangle $ for
all $g\in C_{b}\left(  \overline{D}\right)  .$\hfill
\end{proof}


\begin{thebibliography}{99}                                                                                               %


\bibitem {BaOC}V. Barbu: \textit{Optimal Control of Variational Inequalities},
Pitman Advanced Publishing Program, Boston--London--Melbourne, 100,
1984.\vspace{-0.1in}

\bibitem {BaNS}V. Barbu : \textit{Nonlinear Semigroups and Differential
Equations in Banach Spaces}, Noordhoff Leyden, 1976.\vspace{-0.1in}

\bibitem {BaRaPV}V. Barbu and A. R\u{a}\c{s}canu: \textit{Parabolic
Variational Inequalities with Random Inputs}, submitted.\vspace{-0.1in}

\bibitem {BeRaPV}A. Bensoussan and A. R\u{a}\c{s}canu: \textit{Parabolic
Variational Inequalities with Random Inputs}, \textquotedblleft Les Grands
Syst\`{e}mes des Sciences et de la Technologie\textquotedblright, book in
honor of Prof. Dautray, Masson--Paris, ERMA 28, p. 77--94, 1994.\vspace
{-0.1in}

\bibitem {BRSV}A. Bensoussan and\ A. R\u{a}\c{s}canu: \textit{Stochastic
Variational Inequalities in Infinite Dimensional Spaces}, to appear in
Stoch.Anal.and Appl.\vspace{-0.1in}

\bibitem {BrOM}H. Brezis: \textit{Op\'{e}rateurs maximaux monotones et
semigroupes de contractions dans les espaces de Hilbert}, North--Holland
Publ.Co., 1973.\vspace{-0.1in}

\bibitem {CeED}E. C\'{e}pa: \textit{Equations diff\'{e}rentielles
stochastiques multivoques}, Th\`{e}se, l'Universite d'Orl\'{e}ans,1994.\vspace
{-0.1in}

\bibitem {DZSE}G. Da Prato and J. Zabczyk: \textit{Stochastic Equations in
Infinite Dimensions}, Cambridge University Press, 1992.\vspace{-0.1in}

\bibitem {HPSV}U.G. Haussman and E. Pardoux: \textit{Stochastic Variational
Inequalities of Parabolic Type}, Appl. Math. Optimiz. 20, 163-192,
1989.\vspace{-0.1in}

\bibitem {MeSV}J.L. Menaldi: \textit{Stochastic Variational Inequality for
Reflected Diffusion}, Indiana Math. J. 32,733-744, 1983.\vspace{-0.1in}

\bibitem {PaEA}E. Pardoux: \textit{Equations aux d\'{e}riv\'{e}es partielles
stochastiques nonlin\'{e}aires monotones. Etude de solutions fortes du type
It\^{o}}, Th\`{e}se, Paris-Sud, Orsay, 1975.\vspace{-0.1in}

\bibitem {LSSD}P.L. Lions and A.S. Sznitman:\ \textit{Stochastic Differential
Equations with Reflecting Boundary Conditions}, Comm. Pure Appl. Math., 37, p.
511-537, 1984.\vspace{-0.1in}

\bibitem {SaSD}Y. Saisho:\ \textit{Stochastic Differential Equations for
Multidimensional Domains with Reflecting Boundary}, Prob. Theory Rel. Fields,
74, p. 455-477, 1987.\vspace{-0.1in}

\bibitem {SkSE}A.V. Skorohod: \textit{Stochastic Equations for Diffusions in a
Bounded Region}, Theory of Prob. and its Appl.,6, 264--274, 1961.\vspace
{-0.1in}

\bibitem {SuOT}H. Sussmann: \textit{On the Gap between Deterministic and
Stochastic Ordinary Differential Equations}, Ann. Probab.6, 19--41,
1978.\vspace{-0.1in}
\end{thebibliography}
\end{document}